\documentclass[12pt]{amsart}

\usepackage[utf8]{inputenc}
\usepackage[T1,T2A]{fontenc}
\usepackage[english]{babel}
\usepackage{amsfonts}
\usepackage{amsbsy}
\usepackage{amssymb}
\usepackage{fixltx2e,graphicx}
\usepackage{mathrsfs}
\usepackage{hyperref}
\usepackage{longtable}
\usepackage{enumitem}
\usepackage{comment}

\catcode`~=11 
\newcommand{\urltilde}{\kern -.15em\lower .7ex\hbox{~}\kern .04em}
\catcode`~=13 

\renewcommand{\abovecaptionskip}{0pt}
\renewcommand{\belowcaptionskip}{6pt}
\makeatletter

\renewcommand{\@makecaption}[2]{
\vspace{\abovecaptionskip}%
\sbox{\@tempboxa}{#1. #2}%
\global\@minipagefalse \hbox to \hsize {{\scshape \hfil #1.
#2\hfil}} \vspace{\belowcaptionskip}}

\makeatother

\newcommand{\Pic}{\operatorname{Pic}}
\newcommand{\Hom}{\operatorname{Hom}}
\newcommand{\rk}{\operatorname{rk}}

\newcommand{\GL}{\operatorname{GL}}
\newcommand{\SL}{\operatorname{SL}}
\newcommand{\Sp}{\operatorname{Sp}}
\newcommand{\Spin}{\operatorname{Spin}}
\newcommand{\SO}{\operatorname{SO}}

\newcommand{\ZZ}{\mathbb Z}

\newcommand{\FF}{\mathbb F}
\newcommand{\PP}{\mathbb P}

\newtheorem{theorem}{Theorem}

\newtheorem{proposition}[theorem]{Proposition}

\newtheorem{corollary}[theorem]{Corollary}

\newtheorem*{question*}{Question}

\theoremstyle{definition}

\newtheorem{definition}[theorem]{Definition}

\theoremstyle{remark}

\newtheorem{remark}[theorem]{Remark}

\makeatletter

\@addtoreset{theorem}{section}

\makeatother

\numberwithin{equation}{section}

\makeatletter

\@addtoreset{theorem}{section}

\makeatother

\numberwithin{equation}{section}

\newcounter{num}[table]
\newcounter{subnum}[num]

\newcommand{\newcase}{\refstepcounter{num}\arabic{num}}
\newcommand{\no}{\refstepcounter{subnum}\arabic{num}.\arabic{subnum}}

\begin{document}

\renewcommand{\proofname}{Proof}
\renewcommand{\abstractname}{Abstract}
\renewcommand{\refname}{References}
\renewcommand{\figurename}{Figure}
\renewcommand{\tablename}{Table}

\title[Branching rules related to spherical actions]
{Branching rules related to spherical actions\\on flag varieties}

\author{Roman Avdeev and Alexey Petukhov}


\address{%
{\bf Roman Avdeev} \newline\indent National Research University ``Higher School of Economics'', Moscow, Russia}

\email{suselr@yandex.ru}

\address{%
{\bf Alexey Petukhov}
\newline\indent Institute for Information Transmission Problems, Moscow, Russia}

\email{alex{-}{-}2@yandex.ru}




\begin{abstract}
Let $G$ be a connected semisimple algebraic group and let $H \subset G$ be a connected reductive subgroup. Given a flag variety $X$ of~$G$, a result of Vinberg and Kimelfeld asserts that $H$ acts spherically on $X$ if and only if for every irreducible representation~$R$ of $G$ realized in the space of sections of a homogeneous line bundle on~$X$ the restriction of $R$ to $H$ is multiplicity free. In this case, the information on restrictions to $H$ of all such irreducible representations of~$G$ is encoded in a monoid, which we call the restricted branching monoid. In this paper, we review the cases of spherical actions on flag varieties of simple groups for which the restricted branching monoids are known (this includes the case where $H$ is a Levi subgroup of~$G$) and compute the restricted branching monoids for all spherical actions on flag varieties that correspond to triples $(G,H,X)$ satisfying one of the following two conditions: (1) $G$ is simple and $H$ is a symmetric subgroup of~$G$; (2) $G = \SL_n$.
\end{abstract}

\maketitle

\section{Introduction}

One of the most basic problems of representation theory of algebraic groups is that of describing the restriction of any finite-dimensional representation of a given group~$G$ to a subgroup~$H$; such a description is referred to as a \textit{branching rule} for the pair $(G,H)$. When the ground field $\FF$ is algebraically closed and of characteristic zero (which is assumed in what follows) and both groups in question are reductive, the corresponding branching rule for the pair $(G,H)$ is completely described by the collection of nonnegative integers $\dim \Hom_H(V,W)$ for all possible irreducible representations $W$ of~$G$ and $V$ of~$H$. The number $\dim \Hom_H(V,W)$ is called the \textit{multiplicity} of $V$ in~$W$.

Given a connected reductive algebraic group~$K$, fix a Borel subgroup~$B_K$, let $\Lambda^+(K)$ be the set of dominant weights of~$B_K$, and for every $\lambda \in \Lambda^+(K)$ let $R_K(\lambda)$ denote the irreducible representation of~$K$ with highest weight~$\lambda$. Now consider two connected reductive algebraic groups $G \supset H$ and for every pair $(\lambda; \mu) \in \Lambda^+(G) \times \Lambda^+(H)$ let
\[
m_\lambda(\mu) = \dim \Hom_H(R_H(\mu), \left.R_G(\lambda)\right|_H)
\]
be the corresponding multiplicity. Consider the set $\Gamma(G,H) \subset \Lambda^+(G) \times \Lambda^+(H)$ consisting of all pairs $(\lambda; \mu)$ with the property $m_\lambda(\mu) > 0$. It is well known that $\Gamma(G,H)$ is a finitely generated monoid; following Yacobi~\cite{Yac} we call it the \textit{branching monoid} (or \textit{branching semigroup}) for the pair $(G,H)$.

Of special importance in representation theory of algebraic groups are well-known branching rules for the pairs $(\SL_n, \GL_{n-1})$ and $(\Spin_n, \Spin_{n-1})$, which trace back to the work of Gelfand and Tsetlin~\cite{GT1,GT2}. A remarkable feature of these branching rules is that they are \textit{multiplicity free}, that is, each multiplicity is at most one. As a consequence, in these cases the whole branching rule is completely determined by the branching monoid. For both above-mentioned pairs, using the standard description of the branching rules in terms of interlacing conditions for dominant weights (see \cite[\S\S\,66,\,129]{Zh} or~\cite[Ch.~8]{GW} for details), it is easy to compute the branching monoid, which turns out to be free. The indecomposable elements of this monoid in both cases were explicitly written down in~\cite[Theorem~7]{AkP}. In fact, by a result of Kr\"amer~\cite{Kra} there are no other pairs $(G,H)$ with $G$ simple and simply connected for which the branching rule is multiplicity free.

When the branching rule for a pair $(G,H)$ is not multiplicity free, the multiplicities can be arbitrarily large (see~\cite[Theorem~2]{AkP}), and the problem of describing them becomes much more complicated. In the setting where $G$ is one of the classical groups $\SL_n$, $\SO_n$, or $\Sp_{2n}$, numerous papers have been devoted to developing branching rules for various connected reductive subgroups~$H$ including maximal reductive subgroups and symmetric subgroups; see an extensive list of references on this topic in~\cite{HTW2}.
In these cases, the multiplicities are usually expressed by combinatorial formulas that involve Littlewood--Richardson coefficients, Young tableaux, etc. However, using these formulas even for computing a single multiplicity may require quite nontrivial calculations.

Despite the result of Kr\"amer and the complicated situation with describing branching rules in the general case, for a given pair $(G,H)$ it may happen that multiplicities $m_\lambda(\mu)$ are still at most one whenever $\lambda$ varies over a certain infinite subset of $\Lambda^+(G)$. In this case, one may hope that the corresponding ``part'' of the branching rule for $(G,H)$ admits a simple description, and a natural problem is to find such a description.

From now on assume for simplicity that $G$ is semisimple and simply connected and let $\pi_1, \ldots, \pi_s$ be all the fundamental weights of~$G$. Given a subset $I$ of the set $S = \lbrace 1, \ldots, s\rbrace$, let $\Lambda^+_I(G) \subset \Lambda^+(G)$ be the set of all dominant weights of~$G$ that are linear combinations of fundamental weights $\pi_i$ with $i \in I$, so that $\Lambda^+_\varnothing(G) = \lbrace 0 \rbrace$ and $\Lambda^+_S(G) = \Lambda^+(G)$. It is well known that for every $I \subset S$ there exists a unique flag variety (that is, a complete homogeneous space) $X_I$ of the group~$G$ with the following property: given $\lambda \in \Lambda^+(G)$, the representation $R_G(\lambda)$ is realized as the space of sections of a homogeneous line bundle on~$X_I$ if and only if $\lambda \in \Lambda^+_I(G)$. Note that $X_S$ is nothing but the variety $G/B_G$ according to the Borel--Weil theorem. The above property of $X_I$ suggests that the behaviour of multiplicities $m_\lambda(\mu)$ with $\lambda \in \Lambda^+_I(G)$ is closely related to the geometry of the natural action of~$H$ on~$X_I$.

The starting point for our paper is the following result of Vinberg and Kimelfeld (see~\cite[Corollary~1]{VK78}): given $I \subset S$, the inequality $m_\lambda(\mu) \le 1$ holds for all $\lambda \in \Lambda^+_I(G)$ and $\mu \in \Lambda^+(H)$ if and only if the natural action of $H$ on $X_I$ is \textit{spherical}, that is, $X_I$ contains an open orbit for the induced action of~$B_H$. It is easy to see that for every spherical action of $H$ on $X_I$ the whole set of multiplicities $m_\lambda(\mu)$ with $\lambda \in \Lambda^+_I(G)$ is uniquely determined by the \textit{restricted branching monoid} $\Gamma_I(G,H) = \lbrace (\lambda;\mu) \in \Gamma(G,H) \mid \lambda \in \Lambda^+_I(G) \rbrace$. This motivates the problem of computing the monoids $\Gamma_I(G,H)$ corresponding to all spherical actions on flag varieties.

In its turn, by now the problem of classifying all spherical actions on flag varieties has been solved only under certain restrictions on the groups $G,H$ or the subset~$I$. Apart from the case $I = S$ settled by Kr\"amer (see the above discussion), below we list all cases with $G$ simple for which the classification is known:
\begin{enumerate}[label=\textup{(C\arabic*)},ref=\textup{C\arabic*}]
\item \label{case_C1}
$H$ is a Levi subgroup of~$G$ (with contributions of \cite{Lit, MWZ1, MWZ2, Stem}, see also~\cite{Pon13});

\item \label{case_C2}
$H$ is a symmetric subgroup of~$G$ (see~\cite{HNOO});

\item \label{case_C3}
$G = \SL_n$ (see~\cite{AvP});

\item \label{case_C4}
$G$ is an exceptional simple group, $H$ is a maximal reductive subgroup of~$G$, and $|I|=1$ (see the preprint~\cite{Nie}).
\end{enumerate}

A description of the monoids $\Gamma_I(G,H)$ in case~(\ref{case_C1}) follows from results of the papers~\cite{Pon15,Pon17}; we present it in Theorems~\ref{thm_Levi_sl} and~\ref{thm_Levi_non-sl} for completeness. The monoids $\Gamma_I(G,H)$ in case~(\ref{case_C4}) were computed in~\cite{Nie}. The main goal of the present paper is to compute the monoids $\Gamma_I(G,H)$ for cases~(\ref{case_C2}) and~(\ref{case_C3}), see Theorems~\ref{thm_symmetric} and~\ref{thm_sl}, respectively. We remark that in case~(\ref{case_C2}) (resp.~(\ref{case_C3})) it suffices to consider only triples $(G,H,I)$ that do not fall into case~(\ref{case_C1}) (resp. cases~(\ref{case_C1}) and~(\ref{case_C2})).

In case (\ref{case_C2}) with $G$ an exceptional simple group, the subgroup $H$ is maximal reductive in~$G$, hence for $|I|=1$ we recover (part of) results of~\cite{Nie} for case~(\ref{case_C4}). In fact, in case~(\ref{case_C2}) there are only three triples $(G,H,I)$ with $G$ exceptional and $|I| \ge 2$, see Table~\ref{table_sym}.

A significant particular case of~(\ref{case_C3}) is given by the condition $I = \lbrace 1 \rbrace$. In this situation, $X_I$ is just the projective space~$\PP((\FF^n)^*)$ where $(\FF^n)^*$ stands for the vector space dual to~$\FF^n$. Consequently, a subgroup $H \subset \SL_n$ acts spherically on~$X_I$ if and only if the group $H \times \FF^\times$ acts spherically on~$(\FF^n)^*$. (Here $\FF^\times$ is the multiplicative group of~$\FF$, which acts on $(\FF^n)^*$ by scalar transformations.) More generally, given a connected reductive algebraic group~$K$, every finite-dimensional $K$-module on which $K$ acts spherically (that is, $B_K$ has an open orbit) is called a \textit{spherical $K$-module}. An important invariant of a spherical $K$-module $V$ is its \textit{weight monoid}, consisting of all $\lambda \in \Lambda^+(K)$ for which the $K$-module $R_K(\lambda)$ occurs in the symmetric algebra of~$V^*$. There is a complete classification of all spherical modules obtained in~\cite{Kac}, \cite{BR}, and~\cite{Lea} (see \S\,\ref{subsec_spherical_modules} for more details); moreover, the weight monoids of all spherical modules are also known thanks to the works~\cite{HU} and~\cite{Lea}. Returning to the situation where a subgroup $H \subset \SL_n$ acts spherically on $X_I = \PP((\FF^n)^*)$, an easy observation shows that the monoid $\Gamma_I(G,H)$ is canonically isomorphic to the weight monoid of the spherical $(\FF^\times \times H)$-module~$(\FF^n)^*$; see~\S\,\ref{subsec_C3}. This enables us to assume $I \ne \lbrace 1 \rbrace$ (and also $I \ne \lbrace n-1 \rbrace$ by duality reasons) when computing the monoids $\Gamma_I(G,H)$ in case~(\ref{case_C3}).

We now briefly describe our method for computing the monoid $\Gamma_I(G,H)$ corresponding to a spherical action of a subgroup $H \subset G$ on a flag variety~$X_I$. A general result (see Theorem~\ref{thm_monoid_is_free}) shows that for simply connected $G$ this monoid is always free, hence it is enough to compute its rank and find the required number of its indecomposable elements. Thanks to a result of Panyushev~\cite{Pan} and the above-mentioned results on spherical modules, determining the rank of $\Gamma_I(G,H)$ reduces to computing a certain Levi subgroup~$M$ of $H$ together with a certain spherical $M$-module~$V$ (see Corollary~\ref{crl_sphericity_criterion}). In turn, computing the pair $(M,V)$ becomes effective for an appropriate choice of~$H$ within its conjugacy class (see Corollary~\ref{crl_rank_of_Gamma_refined}). Remarkably, all the pairs $(M,V)$ for case~(\ref{case_C2}) were computed in~\cite{HNOO}, therefore in this case we get all the values of rank of $\Gamma_I(G,H)$ almost for free. Once the rank of $\Gamma_I(G,H)$ has been determined, to find all indecomposable elements of $\Gamma_I(G,H)$ it suffices to explicitly compute the decompositions into irreducible summands for the restrictions to $H$ of several simple $G$-modules with ``small'' highest weights. More precisely, in all cases treated in this paper it turns out to be enough to find the decompositions into irreducible summands for the restrictions to~$H$ of all simple $G$-modules $R_G(\pi_i)$ with $i \in I$ and sometimes of one more simple $G$-module $R_G(\pi_i + \pi_j)$ with (not necessarily distinct) $i,j \in I$.
As a consequence, all indecomposable elements of $\Gamma_I(G,H)$ have the form $(\pi_i; *)$ or $(\pi_i + \pi_j; *)$ with $i,j \in I$. We note however that the latter property does not hold in general as can be seen in various examples in case (\ref{case_C3}) with $I= \lbrace 1 \rbrace$.

It is worth mentioning that the description of the monoids $\Gamma_I(G,H)$ for cases~(\ref{case_C1}) and~(\ref{case_C4}) can be also obtained by using the methods of the present paper. In case~(\ref{case_C1}) this would give a proof completely different from that in~\cite{Pon15,Pon17}. In case~(\ref{case_C4}) this would only simplify computing the rank of the monoids $\Gamma_I(G,H)$ (compared with the method of~\cite{Nie}).

At last, we would like to mention two possible directions for further research related to the present paper. Of course, the first one is to complete the classification of spherical actions on flag varieties of simple groups and determine the corresponding restricted branching monoids. The second direction is to study flag varieties $X_I$ that have complexity~$1$ under the action of a connected reductive subgroup $H \subset G$. (Here complexity means the codimension of a generic $B_H$-orbit in~$X_I$; in this terminology, spherical actions are actions of complexity~$0$.) In this setting, it would be interesting to classify all triples $(G,H,I)$ with $G$ simple such that $H$ acts on $X_I$ with complexity~$1$ and to find the corresponding branching rules for restrictions to $H$ of all representations $R_G(\lambda)$ with $\lambda \in \Lambda^+_I(G)$. (Note that these restrictions are not multiplicity free anymore.) Feasibility of these problems is suggested by the fact that they have already been solved  in the case $I = S$ (see~\cite{AkP}) and in the case where $H$ is a Levi subgroup of~$G$ (see~\cite{Pon15, Pon17}).

This paper is organized as follows. In \S\,\ref{sect_N&C}, we set up notation and conventions used throughout the paper. In \S\,\ref{sect_prelim}, we discuss several basic notions needed in this paper. In \S\,\ref{sect_RBM&SAFV}, we study general properties of restricted branching monoids corresponding to spherical actions on flag varieties. In \S\,\ref{sect_main_theorems}, we present the classification of spherical actions on flag varieties in cases~(\ref{case_C1}), (\ref{case_C2}), and~(\ref{case_C3}) together with the corresponding restricted branching monoids. At last, the restricted branching monoids in cases~(\ref{case_C2}) and~(\ref{case_C3}) are computed in~\S\,\ref{sect_proofs_C2} for case~(\ref{case_C2}) and in~\S\,\ref{sect_proofs_C3} for case~(\ref{case_C3}).

\subsection*{Acknowledgements}
The authors are grateful to Dmitry Timashev for useful discussions. The first author thanks the Institute for Fundamental Science in Moscow for providing excellent working conditions.

The results of \S\S\,\ref{subsec_spin_odd}--\ref{subsec_e7_de6xa1}, \ref{subsec_sl_spslsl}--\ref{subsec_sl_spin} are obtained by the first author supported by the grant RSF--DFG 16-41-01013. The results of \S\S\,\ref{subsec_sl_so}--\ref{subsec_sp_spsp}, \ref{subsec_sl_spsl}--\ref{subsec_sl_spsp} are obtained by the second author supported by  the RFBR grant no. 16-01-00818 and by the DFG grant PE 980/6-1.

\section{Notation and conventions}
\label{sect_N&C}

All objects considered in this paper are defined over an algebraically closed field $\FF$ of characteristic~$0$. We denote by $\FF^\times$ the multiplicative group of~$\FF$.

Throughout the paper, all topological terms refer to the Zariski topology. All subgroups of algebraic groups are assumed to be closed. The Lie algebras of algebraic groups denoted by capital Latin letters are denoted by the corresponding small Gothic letters. Given an algebraic group~$K$, a \textit{$K$-variety} is an algebraic variety equipped with a regular action of~$K$.

Notation:

$\ZZ^+ = \lbrace z \in \ZZ \mid z \ge 0 \rbrace$;

$e$ is the identity element of any group;

$|X|$ is the cardinality of a finite set~$X$;

$V^*$ is the vector space of linear functions on a vector space~$V$;

$\operatorname{S}^d V$ is the $d$th symmetric power of a vector space~$V$;

$\wedge^d V$ is the $d$th exterior power of a vector space~$V$;

$\mathfrak X(K)$ is the character group of a group~$K$ (in additive notation);

$K'$ is the derived subgroup of a group~$K$;

$C_K$ is the connected component of the identity of the center of an algebraic group~$K$;

$K^u$ is the unipotent radical of an algebraic group~$K$;

$\rk K$ is the rank of an algebraic group~$K$, that is, the dimension of a maximal torus of~$K$;

$K_x$ is the stabilizer of a point~$x$ of a $K$-variety $X$;

$\FF[X]$ is the algebra of regular functions on an algebraic variety~$X$;

$\FF(X)$ is the field of rational functions on an irreducible algebraic variety~$X$;

$T_xX$ is the tangent space of an algebraic variety $X$ at a point $x \in X$;

$V^{(K)}_\chi$ is the space of semi-invariants of weight $\chi \in \mathfrak X(K)$ for an action of a group $K$ on a vector space~$V$.

The simple roots and fundamental weights of simple algebraic groups are numbered as in~\cite{Bo}, and the same applies to nodes of connected Dynkin diagrams.

Given two algebraic groups $F \subset K$ and a $K$-module~$V$, the restriction of $V$ to~$F$ is denoted by $\left. V \right|_F$.

Given an algebraic group $K$ and subgroups $K_1, K_2 \subset K$, the notation $K = K_1 \cdot K_2$ means that $K$ is an \textit{almost direct product} of $K_1, K_2$, that is, $K = K_1K_2$, the subgroups $K_1, K_2$ commute with each other and the intersection $K_1 \cap K_2$ is finite (equivalently, there is a Lie algebra direct sum $\mathfrak k = \mathfrak k_1 \oplus \mathfrak k_2$).

For every connected reductive algebraic group~$K$, we fix a Borel subgroup $B_K$ and a maximal torus~$T_K \subset B_K$. Let $B^-_K$ be the Borel subgroup of $K$ opposite to~$B_K$ with respect to~$T_K$, that is, $B_K \cap B^-_K = T_K$. The lattices $\mathfrak X(B_K)$ and $\mathfrak X(B^-_K)$ are identified with $\mathfrak X(T_K)$ via restricting characters to the torus~$T_K$. We denote by $\Lambda^+(K)$ the set of dominant weights of~$T_K$ with respect to~$B_K$. For every $\lambda \in \Lambda^+(K)$, we denote by $R_K(\lambda)$ the simple $K$-module with highest weight~$\lambda$ and by $\lambda^*$ the highest weight of the simple $K$-module $R_K(\lambda)^*$.

Throughout the paper, $G$ denotes a simply connected semisimple algebraic group. Let $\Delta_G \subset \mathfrak X(T_G)$ be the root system of $G$ with respect to~$T_G$ and let $\Pi_G \subset \Delta_G$ be the set of simple roots relative to~$B_G$. For every $\alpha \in \Delta_G$, let $\mathfrak g_\alpha \subset \mathfrak g$ be corresponding root subspace.
Let $\pi_1, \ldots, \pi_s \in \Lambda^+(G)$ be all the fundamental weights of~$G$, so that $\Lambda^+(G) = \ZZ^+ \lbrace \pi_1, \ldots, \pi_s \rbrace$, and put $S = \lbrace 1, \ldots, s \rbrace$. For every $i \in S$, let $\alpha_i \in \Pi_G$ be the corresponding simple root.

For every subset $I \subset S$, we consider the monoid $\Lambda^+_I(G) = \ZZ^+ \lbrace \pi_i \mid i \in I \rbrace \subset \Lambda^+(G)$. We put $\lambda_I = \sum \limits_{i \in I} \pi_i$ and let $P_I$ be the stabilizer in $G$ of the line spanned by a highest weight vector (with respect to~$B_G$) in~$R_G(\lambda_I)$.
Then $P_I$ is a parabolic subgroup of~$G$ containing $B_G$ and the Lie algebra $\mathfrak p_I$ is generated by $\mathfrak b_G$ and the root subspaces $\mathfrak g_{-\alpha_i}$ with $i \notin I$.
Let also $P_I^- \supset B_G^-$ be the parabolic subgroup of~$G$ opposite to~$P_I$, that is, the Lie algebra $\mathfrak p^-_I$ is generated by $\mathfrak b^-_G$ and the root subspaces $\mathfrak g_{\alpha_i}$ with $i \notin I$.
Further, we put $L_I = P_I \cap P_I^-$; this is a Levi subgroup of both $P_I$ and $P_I^-$.
Note that the character lattices $\mathfrak X(P_I)$, $\mathfrak X(P_I^-)$, and $\mathfrak X(L_I)$ are canonically identified with $\ZZ \Lambda^+_I(G) = \ZZ \lbrace \pi_i \mid i \in I \rbrace$.
At last, we let $X_I = G/P_I^-$ be the flag variety of~$G$ corresponding to~$I$.

\section{Preliminaries}
\label{sect_prelim}

\subsection{Spherical varieties}

Given a connected reductive algebraic group~$K$, an irreducible $K$-variety $X$ is said to be \textit{spherical} (or \textit{$K$-spherical}) if $X$ contains an open orbit for the induced action of~$B_K$.

Given a $K$-spherical variety~$X$, put
\[
\Lambda_X = \lbrace \lambda \in \mathfrak X(T_K) \mid \FF(X)^{(B)}_\lambda \ne 0 \rbrace.
\]
Clearly, $\Lambda_X$ is a sublattice of~$\mathfrak X(T_K)$; it is said to be the \textit{weight lattice} of~$X$. The rank of this lattice is referred to as the \textit{rank} of the $K$-spherical variety~$X$; we denote it by~$\rk_K X$.

\subsection{Spherical modules}
\label{subsec_spherical_modules}

All modules considered in this subsection are assumed to be finite-dimensional.

Let $K$ be a connected reductive algebraic group. A \textit{spherical $K$-module} is a $K$-module $V$ that is spherical as a $K$-variety. According to~\cite[Theorem~2]{VK78}, a $K$-module $V$ is spherical if and only if the induced representation of $K$ on $\FF[V]$ is multiplicity free.
In this case, the highest weights of simple $K$-modules occurring in $\FF[V]$ form a monoid $\mathrm E(V)$, called the \textit{weight monoid} of~$V$.
It is well known that $\mathrm E(V)$ is free; see, for instance,~\cite[Theorem~3.2]{Kn}. Moreover, one has $\rk_K V = \rk \mathrm E(V)$; see, for instance,~\cite[Proposition~5.14]{Tim}.

All the terminology introduced below in this subsection follows Knop~\cite[\S\,5]{Kn}.

Given two connected reductive algebraic groups $K_1,K_2$, let $V_1$ be a $K_1$-module, let $V_2$ be a $K_2$-module, and consider the corresponding representations $\rho_1 \colon K_1 \to \GL(V_1)$ and $\rho_2 \colon K_2 \to \GL(V_2)$.
We say that the pairs $(K_1, V_1)$ and $(K_2, V_2)$ are \textit{geometrically equivalent} (or just \textit{equivalent} for short) if there exists an isomorphism $V_1 \xrightarrow{\sim} V_2$ identifying the groups $\rho_1(K_1) \subset \GL(V_1)$ and $\rho_2(K_2) \subset \GL(V_2)$. In other words, the pairs $(K_1, V_1)$ and $(K_2, V_2)$ are equivalent if and only if they define the same linear group. As an important example, every pair $(K,V)$ is equivalent to the pair $(K,V^*)$.

Given a $K$-module~$V$, consider the corresponding representation $\rho \colon K \to \GL(V)$. We say that the $K$-module $V$ is \textit{saturated} if the dimension of the center of~$\rho(K)$ equals the number of irreducible summands of~$V$. (Equivalently, the centralizer of $\rho(K)$ in $\GL(V)$ is contained in~$\rho(K)$.)

We say that a $K$-module $V$ is \textit{decomposable} if there exist connected reductive algebraic groups $K_1, K_2$, a $K_1$-module~$V_1$, and a $K_2$-module~$V_2$ such that the pair $(K,V)$ is equivalent to $(K_1 \times K_2, V_1 \oplus V_2)$. Clearly, in this situation $V$ is a spherical $K$-module if and only if $V_1$ is a spherical $K_1$-module and $V_2$ is a spherical $K_2$-module, in which case $\mathrm E(V) \simeq \mathrm E(V_1) \oplus \mathrm E(V_2)$ and $\rk_K V = \rk_{K_1} V_1 + \rk_{K_2} V_2$. We say that $V$ is \textit{indecomposable} if $V$ is not decomposable.

There exists a complete classification (up to equivalence) of all indecomposable saturated spherical modules. It was obtained in~\cite{Kac} for simple modules and independently in~\cite{BR} and~\cite{Lea} for non-simple modules. Moreover, for each of these modules the corresponding weight monoids are known thanks to the papers~\cite{HU} (the case of simple modules) and~\cite{Lea} (the case of non-simple modules). A complete list (up to equivalence) of all indecomposable saturated spherical modules can be found in~\cite[\S\,5]{Kn} along with various additional data, including the rank and indecomposable elements of the weight monoids.

For an arbitrary spherical $K$-module~$V$, fix a decomposition $V = V_1 \oplus \ldots \oplus V_n$ into a direct sum of simple $K$-submodules and let $Z$ be the subgroup of $\GL(V)$ consisting of the elements that act by scalar transformations on each~$V_i$, $i = 1,\ldots,n$. Then $V$ is a saturated spherical $(Z \times K')$-module, hence the pair $(Z \times K', V)$ is equivalent to $(K_1 \times \ldots \times K_m, W_1 \oplus \ldots \oplus W_m)$ where $K_i$ is a connected reductive algebraic group and $W_i$ is an indecomposable saturated spherical $K_i$-module for each $i = 1, \ldots, m$. In this situation, it is easy to see that $\rk_K V = \rk_{Z \times K'} V = \rk_{K_1} W_1 + \ldots + \rk_{K_m} W_m$. Note that, up to equivalence, for every $i = 1,\ldots, m$ the pair $(K_i, W_i)$ is uniquely determined by the pair $(K'_i, W_i)$, and the whole collection of pairs $(K'_i, W_i)$ is uniquely determined by the pair $(K', V)$, hence to compute $\rk_K V$ it suffices to know only the pair $(K', V)$. The latter observation together with the information from~\cite[\S\,5]{Kn} on ranks of indecomposable saturated spherical modules will be always used for computing the ranks of spherical modules in~\S\S\,\ref{sect_proofs_C2},~\ref{sect_proofs_C3}.

\subsection{Homogeneous line bundles}
\label{subsec_hlb}

Let $K$ be a subgroup of~$G$ and consider the homogeneous space $X = G/K$.

A \textit{homogeneous line bundle} on $X$ is a line bundle $L$ on~$X$ equipped with an action of $G$ such that the natural projection $L \to X$ is $G$-equivariant and the stabilizer in $G$ of any point $x \in X$ acts linearly on the fiber over~$x$. In this case, the space of global sections $H^0(X,L)$ is naturally equipped with the structure of a $G$-module. Let $\Pic_G X$ denote the group of ($G$-equivariant isomorphism classes of) homogeneous line bundles on~$X$.

Given a character $\chi \in \mathfrak X(K)$, consider the one-dimensional $K$-module $\FF^1_{\chi}$ on which $K$ acts via the character~$\chi$. Let $K$ act on $G$ by right multiplication and let $L(\chi)$ be the quotient $(G \times \FF^1_{\chi})/K$ with respect to the diagonal action of~$K$. Considered together with the natural map $L(\chi) \to X$, $L(\chi)$ becomes a homogeneous line bundle on~$X$, and there is a $G$-module isomorphism
\begin{equation} \label{eqn_iso_H^0}
H^0(X, L(\chi)) \simeq \FF[G]^{(K)}_{-\chi}.
\end{equation}

According to \cite[Theorem~4]{Pop}, the above-described map $\mathfrak X(K) \to \Pic_G X$, $\chi \mapsto L(\chi)$, is an isomorphism.

\subsection{Branching monoids and restricted branching monoids}

Let $H \subset G$ be a connected reductive subgroup. For every $\lambda \in \Lambda^+(G)$ and every $\mu \in \Lambda^+(H)$, the number
\[
m_\lambda(\mu) = \dim \Hom_H(R_H(\mu), \left. R_G(\lambda) \right|_H)
\]
is called the \textit{multiplicity} of $R_H(\mu)$ in $R_G(\lambda)$. Clearly, $m_\lambda(\mu) = \dim R_G(\lambda)^{(B_H)}_\mu$. We put
\[
\Gamma(G,H) = \lbrace (\lambda; \mu) \in \Lambda^+(G) \times \Lambda^+(H) \mid m_\lambda(\mu) > 0 \rbrace.
\]
By definition, $(\lambda; \mu) \in \Gamma(G,H)$ if and only if the $H$-module $\left. R_G(\lambda) \right|_H$ contains a submodule isomorphic to~$R_H(\mu)$.

Regard the algebra $\FF[G]$ as a ($G \times G$)-module on which the left (resp. right) factor acts by the formula $(gf)(x) = f(g^{-1}x)$ (resp. $(gf)(x) = f(xg)$), where $g,x \in G$ and $f \in \FF[G]$. Then there is the following well-known ($G \times G$)-module isomorphism (see, for instance, \cite[II.3.1, Theorem~3]{Kr} or \cite[Theorem~2.15]{Tim}):
\begin{equation} \label{eqn_GxG}
\FF[G] \simeq \bigoplus \limits_{\lambda \in \Lambda^+(G)} R_G(\lambda) \otimes R_G(\lambda)^*,
\end{equation}
where on the right-hand side the left (resp. right) factor of $G \times G$ acts on the left (resp. right) tensor factor of each summand.

In what follows, for every subgroup $K \subset G \times G$ the $K$-semi-invariants in $\FF[G]$ are taken with respect to the action of $K$ induced by the above-mentioned action of $G \times G$.

From~(\ref{eqn_GxG}) we see that $m_\lambda(\mu) = \dim \FF[G]^{(B_H \times B_G^-)}_{(\mu, -\lambda)}$. Since
\[
\FF[G]^{(B_H \times B_G^-)}_{(\mu_1, -\lambda_1)} \cdot \FF[G]^{(B_H \times B_G^-)}_{(\mu_2, -\lambda_2)} \subset \FF[G]^{(B_H \times B_G^-)}_{(\mu_1 + \mu_2, -\lambda_1 - \lambda_2)}
\]
for any two pairs $(\lambda_1; \mu_1), (\lambda_2; \mu_2) \in \Lambda^+(G) \times \Lambda^+(H)$ and the algebra $\FF[G]$ contains no zero divisors, it follows that $\Gamma(G,H)$ is a monoid. In fact, this monoid is finitely generated, see \cite[Theorem~2(ii)]{AkP} (compare also with \cite[\S\,2, Theorem]{Ela}). The terminology introduced in the definition below follows Yacobi~\cite[\S\,2.1]{Yac}.

\begin{definition}
The monoid $\Gamma(G,H)$ is called the \textit{branching monoid} for the pair $(G,H)$.
\end{definition}

Given any subset $I \subset S$, we introduce the monoid
\[
\Gamma_I(G,H) = \lbrace (\lambda; \mu) \in \Gamma(G,H) \mid \lambda \in \Lambda^+_I(G) \rbrace.
\]

\begin{definition}
The monoid $\Gamma_I(G,H)$ is called the \textit{restricted branching monoid} corresponding to the subset~$I$.
\end{definition}

Formula (\ref{eqn_GxG}) implies that $(\lambda; \mu) \in \Gamma_I(G,H)$ if and only if $\FF[G]^{(B_H \times P_I^-)}_{(\mu, -\lambda)} \ne 0$. In particular, we obtain the following fact, which can be also deduced directly from the definitions.

\begin{remark} \label{remark_00}
The element $(0;0)$ is the unique element in $\Gamma_I(G,H)$ of the form~$(0;*)$.
\end{remark}

\section{Restricted branching monoids related to spherical actions on flag varieties}
\label{sect_RBM&SAFV}

Throughout this section, $H$ is a connected reductive subgroup of~$G$ and $I \subset S$ is an arbitrary subset.

\subsection{Characterization of spherical actions on flag varieties}

The following theorem is a particular case of~\cite[Corollary~1]{VK78}.

\begin{theorem} \label{thm_criterion_spherical}
The following conditions are equivalent:
\begin{enumerate}[label=\textup{(\arabic*)},ref=\textup{\arabic*}]
\item \label{thm_criterion_spherical_1}
For every $\lambda \in \Lambda^+_I(G)$, the $H$-module $\left. R_G(\lambda) \right|_H$ is multiplicity free.

\item
The flag variety $X_I$ is $H$-spherical.
\end{enumerate}
\end{theorem}

Note that condition~(\ref{thm_criterion_spherical_1}) is equivalent to $m_\lambda(\mu) \le 1$ for all $\lambda \in \Lambda^+_I(G)$ and~$\mu \in \Lambda^+(H)$.

It is easy to see that under the conditions of Theorem~\ref{thm_criterion_spherical} the restriction to $H$ of any simple $G$-module $R_G(\lambda)$ with $\lambda \in \Lambda^+_I(G)$ is uniquely determined by the monoid $\Gamma_I(G,H)$ as follows:
\begin{equation} \label{eqn_restriction}
\left. R_G(\lambda) \right|_H \simeq \bigoplus \limits_{\mu \in \Lambda^+(H) \, : \, (\lambda; \mu) \in \Gamma_I(G,H)} R_H(\mu).
\end{equation}

Given a character $\lambda \in \mathfrak X(P_I^-)$, let $L(\lambda)$ be the corresponding homogeneous line bundle on~$X_I$ (see~\S\,\ref{subsec_hlb}). Comparing formulas~(\ref{eqn_iso_H^0}) and~(\ref{eqn_GxG}) we find that there is a $G$-module isomorphism
\begin{equation} \label{eqn_H^0_for_XI}
H^0(X_I, L(\lambda)) \simeq
\begin{cases} R_G(\lambda) & \text{if} \ \lambda \in \Lambda^+_I(G);\\
0 & \text{otherwise}.
\end{cases}
\end{equation}

\subsection{Freeness of restricted branching monoids}

Let $\mathcal D$ denote the set of $B_H$-stable prime divisors on~$X_I$.
Note that $\mathcal D$ is finite whenever $X_I$ is $H$-spherical.

\begin{theorem} \label{thm_monoid_is_free}
Under the conditions of Theorem~\textup{\ref{thm_criterion_spherical}}, the monoid $\Gamma_I(G,H)$ is free and its rank equals the cardinality of~$\mathcal D$.
\end{theorem}

\begin{proof}
Let $\ZZ^+ \mathcal D$ be the monoid of nonnegative integer linear combinations of elements in~$\mathcal D$ and consider the map $\varphi \colon \Gamma_I(G,H) \to \ZZ^+ \mathcal D$ sending a pair $(\lambda; \mu)$ to the divisor of zeros of a (unique up to proportionality) $B_H$-semi-invariant section in $H^0(X_I, L(\lambda)) \simeq R_G(\lambda)$ of weight~$\mu$. It is easy to see that $\varphi$ is an injective monoid homomorphism. Now consider an arbitrary divisor $D \in \ZZ^+ \mathcal D$. Being an effective Cartier divisor, $D$ determines a line bundle $L$ on $X_I$ together with a (unique up to proportionality) section $s_D \in H^0(X_I, L)$ such that $D$ is the divisor of zeros of~$s_D$. As $G$ is simply connected, $L$ admits a unique structure of a homogeneous line bundle (see~\cite[Proposition~1 and Theorem~4]{Pop}), so that $L \simeq L(\lambda)$ for some $\lambda \in \mathfrak X(P_I^-)$. By~(\ref{eqn_H^0_for_XI}) the condition $H^0(X,L(\lambda)) \ne 0$ implies $\lambda \in \Lambda^+_I(G)$. Since the divisor $D$ is $B_H$-stable, it follows that the section $s_D$ is $B_H$-semi-invariant of some weight~$\mu \in \Lambda^+(H)$. Consequently, $D$ is the image of the pair $(\lambda; \mu)$, and we have proved the surjectivity of~$\varphi$. Thus $\varphi$ is an isomorphism.
\end{proof}

\begin{remark}
Theorem~\ref{thm_monoid_is_free} is a particular case of a general result on Cox rings of spherical varieties, see \cite[Proposition~4.2.3]{Bri} or \cite[Theorem~5.4.6(i)]{ADHL}.
\end{remark}

\subsection{Formulas for the rank of restricted branching monoids}

\begin{proposition} \label{prop_rank_of_Gamma}
Under the conditions of Theorem~\textup{\ref{thm_criterion_spherical}}, $\rk \Gamma_I(G,H) = |I| + \rk_H X_I$.
\end{proposition}

\begin{proof}
Let $\Lambda$ be the weight lattice of~$X_I$ as an $H$-spherical variety and let $\ZZ \mathcal D$ denote the free Abelian group generated by~$\mathcal D$. There is an injective map $\Lambda \to \ZZ \mathcal D$ sending an element $\lambda$ to the divisor of a (unique up to proportionality) $B_H$-semi-invariant rational function on $X_I$ of weight~$\lambda$. Further, since $\Pic X_I$ is freely generated by the images of all $B_G$-stable prime divisors, there is a surjective map $\ZZ \mathcal D \to \Pic X_I$. Clearly, the composite map $\Lambda \to \ZZ \mathcal D \to \Pic X_I$ is zero and every element of $\ZZ \mathcal D$ with zero image in $\Pic X_I$ is the divisor of zeros of a $B_H$-semi-invariant rational function on~$X_I$, which yields an exact sequence
\[
0 \to \Lambda \to \ZZ \mathcal D \to \Pic X_I \to 0.
\]
Consequently, $|\mathcal D| = \rk \Pic X_I + \rk_H X_I$. As $\rk \Pic X_I = |I|$, the claim is implied by Theorem~\ref{thm_monoid_is_free}.
\end{proof}

In order to compute the rank of an $H$-spherical variety~$X_I$, we shall need the result of Panyushev stated in Theorem~\ref{thm_Panyushev} below.

Suppose that a connected reductive algebraic group $K$ acts on a smooth irreducible variety~$X$ and $Y \subset X$ is a smooth $K$-stable locally closed subvariety. Then one can consider the normal bundle $N_{X/Y}$ and the conormal bundle $N^\vee_{X/Y}$ of $Y$ in~$X$, which are $K$-varieties in a natural way; see~\cite[\S\,2]{Pan} for details. The next theorem is a particular case of~\cite[Corollary~2.4]{Pan}.

\begin{theorem} \label{thm_Panyushev}
The following conditions are equivalent:
\begin{enumerate}[label=\textup{(\arabic*)},ref=\textup{\arabic*}]
\item
$X$ is $K$-spherical.

\item
$N_{X/Y}$ is $K$-spherical.

\item
$N^\vee_{X/Y}$ is $K$-spherical.
\end{enumerate}
Moreover, under the above three conditions one has $\rk_K X = \rk_K N_{X/Y} = \rk_K N^\vee_{X/Y}$.
\end{theorem}

Here is a useful consequence of the above theorem.

\begin{proposition} \label{prop_sphericity_criterion}
Suppose that $X$ is complete, $Y \subset X$ is a closed $K$-orbit, $y \in Y$, and $M$ is a Levi subgroup of~$K_y$. Then the following conditions are equivalent:
\begin{enumerate}[label=\textup{(\arabic*)},ref=\textup{\arabic*}]
\item \label{X_spherical_1}
$X$ is a $K$-spherical variety.

\item \label{X_spherical_2}
$T_y X / T_y Y$ is a spherical $M$-module.
\end{enumerate}
Moreover, under the above two conditions one has $\rk_K X = \rk_M (T_y X / T_y Y)$.
\end{proposition}

\begin{proof}
By Theorem~\ref{thm_Panyushev}, condition~(\ref{X_spherical_1}) holds if and only if the normal bundle $N_{X/Y}$ is $K$-spherical.
Let $\varphi \colon N_{X/Y} \to Y$ be the natural $K$-equivariant projection and consider the $M$-module~$V = \varphi^{-1}(y) = T_y X / T_y Y$.
As $Y$ is a closed $K$-orbit in~$X$, it follows that $K_y$ is a parabolic subgroup of~$K$.
Without loss of generality we may assume that the Borel subgroup $B_K$ is chosen in such a way that $K_y \cap B_K = M \cap B_K$ is a Borel subgroup of~$M$.
In this case, $B_K$ has an open orbit~$O$ in~$Y$.
Restricting $\varphi$ to the open $B_K$-stable subset $\varphi^{-1}(O)$ we easily see that the existence of an open $B_K$-orbit in $N_{X/Y}$ is equivalent to the existence of an open $(M \cap B_K)$-orbit in~$V$, hence $N_{X/Y}$ being $K$-spherical is equivalent to~(\ref{X_spherical_2}).

Now suppose that conditions~(\ref{X_spherical_1}) and~(\ref{X_spherical_2}) hold. Then $\rk_K X = \rk_K N_{X/Y}$ by Theorem~\ref{thm_Panyushev}. As the group $B_K^u$ acts transitively on~$O$, $B_K$-semi-invariant rational functions on~$N_{X/Y}$ restrict bijectively to $(M \cap B_K)$-semi-invariant rational functions on~$V$. The latter yields $\rk_K N_{X/Y} = \rk_M V$ as required.
\end{proof}

\begin{corollary}[{compare with~\cite[Theorem~4.2]{HNOO}}] \label{crl_sphericity_criterion}
Suppose that $X$ is a flag variety for~$G$, $Y \subset X$ is a closed $H$-orbit, and $y \in Y$. Put $P = G_y$ and  let $M$ be a Levi subgroup of~$H_y$. Then the following conditions are equivalent:
\begin{enumerate}[label=\textup{(\arabic*)},ref=\textup{\arabic*}]
\item
$X$ is an $H$-spherical variety.

\item
$\mathfrak g / (\mathfrak p + \mathfrak h)$ is an $M$-spherical module.
\end{enumerate}
Moreover, under the above two conditions one has $\rk_H X = \rk_M (\mathfrak g / (\mathfrak p + \mathfrak h))$.
\end{corollary}

\begin{proof}
This follows directly from Proposition~\ref{prop_sphericity_criterion} in view of the $M$-module isomorphism $T_y X / T_y Y \simeq (\mathfrak g/ \mathfrak p) / (\mathfrak h / (\mathfrak h \cap \mathfrak p)) \simeq \mathfrak g / (\mathfrak p + \mathfrak h)$.
\end{proof}

It is well known that, under an appropriate choice of $H$ within its conjugacy class in~$G$, one can achieve the inclusion $B_H^- \subset B_G^-$. In this situation, we obtain the following refinement of Proposition~\ref{prop_rank_of_Gamma}, which will be extensively used in our paper.

\begin{corollary} \label{crl_rank_of_Gamma_refined}
Under the conditions of Theorem~\textup{\ref{thm_criterion_spherical}}, suppose in addition that $B_H^- \subset B_G^-$ and let $M$ be a Levi subgroup of~$P_I^- \cap H$. Then $\mathfrak g / (\mathfrak p_I^- + \mathfrak h)$ is a spherical $M$-module and
\[
\rk \Gamma_I(G,H) = |I| + \rk_M (\mathfrak g / (\mathfrak p_I^- + \mathfrak h)).
\]
\end{corollary}

\begin{proof}
Since $B_H^- \subset B_G^-$, it follows that $Q = P_I^- \cap H$ is a parabolic subgroup of~$H$. Clearly, $Q$ is the stabilizer in~$H$ of the point $y = eP_I^- \in X_I$, hence $Y = Hy$ is a closed $H$-orbit in~$X_I$. Now the required result follows from Corollary~\ref{crl_sphericity_criterion} and Proposition~\ref{prop_rank_of_Gamma}.
\end{proof}

\subsection{Indecomposable elements of restricted branching monoids}

The statements in this subsection turn out to be enough to determine all indecomposable elements of the monoids $\Gamma_I(G,H)$ in \S\S\,\ref{sect_proofs_C2},\,\ref{sect_proofs_C3}.

\begin{proposition} \label{prop_indec_I}
For every $i \in I$ and every $\mu \in \Lambda^+(H)$ with $m_{\pi_i}(\mu) > 0$, the element $(\pi_i; \mu)$ is indecomposable in~$\Gamma_I(G,H)$.
\end{proposition}

\begin{proof}
This follows from Remark~\ref{remark_00}.
\end{proof}

For every $i \in I$ we set $M(i) = \lbrace \mu \in \Lambda^+(H) \mid m_{\pi_i}(\mu) > 0 \rbrace$. For every $i,j \in I$ (not necessarily distinct) we define the $H$-module $W_{i,j}(G,H) = \bigoplus \limits_{\mu \in M(i) + M(j)} R_H(\mu)$. Since $\Gamma_I(G,H)$ is a monoid, it follows from the definition that $W_{i,j}(G,H)$ is a submodule of~$\left. R_G(\pi_i + \pi_j) \right|_H$.

\begin{proposition} \label{prop_indec_II}
For every $i,j \in I$ and every $\mu \in \Lambda^+(H)$ such that $m_{\pi_i+\pi_j}(\mu) > 0$ and $\mu \notin M(i) + M(j)$, the element $(\pi_i + \pi_j; \mu)$ is indecomposable in~$\Gamma_I(G,H)$.
\end{proposition}

\begin{proof}
This follows from the definition of the set $M(i)+M(j)$ and again from Remark~\ref{remark_00}.
\end{proof}

\section{The restricted branching monoids in cases (\ref{case_C1}),~(\ref{case_C2}), and~(\ref{case_C3})}
\label{sect_main_theorems}

In this section, we present the rank and indecomposable elements of all restricted branching monoids in cases~(\ref{case_C1}),~(\ref{case_C2}), and~(\ref{case_C3}).

\subsection{Reductions}
The goal of this subsection is to describe several reductions that simplify the statement of all the theorems in this section. Fix a connected reductive subgroup $H \subset G$ and a subset $I \subset S$ and suppose that the variety $X_I$ is $H$-spherical.

\textit{Reduction}~1. Choose an arbitrary subset~$I' \subset I$. Then the variety $X_{I'}$ is automatically $H$-spherical and the indecomposable elements of $\Gamma_{I'}(G,H)$ are those of $\Gamma_I(G,H)$ for which the first component belongs to $\Lambda^+_{I'}(G)$. Therefore, for a given pair $(G,H)$, it is enough to consider subsets $I \subset S$ that are maximal with the property that $X_I$ is $H$-spherical.

\textit{Reduction}~2. For every $i \in S$ define $i^* \in S$ in such a way that $\pi_{i^*} = \pi_i^*$ and put $I^* = \lbrace i^* \mid i \in I \rbrace$.
Given $\lambda \in \Lambda^+_I(G)$, taking the duals of the both sides of~(\ref{eqn_restriction}) we get
\[
\left. R_G(\lambda^*) \right|_H \simeq \bigoplus \limits_{\mu \in \Lambda^+(H) \, : \, (\lambda,\mu) \in \Gamma_I(G,H)} R_H(\mu^*).
\]
Then Theorem~\ref{thm_criterion_spherical} implies that $X_{I^*}$ is $H$-spherical as well. Moreover, $(\lambda; \mu) \in \nobreak \Gamma_I(G,H)$ if and only if $(\lambda^*; \mu^*) \in \Gamma_{I^*}(G,H)$. In particular, $(\lambda; \mu)$ is an indecomposable element of $\Gamma_I(G,H)$ if and only if $(\lambda^*; \mu^*)$ is an indecomposable element of $\Gamma_{I^*}(G,H)$. In this situation, we say that the triples $(G,H,I)$ and $(G,H,I^*)$ are related by duality.

\textit{Reduction}~3. Given an outer automorphism $\sigma$ of~$G$, denote by the same letter $\sigma$ the following objects:
\begin{itemize}
\item
the bijection of $S$ induced by the corresponding automorphism of the Dynkin diagram of~$G$;

\item
the induced bijection $\Lambda^+(G) \to \Lambda^+(G)$;

\item
the induced bijection $\Lambda^+(H) \to \Lambda^+(\sigma(H))$.
\end{itemize}
Now, given $\lambda \in \Lambda^+_I(G)$, after changing the action of $G$ on $R_G(\lambda)$ to $(g,v) \mapsto \sigma^{-1}(g)(v)$ formula~(\ref{eqn_restriction}) takes the form
\[
\left. R_G(\sigma(\lambda)) \right|_{\sigma(H)} \simeq \bigoplus \limits_{\mu \in \Lambda^+(H) \, : \, (\lambda; \mu) \in \Gamma_I(G,H)} R_{\sigma(H)}(\sigma(\mu)).
\]
Then Theorem~\ref{thm_criterion_spherical} implies that $X_{\sigma(I)}$ is $\sigma(H)$-spherical. Moreover, $(\lambda; \mu) \in \nobreak \Gamma_I(G,H)$ if and only if $(\sigma(\lambda); \sigma(\mu)) \in \Gamma_{\sigma(I)}(G,\sigma(H))$. In particular, $(\lambda; \mu)$ is an indecomposable element of $\Gamma_I(G,H)$ if and only if $(\sigma(\lambda); \sigma(\mu))$ is an indecomposable element of $\Gamma_{\sigma(I)}(G,\sigma(H))$. In this situation, we say that the triple $(G,\sigma(H),\sigma(I))$ is obtained from $(G,H,I)$ by the automorphism~$\sigma$. Note that if $H = L_J$ for a subset $J \subset S$ then $\sigma(H)$ is conjugate to $L_{\sigma(J)}$.

\subsection{The Levi subgroup case}
\label{subsec_C1}

In this subsection, we present a description of restricted branching monoids for all spherical actions on flag varieties in case~(\ref{case_C1}). This description follows from results of Ponomareva obtained in~\cite{Pon15, Pon17}. More precisely, \cite[Theorem~2, part b) and Theorem~1]{Pon17} reduce the description of $\Gamma_I(G,H)$ in each case to computing the algebras of unipotent invariants of Cox rings of corresponding double flag varieties; in turn, these algebras are computed in \cite[\S\S\,7--9]{Pon15} for the classical groups~$G$ and in~\cite[\S\S\,6,7]{Pon17} for the exceptional groups~$G$.

For the sake of convenience, we separate the cases $G = \SL_n$ and $G \not\simeq \SL_n$.

First, we consider the case $G = \SL_n$. It is well known that the (conjugacy classes of) Levi subgroups of $\SL_n$ are in bijection with partitions of~$n$, that is, tuples of positive integers $(a_1, \ldots, a_t)$ such that $a_1 \ge \ldots \ge a_t$ and $a_1 + \ldots + a_t = n$. Given such a partition $\mathbf a = (a_1, \ldots, a_t)$, the corresponding Levi subgroup $L_{\mathbf{a}}$ of $\SL_n$ is determined as follows. Fix a decomposition $\FF^n = V_1 \oplus \ldots \oplus V_t$ into a direct sum of subspaces $V_1, \ldots, V_t$ such that $\dim V_i = a_i$ for all $i = 1, \ldots, t$. Then $L_{\mathbf a}$ consists of all elements of $\SL_n$ stabilizing each of the subspaces $V_1, \ldots, V_t$. Clearly, $L'_{\mathbf a} \simeq \SL_{a_1} \times \ldots \times \SL_{a_t}$.

\begin{theorem} \label{thm_Levi_sl}
Suppose that $G = \SL_n$, $\mathbf a$ is a partition of~$n$, and $I \subset S$ is a nonempty subset. Put $H = L_{\mathbf a}$. Then the following conditions are equivalent:
\begin{enumerate}[label=\textup{(\arabic*)},ref=\textup{\arabic*}]
\item
The variety $X_I$ is $H$-spherical and $I$ is maximal with this property.

\item
Up to duality, the triple $(G, \mathbf a, I)$ is contained in Table~\textup{\ref{table_Levi_sl}} \textup(see \textup{\S\,\ref{subsec_the_tables})}.
\end{enumerate}
Moreover, Table~\textup{\ref{table_Levi_sl}} lists also the rank and indecomposable elements of the monoid $\Gamma_I(G,H)$ for each of the triples $(G, \mathbf a, I)$.
\end{theorem}

\begin{remark}
Under the assumptions of Theorem~\ref{thm_Levi_sl}, for each triple $(G,\mathbf a,I)$ in Table~\ref{table_Levi_sl}, all subgroups $K \subset G$ such that $H' \subset K \subset H$ and $K$ acts spherically on $X_I$ are described in~\cite[Theorem~1.7]{AvP}. By Theorem~\ref{thm_criterion_spherical}, such subgroups $K$ are characterized by the property that the restrictions to $\Lambda_I^+(G) \oplus \Lambda^+(K)$ of all the indecomposable elements of $\Gamma_I(G,H)$ are linearly independent.
\end{remark}

In \S\S\,\ref{sect_proofs_C2},\,\ref{sect_proofs_C3} we shall need the following consequence of Theorem~\ref{thm_Levi_sl}, obtained by Reduction~1 from Case~\ref{sl_Levi_pq_1i_part1} in Table~\ref{table_Levi_sl}. (The notation is the same as in Table~\ref{table_Levi_sl}, see \S\,\ref{subsec_the_tables}.)

\begin{proposition} \label{prop_sl_Levi_pq_i}
Suppose that $G = \SL_n$, $H = L_{\mathbf a}$ with~$\mathbf a = (p,q)$, and $I = \lbrace i \rbrace$ with $1 \le i \le\nobreak p$. Then the variety $X_I$ is $H$-spherical, $\rk \Gamma_I(G/H) = \min(i,q) + 1$, and the indecomposable elements of $\Gamma_I(G/H)$ are \mbox{$(\pi_i; \pi_{i-k} + \pi'_k + (i-k)\chi_1 + k\chi_2)$} for $0 \le k \le\nobreak \min(i,q)$.
\end{proposition}

We now turn to the case $G \not\simeq \SL_n$.

\begin{theorem} \label{thm_Levi_non-sl}
Suppose that $G \not\simeq \SL_n$ and $I,J \subset S$ are nonempty subsets. Put $H = L_J$. Then the following conditions are equivalent:
\begin{enumerate}[label=\textup{(\arabic*)},ref=\textup{\arabic*}]
\item
The variety $X_I$ is $H$-spherical and $I$ is maximal with this property.

\item
Up to automorphism of~$G$, the triple $(G,J,I)$ is contained in Table~\textup{\ref{table_Levi_non-sl}} \textup(see \textup{\S\,\ref{subsec_the_tables})}.
\end{enumerate}
Moreover, Table~\textup{\ref{table_Levi_non-sl}} lists also the rank and indecomposable elements of the monoid $\Gamma_I(G,H)$ for each of the triples $(G,J,I)$.
\end{theorem}

\begin{remark}
Under the assumptions of Theorems~\ref{thm_Levi_sl} and~\ref{thm_Levi_non-sl}, the description of $\Gamma_I(G, H)$ enables one to decompose any representation $R_G(\lambda)$ with $\lambda \in \Lambda^+_I(G)$ to any subgroup between $H$ and~$H'$ by restricting characters of~$C_H$.
\end{remark}

\subsection{The symmetric subgroup case}
\label{subsec_C2}

A subgroup $K$ of $G$ is said to be \textit{symmetric} if $K$ is the subgroup of fixed points of a nontrivial involutive automorphism $\theta$ of~$G$. As $G$ is simply connected, in this case $K$ is reductive and connected by~\cite[Theorem~8.1]{Stei}.

\begin{theorem} \label{thm_symmetric}
Suppose that $I \subset S$ is a nonempty subset and
\begin{itemize}
\item
$H$ is a symmetric subgroup of $G$;

\item
$H$ is not a Levi subgroup of $G$.
\end{itemize}
Then the following conditions are equivalent:
\begin{enumerate}[label=\textup{(\arabic*)},ref=\textup{\arabic*}]
\item \label{sym1}
The variety $X_I$ is $H$-spherical and $I$ is maximal with this property.

\item \label{sym2}
Up to duality and up to an automorphism of~$G$, the triple $(G,H,I)$ is contained in Table~\textup{\ref{table_sym}} \textup(see \textup{\S\,\ref{subsec_the_tables})}.
\end{enumerate}
Moreover, Table~\textup{\ref{table_sym}} lists also the rank and indecomposable elements of the monoid $\Gamma_I(G,H)$ for each of the triples $(G,H,I)$.
\end{theorem}

\begin{remark}
The equivalence of conditions~(\ref{sym1}) and~(\ref{sym2}) in Theorem~\ref{thm_symmetric} is implied by \cite[Theorems~5.2 and~4.2]{HNOO}.
The computations of the monoid $\Gamma_I(G,H)$ for each case in Table~\ref{table_sym} are carried out in~\S\,\ref{sect_proofs_C2}.
\end{remark}

\subsection{The case~\texorpdfstring{$G = \SL_n$}{G=SL\_n}}
\label{subsec_C3}

First of all, we discuss separately the case $I = \lbrace 1 \rbrace$. As was mentioned in the introduction, in this situation $X_I$ is the projective space~$\PP((\FF^n)^*)$. Consequently, given a connected reductive subgroup~$H \subset \SL_n$, the variety $X_I$ is spherical if and only if $(\FF^n)^*$ is a spherical $(H \times \FF^\times)$-module, where $\FF^\times$ acts by scalar transformations. Let $\varepsilon$ be the character via which $\FF^\times$ acts on~$\FF^n$. Then $\Lambda^+(H \times \FF^\times)$ is naturally identified with $\Lambda^+(H) \oplus \ZZ \varepsilon$. As $\FF[(\FF^n)^*]$ is the symmetric algebra of~$\FF^n$ and $\operatorname{S}^d (\FF^n) \simeq R_G(d\pi_1)$ for all $d \in \ZZ^+$, a comparison of the definitions of $\Gamma_I(G,H)$ and $\mathrm E((\FF^n)^*)$ yields the following result: given $d \in \ZZ^+$ and $\mu \in \Lambda^+(H)$, one has $(d\pi_1; \mu) \in \Gamma_I(G,H)$ if and only if $\mu + d\varepsilon \in \mathrm E((\FF^n)^*)$. The latter yields a canonical isomorphism $\Gamma_I(G,H) \simeq \mathrm E((\FF^n)^*)$. Taking into account the fact that the weight monoids are known for all spherical modules (see~\S\,\ref{subsec_spherical_modules}), in Theorem~\ref{thm_sl} below we assume $I \ne \lbrace 1 \rbrace$ and by duality $I \ne \lbrace n-1 \rbrace$.

\begin{remark}
In fact, if a connected reductive subgroup $H \subset \SL_n$ acts spherically on $X_I$ for some $I \ne \varnothing$ then $H$ automatically acts spherically on $\PP(\FF^n) \simeq X_{\lbrace n-1 \rbrace}$ (and hence on $\PP((\FF^n)^*) \simeq X_{\lbrace 1 \rbrace}$ by duality); see~\cite[Theorem~5.8]{Pet} or~\cite[Proposition~3.7]{AvP}.
\end{remark}

Given a connected reductive subgroup $H \subset G$, fix a decomposition $\FF^n = V_1 \oplus \ldots \oplus V_t$ into a direct sum of simple $H$-modules. Let $Z \subset G$ be the subgroup of elements that act by scalar transformations on each~$V_i$, $i = 1, \ldots, t$. Clearly, $H' = (Z H)'$ and $C_H \subset Z$.

If $X_I$ is an $H$-spherical variety for some $I \subset S$ then $X_I$ is $(Z  H)$-spherical as well. In this situation, restricting characters of $Z$ to~$C_H$ yields a natural isomorphism $\Gamma_I(G,Z H) \simeq \Gamma_I(G,H)$. Hence in the theorem below it is enough to restrict ourselves to the subgroups $H$ satisfying $C_H = Z$.

\begin{theorem} \label{thm_sl}
Suppose that $G = \SL_n$ with $n \ge 2$, $I \subset S$ is a nonempty subset distinct from $\lbrace 1 \rbrace$ and $\lbrace n - 1 \rbrace$, and $H \subset G$ is a connected reductive subgroup such that
\begin{itemize}
\item
$C_H = Z$;

\item
$H$ is not a Levi subgroup of~$G$;

\item
$H$ is not a symmetric subgroup of~$G$.
\end{itemize}
Then the following conditions are equivalent:
\begin{enumerate}[label=\textup{(\arabic*)},ref=\textup{\arabic*}]
\item \label{sl1}
The variety $X_I$ is $H$-spherical and $I$ is maximal with this property.

\item \label{sl2}
Up to duality, the triple $(G,H',I)$ is contained in Table~\textup{\ref{table_sl}} \textup(see \textup{\S\,\ref{subsec_the_tables})}.
\end{enumerate}
Moreover, Table~\textup{\ref{table_sl}} lists also the rank and indecomposable elements of the monoid $\Gamma_I(G,H)$ for each of the triples $(G,H',I)$.
\end{theorem}

\begin{remark}
The equivalence of conditions~(\ref{sl1}) and~(\ref{sl2}) in Theorem~\ref{thm_sl} is implied by~\cite[Theorem~1.7]{AvP}.
The computations of the monoid $\Gamma_I(G,H)$ for each case in Table~\ref{table_sl} are carried out in~\S\,\ref{sect_proofs_C3}.
\end{remark}

\begin{remark}
Under the assumptions of Theorem~\ref{thm_sl}, for each triple $(G,H',I)$ in Table~\ref{table_sl}, all subgroups $K \subset G$ such that $H' \subset K \subset H$ and $K$ acts spherically on $X_I$ are also described in~\cite[Theorem~1.7]{AvP}. By Theorem~\ref{thm_criterion_spherical}, such subgroups $K$ are characterized by the property that the restrictions to $\Lambda_I^+(G) \oplus \Lambda^+(K)$ of all the indecomposable elements of $\Gamma_I(G,H)$ are linearly independent.
\end{remark}

\subsection{The tables}
\label{subsec_the_tables}

Before presenting our tables, we explain some notation and introduce several conventions used in the tables.

The symbol $\delta_i^j$ denotes the Kronecker delta, that is, $\delta_i^j = 1$ for $i = j$ and $\delta_i^j = 0$ otherwise.

Whenever an element $(\lambda; \mu)$ in the last column is followed by a parenthesis containing a certain condition (equality or inequality) on parameters, this means that $(\lambda; \mu)$ is an indecomposable element of $\Gamma_I(G,H)$ if and only if the condition is satisfied.

In Tables~\ref{table_Levi_sl}, \ref{table_sym}, and~\ref{table_sl}, if the group $H'$ is a product of one, two, or three factors each being either simple or~$\Spin_4$, then $\pi_i$ (resp. $\pi'_i$, $\pi''_i$) stands for the $i$th fundamental weight of the first (resp. second, third) factor of~$H'$ (for $\Spin_4$, the fundamental weights have numbers~$1$ and~$2$). For convenience in certain formulas, we put $\pi_0 = \pi'_0 = \pi''_0 = 0$. Likewise, if the first (resp. second, third) factor of~$H'$ is $\SL_m$ for some~$m$, we put by convention $\pi_m = 0$ (resp. $\pi'_m = 0$, $\pi''_m = 0$).

In each case of Table~\ref{table_Levi_non-sl}, we have $H = L_J$ for a subset $J \subset S$. In this situation, we choose $B_H = B_G \cap H$, $T_H = T_G$ and identify $\Lambda^+(H)$ with a submonoid of $\mathfrak X(T_G) = \ZZ \lbrace \pi_1, \ldots, \pi_s \rbrace$. Note that this identification implies $-\pi_j \in \Lambda^+(H)$ for all $j \in J$. For convenience in certain formulas, we put $\pi_0 = 0$.

In Cases~\ref{spin_odd} and~\ref{spin_even} of Table~\ref{table_sym}, the subgroup $H = \Spin_p \cdot \Spin_q$ is the preimage of the subgroup $\SO_p \times \SO_q \subset \SO_n$ under the covering homomorphism $\Spin_n \to \SO_n$. By convention, $\Spin_1 = \lbrace e \rbrace$.

In Cases~\ref{f4_b4}--\ref{e7_d6xa1} of Table~\ref{table_sym}, $H$ is a connected semisimple subgroup of~$G$ of the indicated type.

The following convention applies to Tables~\ref{table_Levi_sl},  \ref{table_sym}, and~\ref{table_sl}. When the Dynkin diagram of a factor (either simple or $\Spin_4$) of~$H'$ has nontrivial symmetries, the numbering of simple roots of this factor is not uniquely determined and hence can influence the explicit expressions for (some of) indecomposable elements of~$\Gamma_I(G,H)$. This is fixed by making the following additional conventions:
\begin{itemize}
\item
in Cases~\ref{sl_Levi_p1}--\ref{sl_Levi_pqr} of Table~\ref{table_Levi_sl} and Cases~\ref{sl_spsl}, \ref{sl_spslsl},~\ref{sl_spspsl} of Table~\ref{table_sl}: for each simple factor~$F$ of~$H'$, the highest weight of the unique nontrivial simple $F$-module in $\left. R_G(\pi_1) \right|_F$ is the first fundamental weight of~$F$;

\item
in Cases~\ref{spin_odd}, \ref{spin_even} of Table~\ref{table_sym} with $p = 4$: the first fundamental weight of the factor $\Spin_p$ of~$H$ is chosen in such a way that $\left. R_G(\pi_1) \right|_{\Spin_p}$ contains a submodule isomorphic to $R_{\Spin_p}(\pi_1)$; a similar convention applies to Case~\ref{spin_even} of Table~\ref{table_sym} with $q=4$;

\item
in Case~\ref{spin_even_n_pq_even} of Table~\ref{table_sym}: the last two fundamental weights of the factor $\Spin_q$ may be chosen arbitrarily and the $(p/2)$th fundamental weight of the factor $\Spin_p$ is chosen in such a way that $\left. R_G(\pi_n) \right|_H$ contains a submodule isomorphic to $R_H(\pi_{\frac p2} + \pi'_{\frac q2})$;

\item
in Case~\ref{e6_a5xa1} of Table~\ref{table_sym}: the first fundamental weight of the factor $\mathsf A_5$ of~$H$ is chosen in such a way that $\left. R_G(\pi_1) \right|_H$ contains a submodule isomorphic to $R_H(\pi_1 + \pi'_1)$;

\item
in Case~\ref{e7_d6xa1} of Table~\ref{table_sym}: the $6$th fundamental weight of the factor $\mathsf D_6$ is chosen in such a way that $\left. R_G(\pi_7) \right|_H$ contains a submodule isomorphic to~$R_H(\pi_6)$.
\end{itemize}

Recall from \S\S\,\ref{subsec_C1},~\ref{subsec_C3} that each case in Tables~\ref{table_Levi_sl} and~\ref{table_sl} comes together with a decomposition $\FF^n = V_1 \oplus \ldots \oplus V_t$ into a direct sum of simple $H$-modules. We assume in all cases that the $i$th factor of $H'$ acts irreducibly on $V_i$ and trivially on $V_j$ with $j \ne i$. Moreover, if the $i$th factor of $H'$ is of type $\SL$ or $\Sp$ then $V_i$ is the space of the tautological representation of this factor. In Case~\ref{sl_spin} of Table~\ref{table_sl}, the group $H = \Spin_7$ acts on~$\FF^8$ via its spinor representation. In all the cases, for each $i = 1, \ldots, t$ we denote by $\chi_i$ the character of~$C_H$ via which $C_H$ acts on~$V_i$.

Under the assumptions of the previous paragraph, the condition $H \subset \SL_n$ implies the relation
$\sum \limits_{i=1}^t (\dim V_i) \chi_i = 0$, therefore the explicit expression of each indecomposable element of $\Gamma_I(G,H)$ is not uniquely determined. To resolve this ambiguity, consider the group $\widehat G = \GL_n$ and put $\widehat H = C_{\widehat G} \cdot H$. For every $i = 1,\ldots, t$, denote by the same symbol $\chi_i$ the character of~$C_{\widehat H}$ via which $C_{\widehat H}$ acts on~$V_i$. For every $i = 1, \ldots, n-1$, let $\pi_i \in \Lambda^+(\widehat G)$ be the highest weight of the $\widehat G$-module $\wedge^i \FF^n$. Then for each case in Table~\ref{table_Levi_sl} and for Cases~\ref{sl_spsl}--\nolinebreak\ref{sl_spspsp} in Table~\ref{table_sl} the explicit expression of each indecomposable element $(\lambda; \mu)$ of $\Gamma_I(G,H)$ is chosen in such a way that $(\lambda; \mu)$ belongs to $\Gamma(\widehat G, \widehat H)$ when regarded as an element of $\Lambda^+(\widehat G) \oplus \Lambda^+(\widehat H)$. As a consequence, all formulas of the form~(\ref{eqn_restriction}) obtained for the group~$G$ in these cases remain valid after replacing $G$ with~$\widehat G$ and $H$ with~$\widehat H$.

As a final remark, we point out that in all tables the symbols $\pi_i$ are used in two different meanings: as the fundamental weights of~$G$ and as certain dominant weights of~$H$ (specified above for each table). We note however that this does not cause any ambiguity.

\begin{longtable}{|c|l|l|l|}

\caption{The Levi subgroup case: $G = \SL_n$}
\label{table_Levi_sl}
\\
\hline
No. & Conditions & Rank & Indecomposable elements of~$\Gamma_I(G,H)$ \endfirsthead

\caption{The Levi subgroup case: $G = \SL_n$ (continued)} \\
\hline
No. & Conditions & Rank & Indecomposable elements of~$\Gamma_I(G,H)$ \endhead

\hline
\newcase \label{sl_Levi_p1}
&
\renewcommand{\tabcolsep}{0pt}%
\begin{tabular}{l}
$\mathbf a {=} (p,1)$, \\
$I {=} S$
\end{tabular}
&
$2n{-}2$
&
\renewcommand{\tabcolsep}{0pt}%
\begin{tabular}{l}
$(\pi_k; \pi_k {+} k \chi_1)$ for $1 {\le} k {\le} n{-}1$; \\
$(\pi_k; \pi_{k-1} {+} (k{-}1)\chi_1 {+} \chi_2)$ for $1 {\le} k {\le} n {-} 1$
\end{tabular}
\\

\hline
\newcase \label{sl_Levi_pq_1i_part1}
&
\renewcommand{\tabcolsep}{0pt}%
\begin{tabular}{l}
$\mathbf a {=} (p,q)$, \\
$I {=} \lbrace 1, i \rbrace$, \\
$2 {\le} i {\le} p$
\end{tabular}
&
$2 {+} \min(2i{-}1, 2q)$
&
\renewcommand{\tabcolsep}{0pt}%
\begin{tabular}{l}
$(\pi_1; \pi_1 {+} \chi_1)$, $(\pi_1; \pi'_1{+}\chi_2)$, \\
$(\pi_i; \pi_{i-k} {+} \pi'_{k} {+} (i{-}k)\chi_1 + k\chi_2)$ \\
for $0 {\le} k {\le} \min(i,q)$, \\
$(\pi_1 {+} \pi_i; \pi_{i+1-k} {+} \pi'_{k} {+} (i{+}1{-}k)\chi_1 + k\chi_2)$ \\
for $2 {\le} k {\le} \min(i{-}1,q)$
\end{tabular}
\\

\hline
\newcase \label{sl_Levi_pq_1i_part2}
&
\renewcommand{\tabcolsep}{0pt}%
\begin{tabular}{l}
$\mathbf a {=} (p,q)$, \\
$I {=} \lbrace 1, i \rbrace$ \\
$i {>} p {\ge} 2$
\end{tabular}
&
$3 {+} 2n{-}2i$
&
\renewcommand{\tabcolsep}{0pt}%
\begin{tabular}{l}
$(\pi_1; \pi_1 {+} \chi_1)$, $(\pi_1; \pi'_1{+}\chi_2)$, \\
$(\pi_i; \pi_{i-k} {+} \pi'_{k} {+} (i{-}k)\chi_1 {+} k\chi_2)$ \\
for $i{-}p {\le} k {\le} q$, \\
$(\pi_1 {+} \pi_i; \pi_{i+1-k} {+} \pi'_{k} {+} (i{+}1{-}k)\chi_1 {+} k\chi_2)$ \\
for $i{+}1{-}p {\le} k {\le} q$
\end{tabular}
\\

\hline
\newcase &
\renewcommand{\tabcolsep}{0pt}%
\begin{tabular}{l}
$\mathbf a {=} (p,q)$, \\
$I {=} \lbrace i, i{+}1 \rbrace$, \\
$p {\ge} i{+}1$
\end{tabular}
&
$2{+}\min(2i{+}1,2q)$
&
\renewcommand{\tabcolsep}{0pt}%
\begin{tabular}{l}
$(\pi_i; \pi_{i-k} {+} \pi'_k {+} (i{-}k)\chi_1 {+} k\chi_2)$ \\
for $0 {\le} k {\le} \min(i,q)$, \\
$(\pi_{i+1}; \pi_{i+1-k} {+} \pi'_{k} {+} (i{+}1{-}k)\chi_1 {+} k\chi_2)$ \\
for $0 {\le} k {\le} \min(i{+}1, q)$
\end{tabular}\\

\hline
\newcase &
\renewcommand{\tabcolsep}{0pt}%
\begin{tabular}{l}
$\mathbf a {=} (p,2)$, \\
$I {=} \lbrace i, j \rbrace$, \\
$2 {\le} i {<} j$, \\
$j{-}i {\ge} 2$, \\
$n {-} j {\ge} 2$
\end{tabular}
& $7$
&
\renewcommand{\tabcolsep}{0pt}%
\begin{tabular}{l}
$(\pi_i; \pi_i {+} i\chi_1)$, $(\pi_i; \pi_{i-1} {+} \pi'_1 {+} (i{-}1)\chi_1 {+} \chi_2)$,\\
$(\pi_i; \pi_{i-2} {+} (i{-}2)\chi_1 {+} 2\chi_2)$, \\
$(\pi_j; \pi_j {+} j\chi_1)$, $(\pi_j; \pi_{j-1} {+} \pi'_1 {+} (j{-}1)\chi_1 {+} \chi_2)$,\\
$(\pi_j; \pi_{j-2} {+} (j{-}2)\chi_1 {+} 2\chi_2)$, \\
$(\pi_i {+} \pi_j; \pi_{i-1} {+} \pi_{j-1} {+} (i{-}1)\chi_1 {+} (j{-}1)\chi_2 {+} 2\chi_2)$
\end{tabular}\\

\hline
\newcase \label{sl_Levi_pq1_part1}
&
\renewcommand{\tabcolsep}{0pt}%
\begin{tabular}{l}
$\mathbf a {=} (p,q,1)$, \\
$I {=} \lbrace i \rbrace$, \\
$1 {\le} i {\le} p$
\end{tabular}
&
$2 {+} \min(2i{-}1,2q)$
&
\renewcommand{\tabcolsep}{0pt}%
\begin{tabular}{l}
$(\pi_i; \pi_{i-1-k} {+} \pi'_{k} + (i{-}1{-}k)\chi_1 {+} k\chi_2 {+} \chi_3)$ \\
for $0 {\le} k {\le} \min(i{-}1,q)$, \\
$(\pi_i; \pi_{i-k} {+} \pi'_{k} {+} (i{-}k)\chi_1 {+} k\chi_2)$ \\
for $0 {\le} k {\le} \min(i,q)$
\end{tabular}\\

\hline
\newcase &
\renewcommand{\tabcolsep}{0pt}%
\begin{tabular}{l}
$\mathbf a {=} (p,q,1)$, \\
$I {=} \lbrace i \rbrace$, \\
$i {>} p$
\end{tabular}
&
$1 {+} 2n {-} 2i$
&
\renewcommand{\tabcolsep}{0pt}%
\begin{tabular}{l}
$(\pi_i; \pi_{i-1-k} {+} \pi'_{k} {+} (i{-}1{-}k)\chi_1 {+} k\chi_2 {+} \chi_3)$ \\
for $i{-}1{-}p {\le} k {\le} q{-}1$, \\
$(\pi_i; \pi_{i-k} {+} \pi'_{k} {+} (i{-}k)\chi_1 {+} k\chi_2)$ \\
for $i{-}p {\le} k {\le} q{-}1$
\end{tabular}\\

\hline
\newcase \label{sl_Levi_pqr}
&
\renewcommand{\tabcolsep}{0pt}%
\begin{tabular}{l}
$\mathbf a {=} (p,q,r)$, \\
$I {=} \lbrace 2 \rbrace$, \\
$r {\ge} 2$
\end{tabular}
& $6$
&
\renewcommand{\tabcolsep}{0pt}%
\begin{tabular}{l}
$(\pi_2; \pi_2 {+} 2\chi_1)$, $(\pi_2; \pi'_2 {+} 2\chi_2)$,
$(\pi_2; \pi''_2 {+} 2\chi_3)$, \\
$(\pi_2; \pi_1 {+} \pi'_1 {+} \chi_1 {+} \chi_2)$, 
$(\pi_2; \pi_1 {+} \pi''_1 {+} \chi_1 {+} \chi_3)$, \\
$(\pi_2; \pi'_1 {+} \pi''_1 {+} \chi_2 {+} \chi_3)$
\end{tabular}\\

\hline
\newcase &
\renewcommand{\tabcolsep}{0pt}%
\begin{tabular}{l}
$\mathbf a {=} (1,1,\ldots,1)$, \\
$I {=} \lbrace 1 \rbrace$
\end{tabular}
& $n$
&
$(\pi_1; \chi_k)$ for $1 {\le} k {\le} n$
\\

\hline

\end{longtable}


\begin{longtable}{|c|l|l|l|}

\caption{The Levi subgroup case: $G \not\simeq \SL_n$}
\label{table_Levi_non-sl}
\\
\hline
No. & Conditions & Rank & Indecomposable elements of~$\Gamma_I(G,H)$ \endfirsthead

\caption{The Levi subgroup case: $G \not\simeq \SL_n$ (continued)} \\
\hline
No. & Conditions & Rank & Indecomposable elements of~$\Gamma_I(G,H)$ \endhead

\hline

\newcase & \multicolumn{3}{|c|}{$G = \Sp_{2n}$, $n \ge 2$}\\

\hline
\no &
\renewcommand{\tabcolsep}{0pt}%
\begin{tabular}{l}
$J {=} \lbrace 1 \rbrace$, \\
$I {=} \lbrace i \rbrace$ \\
\end{tabular}
& $4 {-} \delta_i^1 {-} \delta_i^n$
&
\renewcommand{\tabcolsep}{0pt}%
\begin{tabular}{l}
$(\pi_i; \pi_i)$, $(\pi_i; \pi_{i-1} {-} \pi_1)$, \\
$(\pi_i; \pi_{i+1} {-} \pi_1)$ ($i {\le} n{-}1$), 
$(\pi_i; \pi_i {-} 2\pi_1)$ ($i {\ge} 2$)
\end{tabular}\\

\hline
\no &
\renewcommand{\tabcolsep}{0pt}%
\begin{tabular}{l}
$J {=} \lbrace j \rbrace$, \\
$I {=} \lbrace 1 \rbrace$ \\
\end{tabular}
& $4 {-} \delta_j^1 {-} \delta_j^n$
&
\renewcommand{\tabcolsep}{0pt}%
\begin{tabular}{l}
$(\pi_1; \pi_1)$, $(\pi_1; \pi_{j-1} {-} \pi_j)$, \\
$(\pi_1; \pi_{j+1} {-} \pi_j)$ ($j {\le} n{-}1$), 
$(2\pi_1; 0)$ ($j {\ge} 2$)
\end{tabular}\\

\hline
\no &
\renewcommand{\tabcolsep}{0pt}%
\begin{tabular}{l}
$J {=} \lbrace n \rbrace$, \\
$I {=} \lbrace n \rbrace$
\end{tabular}
& $1 {+} n$
&
\renewcommand{\tabcolsep}{0pt}%
\begin{tabular}{l}
$(\pi_n; \pi_n)$, $(\pi_n; - \pi_n)$, \\
$(\pi_n; 2\pi_k {-} \pi_n)$ for $1 {\le} k {\le} n{-}1$
\end{tabular}\\

\hline
\newcase & \multicolumn{3}{|c|}{$G = \Spin_{2n+1}$, $n \ge 3$}\\

\hline
\no &
\renewcommand{\tabcolsep}{0pt}%
\begin{tabular}{l}
$J {=} \lbrace 1 \rbrace$, \\
$I {=} \lbrace i \rbrace$ \\
\end{tabular}
& $4 {-} \delta_i^1 {-} \delta_i^n$
&
\renewcommand{\tabcolsep}{0pt}%
\begin{tabular}{l}
$(\pi_i; \pi_i)$, 
$(\pi_i; \pi_{i-1} {-} \pi_1)$ ($i {\le} n{-}1$), 
$(2\pi_n; \pi_{n-1} {-} \pi_1)$ ($i {=} n$), \\
$(\pi_i; \pi_{i+1} {-} \pi_1)$ ($i {\le} n {-} 2$), 
$(\pi_{n-1}; 2\pi_n {-} \pi_1)$ ($i {=} n{-}1$), \\
$(\pi_i; \pi_i {-} 2\pi_1)$ ($2 {\le} i {\le} n{-}1$), 
$(\pi_n; \pi_n {-} \pi_1)$ ($i {=} n$)
\end{tabular}\\

\hline
\no &
\renewcommand{\tabcolsep}{0pt}%
\begin{tabular}{l}
$J {=} \lbrace j \rbrace$, \\
$I {=} \lbrace 1 \rbrace$ \\
\end{tabular}
& $4 {-} \delta_j^1 {-} \delta_j^n$
&
\renewcommand{\tabcolsep}{0pt}%
\begin{tabular}{l}
$(\pi_1; \pi_1)$, 
$(\pi_1; \pi_{j-1} {-} \pi_j)$ ($j {\le} n{-}1$), \\
$(\pi_1; \pi_{n-1} {-} 2\pi_n)$ ($j {=} n$), \\
$(\pi_1; \pi_{j+1} {-} \pi_j)$ ($j {\le} n{-}2$), 
$(\pi_1; 2\pi_n {-} \pi_{n-1})$ ($j {=} n{-}1$), \\
$(\pi_1; 0)$ ($j {=} n$), 
$(2\pi_1; 0)$ ($2 {\le} j {\le} n{-}1$)
\end{tabular}\\

\hline
\no &
\renewcommand{\tabcolsep}{0pt}%
\begin{tabular}{l}
$J {=} \lbrace n \rbrace$, \\
$I {=} \lbrace n \rbrace$
\end{tabular}
& $1 {+} n$
&
\renewcommand{\tabcolsep}{0pt}%
\begin{tabular}{l}
$(\pi_n; \pi_n)$, $(\pi_n; - \pi_n)$, \\
$(\pi_n; \pi_k {-} \pi_n)$ for $1 {\le} k {\le} n{-}1$
\end{tabular}\\

\hline
\newcase & \multicolumn{3}{|c|}{$G = \Spin_{2n}$, $n \ge 4$}\\

\hline
\no
&
\renewcommand{\tabcolsep}{0pt}%
\begin{tabular}{l}
$J {=} \lbrace 1 \rbrace$,\\
$I {=} \lbrace i, n \rbrace$,\\
$1 {\le} i {\le} n{-}2$
\end{tabular}
&
$6 - \delta_i^1$
&
\renewcommand{\tabcolsep}{0pt}%
\begin{tabular}{l}
$(\pi_i; \pi_i)$, $(\pi_i; \pi_{i-1} {-} \pi_1)$, 
$(\pi_i; \pi_{i+1} {-} \pi_1)$ ($i {\le} n {-} 3$), \\
$(\pi_{n-2}; \pi_{n-1} {+} \pi_n {-} \pi_1)$ ($i {=} n{-}2$), 
$(\pi_i; \pi_i {-} 2\pi_1)$ $(i {\ge} 2)$,\\
$(\pi_n; \pi_n)$, $(\pi_n; \pi_{n-1} {-} \pi_1)$
\end{tabular}
\\

\hline
\no
&
\renewcommand{\tabcolsep}{0pt}%
\begin{tabular}{l}
$J {=} \lbrace 1 \rbrace$,\\
$I {=} \lbrace n {-} 1, n \rbrace$
\end{tabular}
&
$5$
&
\renewcommand{\tabcolsep}{0pt}%
\begin{tabular}{l}
$(\pi_{n-1}; \pi_{n-1})$, $(\pi_{n-1}; \pi_n {-} \pi_1)$, \\
$(\pi_n; \pi_n)$, $(\pi_n; \pi_{n-1} {-} \pi_1)$, 
$(\pi_{n-1} {+} \pi_n; \pi_{n-2} {-} \pi_1)$
\end{tabular}
\\

\hline
\no
&
\renewcommand{\tabcolsep}{0pt}%
\begin{tabular}{l}
$J {=} \lbrace 2 \rbrace$,\\
$I {=} \lbrace n \rbrace$
\end{tabular}
&
$4$
&
\renewcommand{\tabcolsep}{0pt}%
\begin{tabular}{l}
$(\pi_n; \pi_n)$, $(\pi_n; \pi_n {-} \pi_2)$,\\
$(\pi_n; \pi_1 {+} \pi_{n-1} {-} \pi_2)$, $(2\pi_n; \pi_{n-2} {-} \pi_2)$
\end{tabular}
\\

\hline
\no
&
\renewcommand{\tabcolsep}{0pt}%
\begin{tabular}{l}
$J {=} \lbrace 3 \rbrace$,\\
$I {=} \lbrace n \rbrace$,\\
$n {\ge} 5$
\end{tabular}
&
$7 {-} \delta_n^5$
&
\renewcommand{\tabcolsep}{0pt}%
\begin{tabular}{l}
$(\pi_n; \pi_n)$, $(\pi_n; \pi_1 {+} \pi_n {-} \pi_3)$, 
$(\pi_n; \pi_2 {+} \pi_{n-1} {-} \pi_3)$, \\
$(\pi_n; \pi_{n-1} {-} \pi_3)$, 
$(2\pi_n; \pi_1 {+} \pi_{n-2} {-} \pi_3)$, $(2\pi_n; \pi_{n-3} {-} \pi_3)$,\\
$(2\pi_n; \pi_2 {+} \pi_{n-2} {-} 2\pi_3)$ ($n {\ge} 6$)
\end{tabular}
\\

\hline
\no
&
\renewcommand{\tabcolsep}{0pt}%
\begin{tabular}{l}
$J {=} \lbrace j \rbrace$,\\
$I {=} \lbrace 1 \rbrace$,\\
$1 {\le} j {\le} n {-} 2$
\end{tabular}
&
$4 {-} \delta_j^1$
&
\renewcommand{\tabcolsep}{0pt}%
\begin{tabular}{l}
$(\pi_1; \pi_1)$, $(\pi_1; \pi_{j-1} {-} \pi_j)$, 
$(\pi_1; \pi_{j+1} {-} \pi_j)$ ($j {\le} n {-} 3$), \\
$(\pi_1; \pi_{n-1} {+} \pi_n {-} \pi_{n-2})$ ($j {=} n {-} 2$), \\
$(2\pi_1; 0)$ ($j{\ge}2$)
\end{tabular}
\\

\hline
\no
&
\renewcommand{\tabcolsep}{0pt}%
\begin{tabular}{l}
$J {=} \lbrace n \rbrace$,\\
$I {=} \lbrace 1, 2 \rbrace$\\
\end{tabular}
&
$6$
&
\renewcommand{\tabcolsep}{0pt}%
\begin{tabular}{l}
$(\pi_1; \pi_1)$, $(\pi_1; \pi_{n-1} {-} \pi_n)$, 
$(\pi_2; \pi_2)$, $(\pi_2; 0)$, \\
$(\pi_2; \pi_1 {+} \pi_{n-1} {-} \pi_n)$,
$(\pi_2; \pi_{n-2} {-} 2\pi_n)$
\end{tabular}
\\

\hline
\no
&
\renewcommand{\tabcolsep}{0pt}%
\begin{tabular}{l}
$J {=} \lbrace n \rbrace$,\\
$I {=} \lbrace 3 \rbrace$,\\
$n {\ge} 5$
\end{tabular}
&
$7 {-} \delta_n^5$
&
\renewcommand{\tabcolsep}{0pt}%
\begin{tabular}{l}
$(\pi_3; \pi_3)$, $(\pi_3; \pi_1)$, 
$(\pi_3; \pi_2 {+} \pi_{n-1} {-} \pi_n)$, $(\pi_3; \pi_{n-1} {-} \pi_n)$, \\
$(\pi_3; \pi_1 {+} \pi_{n-2} {-} 2\pi_n)$, $(\pi_3; \pi_{n-3} {-} 2\pi_n)$, \\
$(2\pi_3; \pi_2 {+} \pi_{n-2} {-} 2\pi_n)$ ($n {\ge} 6$)
\end{tabular}
\\

\hline
\no
&
\renewcommand{\tabcolsep}{0pt}%
\begin{tabular}{l}
$J {=} \lbrace n \rbrace$,\\
$I {=} \lbrace 1, n{-}1 \rbrace$
\end{tabular}
&
$n {+} 1$
&
\renewcommand{\tabcolsep}{0pt}%
\begin{tabular}{l}
$(\pi_1; \pi_1)$, $(\pi_1; \pi_{n-1} {-} \pi_n)$, 
$(\pi_{n-1}; \pi_{n-1})$, \\
$(\pi_{n-1}; \pi_{n-2k-1} {-} \pi_n)$ for $1 {\le} k {\le} [\frac{n-1}2]$, \\
$(\pi_1 {+} \pi_{n-1}; \pi_{n - 2k} {-} \pi_n)$ for $1 {\le} k {\le} [\frac{n-2}2]$
\end{tabular}
\\

\hline
\no
&
\renewcommand{\tabcolsep}{0pt}%
\begin{tabular}{l}
$J {=} \lbrace n \rbrace$,\\
$I {=} \lbrace 1, n \rbrace$
\end{tabular}
&
$n{+}1$
&
\renewcommand{\tabcolsep}{0pt}%
\begin{tabular}{l}
$(\pi_1; \pi_1)$, $(\pi_1; \pi_{n-1} {-} \pi_n)$, 
$(\pi_n; \pi_n)$, \\
$(\pi_n; \pi_{n-2k} {-} \pi_n)$ for $1 {\le} k {\le} [\frac{n}2]$,\\
$(\pi_1 {+} \pi_n; \pi_{n-2k-1} {-} \pi_n)$ for $1 {\le} k {\le} [\frac{n-3}2]$
\end{tabular}
\\

\hline
\no
&
\renewcommand{\tabcolsep}{0pt}%
\begin{tabular}{l}
$J {=} \lbrace n \rbrace$,\\
$I {=} \lbrace n {-} 1, n \rbrace$
\end{tabular}
&
$n{+}1$
&
\renewcommand{\tabcolsep}{0pt}%
\begin{tabular}{l}
$(\pi_{n-1}; \pi_{n-1})$, $(\pi_n; \pi_n)$,\\
$(\pi_{n-1}; \pi_{n-2k-1} {-} \pi_n)$ for $1 {\le} k {\le} [\frac{n-1}2]$,\\
$(\pi_n; \pi_{n-2k} {-} \pi_n)$ for $1 {\le} k {\le} [\frac{n}2]$
\end{tabular}
\\

\hline
\no
&
\renewcommand{\tabcolsep}{0pt}%
\begin{tabular}{l}
$J {=} \lbrace 1, 2 \rbrace$,\\
$I {=} \lbrace n \rbrace$
\end{tabular}
&
$5$
&
\renewcommand{\tabcolsep}{0pt}%
\begin{tabular}{l}
$(\pi_n; \pi_n)$, $(\pi_n; \pi_n {-} \pi_2)$, 
$(\pi_n; \pi_{n-1} {-} \pi_1)$, \\
$(\pi_n; \pi_1 {+} \pi_{n-1} {-} \pi_2)$,
$(2\pi_n; \pi_{n-2} {-} \pi_2)$
\end{tabular}
\\

\hline
\no
&
\renewcommand{\tabcolsep}{0pt}%
\begin{tabular}{l}
$J {=} \lbrace 1, n{-}1 \rbrace$,\\
$I {=} \lbrace n \rbrace$
\end{tabular}
&
$n$
&
\renewcommand{\tabcolsep}{0pt}%
\begin{tabular}{l}
$(\pi_n; \pi_n)$, $(\pi_n; \pi_{n-1} {-} \pi_1)$,\\
$(\pi_n; \pi_{n-2k-1} {-} \pi_{n-1})$ for $1 {\le} k {\le} [\frac{n-1}2]$,\\
$(\pi_n; \pi_{n-2k} {-} \pi_1 {-} \pi_{n-1})$ for $1 {\le} k {\le} [\frac{n-2}2]$
\end{tabular}
\\

\hline
\no
&
\renewcommand{\tabcolsep}{0pt}%
\begin{tabular}{l}
$J {=} \lbrace 1, n \rbrace$,\\
$I {=} \lbrace n \rbrace$
\end{tabular}
&
$n$
&
\renewcommand{\tabcolsep}{0pt}%
\begin{tabular}{l}
$(\pi_n; \pi_n)$, $(\pi_n; \pi_{n-1} {-} \pi_1)$,\\
$(\pi_n; \pi_{n-2k} {-} \pi_n)$ for $1 {\le} k {\le} [\frac{n}2]$,\\
$(\pi_n; \pi_{n-2k-1} {-} \pi_1 {-} \pi_n)$ for $1 {\le} k {\le} [\frac{n-3}2]$
\end{tabular}
\\

\hline
\no
&
\renewcommand{\tabcolsep}{0pt}%
\begin{tabular}{l}
$J {=} \lbrace j, n \rbrace$,\\
$I {=} \lbrace 1 \rbrace$,\\
$1 {\le} j {\le} n {-} 2$
\end{tabular}
&
$5 {-} \delta_j^1$
&
\renewcommand{\tabcolsep}{0pt}%
\begin{tabular}{l}
$(\pi_1; \pi_1)$, $(\pi_1; \pi_{j-1} {-} \pi_j)$, 
$(\pi_1; \pi_{j+1} {-} \pi_j)$ ($j {\le} n{-}3$), \\
$(\pi_1; \pi_{n-1} {+} \pi_n {-} \pi_{n-2})$ ($j {=} n{-}2$), 
$(\pi_1; \pi_{n-1} {-} \pi_n)$, \\
$(2\pi_1; 0)$ ($j {\ge} 2$)
\end{tabular}
\\

\hline
\no
&
\renewcommand{\tabcolsep}{0pt}%
\begin{tabular}{l}
$J {=} \lbrace n {-} 1, n \rbrace$,\\
$I {=} \lbrace 1 \rbrace$
\end{tabular}
&
$4$
&
\renewcommand{\tabcolsep}{0pt}%
\begin{tabular}{l}
$(\pi_1; \pi_1)$, $(\pi_1; \pi_{n-2} {-} \pi_{n-1} {-} \pi_n)$,\\
$(\pi_1; \pi_{n-1} {-} \pi_n)$, $(\pi_1; \pi_n {-} \pi_{n-1})$
\end{tabular}
\\

\hline
\no
&
\renewcommand{\tabcolsep}{0pt}%
\begin{tabular}{l}
$J {=} \lbrace n{-}1, n \rbrace$,\\
$I {=} \lbrace n \rbrace$
\end{tabular}
&
$n$
&
\renewcommand{\tabcolsep}{0pt}%
\begin{tabular}{l}
$(\pi_n; \pi_n)$,\\
$(\pi_n; \pi_{n-2k-1} {-} \pi_{n-1})$ for $1 {\le} k {\le} [\frac{n-1}2]$,\\
$(\pi_n; \pi_{n-2k} {-} \pi_n)$ for $1 {\le} k {\le} [\frac{n}2]$
\end{tabular}
\\

\hline
\newcase & \multicolumn{3}{|c|}{$G = \mathsf E_6$}\\

\hline
\no
&
\renewcommand{\tabcolsep}{0pt}%
\begin{tabular}{l}
$J {=} \lbrace 1 \rbrace$,\\
$I {=} \lbrace 1, 6 \rbrace$
\end{tabular}
&
$7$
&
\renewcommand{\tabcolsep}{0pt}%
\begin{tabular}{l}
$(\pi_1; \pi_1)$, $(\pi_1; \pi_3 {-} \pi_1)$,
$(\pi_1; \pi_6 {-} \pi_1)$, 
$(\pi_6; \pi_6)$, 
$(\pi_6; - \pi_1)$, \\
$(\pi_6; \pi_2 {-} \pi_1)$, 
$(\pi_1 {+} \pi_6; \pi_5 {-} \pi_1)$
\end{tabular}
\\

\hline
\no
&
\renewcommand{\tabcolsep}{0pt}%
\begin{tabular}{l}
$J {=} \lbrace 1 \rbrace$,\\
$I {=} \lbrace 2 \rbrace$
\end{tabular}
&
$4$
&
\renewcommand{\tabcolsep}{0pt}%
\begin{tabular}{l}
$(\pi_2; \pi_2)$, $(\pi_2; 0)$, 
$(\pi_2; \pi_5 {-} \pi_1)$,\\
$(\pi_3; \pi_3 {-} 2\pi_1)$
\end{tabular}
\\

\hline
\no
&
\renewcommand{\tabcolsep}{0pt}%
\begin{tabular}{l}
$J {=} \lbrace 1 \rbrace$,\\
$I {=} \lbrace 3 \rbrace$
\end{tabular}
&
$6$
&
\renewcommand{\tabcolsep}{0pt}%
\begin{tabular}{l}
$(\pi_3; \pi_3)$, $(\pi_3; \pi_6)$, 
$(\pi_3; \pi_2 {-} \pi_1)$, 
$(\pi_3; \pi_4 {-} \pi_1)$, \\
$(\pi_3; \pi_3 {+} \pi_6 {-} 2\pi_1)$, $(\pi_3; \pi_5 {-} 2\pi_1)$
\end{tabular}
\\

\hline
\no
&
\renewcommand{\tabcolsep}{0pt}%
\begin{tabular}{l}
$J {=} \lbrace 1 \rbrace$,\\
$I {=} \lbrace 5 \rbrace$
\end{tabular}
&
$6$
&
\renewcommand{\tabcolsep}{0pt}%
\begin{tabular}{l}
$(\pi_5; \pi_5)$, $(\pi_5; \pi_3 {-} \pi_1)$, 
$(\pi_5; \pi_6 - \pi_1)$, 
$(\pi_5; \pi_4 {+} \pi_5 {-} \pi_1)$, \\
$(\pi_5; \pi_2 {-} 2\pi_1)$, $(\pi_5; \pi_4 {-} 2\pi_1)$
\end{tabular}
\\

\hline
\no
&
\renewcommand{\tabcolsep}{0pt}%
\begin{tabular}{l}
$J {=} \lbrace 2 \rbrace$,\\
$I {=}\lbrace 1 \rbrace$
\end{tabular}
&
$4$
&
\renewcommand{\tabcolsep}{0pt}%
\begin{tabular}{l}
$(\pi_1; \pi_1)$, $(\pi_1; \pi_1 {-} \pi_2)$, 
$(\pi_1; \pi_5 {-} \pi_2)$, \\
$(2\pi_1; \pi_3 {-} \pi_2)$
\end{tabular}
\\

\hline
\no
&
\renewcommand{\tabcolsep}{0pt}%
\begin{tabular}{l}
$J {=} \lbrace 3 \rbrace$,\\
$I {=} \lbrace 1 \rbrace$
\end{tabular}
&
$6$
&
\renewcommand{\tabcolsep}{0pt}%
\begin{tabular}{l}
$(\pi_1; \pi_1)$, $(\pi_1; \pi_1 {+} \pi_6 {-} \pi_3)$, 
$(\pi_1; \pi_2 {-} \pi_3)$, $(\pi_1; \pi_4 {-} \pi_3)$, \\
$(2\pi_1; \pi_5 {-} \pi_3)$, $(2\pi_1; \pi_6)$
\end{tabular}
\\

\hline
\no
&
\renewcommand{\tabcolsep}{0pt}%
\begin{tabular}{l}
$J {=} \lbrace 5 \rbrace$,\\
$I {=} \lbrace 1 \rbrace$
\end{tabular}
&
$6$
&
\renewcommand{\tabcolsep}{0pt}%
\begin{tabular}{l}
$(\pi_1; \pi_1)$, $(\pi_1; \pi_3 {-} \pi_5)$, 
$(\pi_1; \pi_4)$, $(\pi_1; \pi_6 {-} \pi_5)$, \\
$(2\pi_1; \pi_2 {-} \pi_5)$, $(2\pi_1; \pi_4 {-} \pi_5)$
\end{tabular}
\\

\hline
\no
&
\renewcommand{\tabcolsep}{0pt}%
\begin{tabular}{l}
$J {=} \lbrace 6 \rbrace$,\\
$I {=} \lbrace 1 \rbrace$
\end{tabular}
&
$3$
&
$(\pi_1; \pi_1)$, $(\pi_1; \pi_2 {-} \pi_6)$, $(\pi_1; {-} \pi_6)$
\\

\hline
\no
&
\renewcommand{\tabcolsep}{0pt}%
\begin{tabular}{l}
$J {=} \lbrace 1, 6 \rbrace$,\\
$I {=} \lbrace 1 \rbrace$
\end{tabular}
&
$6$
&
\renewcommand{\tabcolsep}{0pt}%
\begin{tabular}{l}
$(\pi_1; \pi_1)$, $(\pi_1; \pi_3 {-} \pi_1)$, 
$(\pi_1; \pi_6 {-} \pi_1)$, $(\pi_1; - \pi_6)$, \\
$(\pi_1; \pi_2 {-} \pi_6)$, $(\pi_1; \pi_5 {-} \pi_1 {-} \pi_6)$
\end{tabular}
\\

\hline
\newcase & \multicolumn{3}{|c|}{$G = \mathsf E_7$}\\

\hline
\no
&
\renewcommand{\tabcolsep}{0pt}%
\begin{tabular}{l}
$J {=} \lbrace 1 \rbrace$,\\
$I {=} \lbrace 7 \rbrace$
\end{tabular}
&
$4$
&
\renewcommand{\tabcolsep}{0pt}%
\begin{tabular}{l}
$(\pi_7; \pi_7)$, $(\pi_7; \pi_7 {-} \pi_1)$, 
$(\pi_7; \pi_2 {-} \pi_1)$, \\
$(2\pi_7; \pi_6 {-} \pi_1)$
\end{tabular}
\\

\hline
\no
&
\renewcommand{\tabcolsep}{0pt}%
\begin{tabular}{l}
$J {=} \lbrace 2 \rbrace$,\\
$I {=} \lbrace 7 \rbrace$
\end{tabular}
&
$7$
&
\renewcommand{\tabcolsep}{0pt}%
\begin{tabular}{l}
$(\pi_7; \pi_7)$, $(\pi_7; \pi_6 {-} \pi_2)$, 
$(\pi_7; \pi_3 {-} \pi_2)$, 
$(\pi_7; \pi_1 {-} \pi_2)$, \\
$(2\pi_7; \pi_5 {-} \pi_2)$, $(2\pi_7; 0)$, $(2\pi_7; \pi_4 {-} 2\pi_2)$
\end{tabular}
\\

\hline
\no
&
\renewcommand{\tabcolsep}{0pt}%
\begin{tabular}{l}
$J {=} \lbrace 7 \rbrace$,\\
$I {=} \lbrace 1 \rbrace$
\end{tabular}
&
$4$
&
\renewcommand{\tabcolsep}{0pt}%
\begin{tabular}{l}
$(\pi_1; \pi_1)$, $(\pi_1; 0)$,\\
$(\pi_1; \pi_2 {-} \pi_7)$, $(\pi_1; \pi_6 {-} 2\pi_7)$
\end{tabular}
\\

\hline
\no
&
\renewcommand{\tabcolsep}{0pt}%
\begin{tabular}{l}
$J {=} \lbrace 7 \rbrace$,\\
$I {=} \lbrace 2 \rbrace$
\end{tabular}
&
$7$
&
\renewcommand{\tabcolsep}{0pt}%
\begin{tabular}{l}
$(\pi_2; \pi_2)$, $(\pi_2; \pi_6 {-} \pi_7)$, $(\pi_2; \pi_3 {-} \pi_7)$, 
$(\pi_2; \pi_1 {-} \pi_7)$, \\
$(\pi_2; \pi_5 {-} 2\pi_7)$, $(\pi_2; \pi_2 {-} 2\pi_7)$, 
$(2\pi_2; \pi_4 {-} 2\pi_7)$
\end{tabular}
\\

\hline
\no
&
\renewcommand{\tabcolsep}{0pt}%
\begin{tabular}{l}
$J {=} \lbrace 7 \rbrace$,\\
$I {=} \lbrace 7 \rbrace$
\end{tabular}
&
$4$
&
\renewcommand{\tabcolsep}{0pt}%
\begin{tabular}{l}
$(\pi_7; \pi_7)$, $(\pi_7; -\pi_7)$,\\
$(\pi_7; \pi_6 {-} \pi_1)$, $(\pi_7; \pi_1 {-} \pi_7)$
\end{tabular}
\\

\hline
\end{longtable}


\begin{longtable}{|c|l|l|l|}

\caption{The non-Levi symmetric subgroup case}
\label{table_sym}
\\
\hline
No. & Conditions & Rank & Indecomposable elements of~$\Gamma_I(G,H)$ \endfirsthead

\caption{The non-Levi symmetric subgroup case (continued)} \\
\hline
No. & Conditions & Rank & Indecomposable elements of~$\Gamma_I(G,H)$ \endhead

\hline
\newcase \label{sl_so}
& \multicolumn{3}{|c|}{$G = \SL_{n}$, $H = \SO_{n}$, $n \ge 3$}\\

\hline
\no \label{sl_so_odd_i}
&
\renewcommand{\tabcolsep}{0pt}%
\begin{tabular}{l}
$n = 2l{+}1$, \\
$I {=} \lbrace i \rbrace$, \\
$i {\le} l$
\end{tabular}
&
$i {+} 1$
&
\renewcommand{\tabcolsep}{0pt}%
\begin{tabular}{l}
$(\pi_i; \pi_i)$ ($i {\le} l{-}1$),
$(\pi_l; 2\pi_l)$ ($i{=}l$),\\
$(2\pi_i; 2\pi_k)$ for $0 {\le} k {\le} i{-}1$
\end{tabular}
\\

\hline
\no \label{sl_so_even_i}
&
\renewcommand{\tabcolsep}{0pt}%
\begin{tabular}{l}
$n = 2l$, \\
$I {=} \lbrace i \rbrace$, \\
$i {\le} l$
\end{tabular}
&
$i {+} 1$
&
\renewcommand{\tabcolsep}{0pt}%
\begin{tabular}{l}
$(\pi_i; \pi_i)$ ($i {\le} l{-}2$),
$(\pi_{l-1}; \pi_{l-1} {+} \pi_l)$ ($i {=} l{-}1$),\\
$(\pi_l; 2\pi_{l-1})$ ($i{=}l$),
$(\pi_l; 2\pi_l)$ ($i{=}l$),\\
$(2\pi_i; 2\pi_k)$ for $0 {\le} k {\le} \min(i{-}1,l{-}2)$
\end{tabular}
\\

\hline
\newcase \label{sl_sp}
& \multicolumn{3}{|c|}{$G = \SL_{2n}$, $H = \Sp_{2n}$, $n \ge 2$}\\

\hline
\no \label{sl_sp_1ipart1}
&
\renewcommand{\tabcolsep}{0pt}%
\begin{tabular}{l}
$I{=} \lbrace 1,i \rbrace$,\\
$2 {\le} i {\le} n$
\end{tabular}
&
$i{+}1$
&
\renewcommand{\tabcolsep}{0pt}%
\begin{tabular}{l}
$(\pi_1; \pi_1)$, $(\pi_i; \pi_{i-2k})$ for $0 {\le} k {\le} [\frac{i}2]$,\\
$(\pi_1 {+} \pi_i; \pi_{i-1-2k})$ for $0 {\le} k {\le} [\frac{i-3}2]$
\end{tabular}
\\

\hline
\no \label{sl_sp_1ipart2}
&
\renewcommand{\tabcolsep}{0pt}%
\begin{tabular}{l}
$I{=} \lbrace 1,i \rbrace$,\\
$i {\ge} n{+}1$
\end{tabular}
&
$2n{-}i{+}2$
&
\renewcommand{\tabcolsep}{0pt}%
\begin{tabular}{l}
$(\pi_1; \pi_1)$, $(\pi_i; \pi_{2n-i-2k})$ for $0 {\le} k {\le} [\frac{2n-i}2]$,\\
$(\pi_1 {+} \pi_i; \pi_{2n-i+1-2k})$ for $0 {\le} k {\le} [\frac{2n-i-1}2]$
\end{tabular}
\\

\hline
\no \label{sl_sp_ii+1}
&
\renewcommand{\tabcolsep}{0pt}%
\begin{tabular}{l}
$I {=} \lbrace i, i{+}1 \rbrace$,\\
$i {\le} n{-}1$
\end{tabular}
&
$i {+} 2$
&
\renewcommand{\tabcolsep}{0pt}%
\begin{tabular}{l}
$(\pi_i; \pi_{i-2k})$ for $0 {\le} k {\le} [\frac{i}2]$,\\
$(\pi_{i+1}; \pi_{i+1-2k})$ for $0 {\le} k {\le} [\frac{i+1}2]$
\end{tabular}
\\

\hline
\no \label{sl_sp_123}
&
$I {=} \lbrace 1,2,3 \rbrace$
&
$6 {-} \delta_n^2$
&
\renewcommand{\tabcolsep}{0pt}%
\begin{tabular}{l}
$(\pi_1; \pi_1)$, $(\pi_2; \pi_2)$, $(\pi_2; 0)$,\\
$(\pi_3; \pi_3)$ ($n {\ge} 3$), $(\pi_3; \pi_1)$,
$(\pi_1 {+} \pi_3; \pi_2)$
\end{tabular}
\\

\hline
\no \label{sl_sp_122n-1}
&
\renewcommand{\tabcolsep}{0pt}%
\begin{tabular}{l}
$n {\ge} 3$, \\
$I {=} \lbrace 1,2,2n{-}1 \rbrace$
\end{tabular}
&
$6$
&
\renewcommand{\tabcolsep}{0pt}%
\begin{tabular}{l}
$(\pi_1; \pi_1)$, $(\pi_2; \pi_2)$, $(\pi_2; 0)$, $(\pi_{2n-1}; \pi_1)$,\\
$(\pi_1 {+} \pi_{2n-1}; \pi_2)$, $(\pi_2 {+} \pi_{2n-1}; \pi_3)$
\end{tabular}
\\

\hline
\newcase \label{sp_spsp}
& \multicolumn{3}{|c|}{$G = \Sp_{2n}$, $H = \Sp_{2p} \times \Sp_{2q}$, $p \ge q \ge 1$, $p + q = n$}\\

\hline
\no \label{sp_spspij}
&
\renewcommand{\tabcolsep}{0pt}%
\begin{tabular}{l}
$q {=} 1$, \\
$I {=} \lbrace i, j \rbrace$, \\
$i{<}j$
\end{tabular}
&
$7 {-} \delta_i^1 {-} \delta_j^{i+1} {-} \delta_j^n$
&
\renewcommand{\tabcolsep}{0pt}%
\begin{tabular}{l}
$(\pi_i; \pi_i)$, $(\pi_i; \pi_{i-1} {+} \pi'_1)$, $(\pi_i; \pi_{i-2})$ ($i {\ge} 2$), \\
$(\pi_j; \pi_j)$ ($j {\le} n{-}1$), $(\pi_j; \pi_{j-1} {+} \pi'_1)$, \\
$(\pi_j; \pi_{j-2})$, $(\pi_i {+} \pi_j; \pi_{i-1} {+} \pi_{j-1})$ ($j {\le} i{+}2$)
\end{tabular}
\\

\hline
\no \label{sp_spsp12}
&
\renewcommand{\tabcolsep}{0pt}%
\begin{tabular}{l}
$q {\ge} 2$, \\
$I {=} \lbrace 1, 2 \rbrace$
\end{tabular}
&
$6$
&
\renewcommand{\tabcolsep}{0pt}%
\begin{tabular}{l}
$(\pi_1; \pi_1)$, $(\pi_1; \pi'_1)$, $(\pi_2; \pi_2)$, \\
$(\pi_2; \pi_1 {+} \pi'_1)$, $(\pi_2; \pi'_2)$, $(\pi_2; 0)$
\end{tabular}
\\

\hline
\no \label{sp_spsp3}
&
\renewcommand{\tabcolsep}{0pt}%
\begin{tabular}{l}
$q {\ge} 2$, \\
$I {=} \lbrace 3 \rbrace$
\end{tabular}
&
$7 {-} \delta_p^2 {-} \delta_q^2$
&
\renewcommand{\tabcolsep}{0pt}%
\begin{tabular}{l}
$(\pi_3; \pi_3)$ ($p {\ge} 3$), $(\pi_3; \pi_1 {+} \pi'_2)$, \\
$(\pi_3; \pi_1)$, $(\pi_3; \pi'_1)$, $(\pi_3; \pi_2 {+} \pi'_1)$, \\
$(\pi_3; \pi'_3)$ ($q {\ge} 3$), $(2\pi_3; \pi_2 {+} \pi'_2)$
\end{tabular}
\\

\hline
\no \label{sp_spspn}
&
\renewcommand{\tabcolsep}{0pt}%
\begin{tabular}{l}
$q {\ge} 2$, \\
$I {=} \lbrace n \rbrace$
\end{tabular}
&
$q{+}1$
&
$(\pi_n; \pi_{p-k} {+} \pi'_{q-k})$ for $0 {\le} k {\le} q$
\\

\hline
\no \label{sp_spspi}
&
\renewcommand{\tabcolsep}{0pt}%
\begin{tabular}{l}
$q {=} 2$, \\
$I {=} \lbrace i \rbrace$, \\
$4 {\le} i {\le} n{-}1$
\end{tabular}
&
$7 {-} \delta_i^{n-1}$
&
\renewcommand{\tabcolsep}{0pt}%
\begin{tabular}{l}
$(\pi_i; \pi_i)$ ($i {\le} n{-}2$), $(\pi_i; \pi_{i-1} {+} \pi'_1)$, \\
$(\pi_i; \pi_{i-2} {+} \pi'_2)$, $(\pi_i; \pi_{i-2})$,
$(\pi_i; \pi_{i-3} {+} \pi'_1)$, \\
$(\pi_i; \pi_{i-4})$,
$(2\pi_i; \pi_{i-1} {+} \pi_{i-3} {+} \pi'_2)$
\end{tabular}
\\

\hline
\newcase \label{spin_odd}
& \multicolumn{3}{|c|}{$G = \Spin_{2n+1}$, $H = \Spin_p \cdot \Spin_q$, $p + q = 2n + 1 \ge 7$, $p$ is even, $q$ is odd}\\

\hline
\no \label{spin_odd_q=1}
&
\renewcommand{\tabcolsep}{0pt}%
\begin{tabular}{l}
$q {=} 1$, \\
$I {=} S$
\end{tabular}
&
$2n$
&
\renewcommand{\tabcolsep}{0pt}%
\begin{tabular}{l}
$(\pi_k; \pi_{k-1})$ for $1 {\le} k {\le} n$,\\
$(\pi_k; \pi_k)$ for $1 {\le} k {\le} n{-}2$,\\
$(\pi_{n-1}; \pi_{n-1} {+} \pi_n)$, $(\pi_n; \pi_n)$
\end{tabular}
\\

\hline
\no \label{spin_odd_I=1}
&
\renewcommand{\tabcolsep}{0pt}%
\begin{tabular}{l}
$p,q {\ge} 3$, \\
$I {=} \lbrace 1 \rbrace$
\end{tabular}
&
$3$
&
$(\pi_1; \pi_1), (\pi_1; \pi'_1), (2\pi_1; 0)$
\\

\hline
\no \label{spin_odd_n}
&
\renewcommand{\tabcolsep}{0pt}%
\begin{tabular}{l}
$p,q {\ge} 3$, \\
$I {=} \lbrace n \rbrace$
\end{tabular}
&
$\min(p,q) {+} 1$
&
\renewcommand{\tabcolsep}{0pt}%
\begin{tabular}{l}
$(\pi_n; \pi_{\frac p2-1} {+} \pi'_{[\frac{q}2]})$,
$(\pi_n; \pi_{\frac p2} {+} \pi'_{[\frac{q}2]})$,\\
$(2\pi_n; \pi_{\frac p2-1} {+} \pi_{\frac p2} {+} \pi'_{[\frac{q}2]-1}),$\\
$(2\pi_n; \pi_{\frac p2 -k} {+} \pi'_{[\frac{q}2]-k})$ for $2 {\le} k {\le} [\frac{\min(p,q)}2]$,\\
$(2\pi_n;\pi_{\frac p2 - k} {+} \pi'_{[\frac{q}2]+1-k})$ for $2 {\le} k {\le} [\frac{\min(p,q)+1}2]$
\end{tabular}
\\

\hline
\newcase \label{spin_even}
& \multicolumn{3}{|c|}{$G = \Spin_{2n}$, $H = \Spin_p \cdot \Spin_q$, $p \ge q \ge 1$, $p + q = 2n \ge 8$}\\

\hline
\no \label{spin_even_q=1}
&
\renewcommand{\tabcolsep}{0pt}%
\begin{tabular}{l}
$q {=} 1$, \\
$I {=} S$
\end{tabular}
&
$2n{-}1$
&
\renewcommand{\tabcolsep}{0pt}%
\begin{tabular}{l}
$(\pi_k; \pi_k)$ for $1 {\le} k {\le} n{-}1$,\\
$(\pi_k; \pi_{k-1})$ for $1 {\le} k {\le} n{-}2$,\\
$(\pi_n; \pi_{n-1})$, $(\pi_{n-1} {+} \pi_n; \pi_{n-2})$
\end{tabular}
\\

\hline
\no \label{spin_even_I=1}
&
\renewcommand{\tabcolsep}{0pt}%
\begin{tabular}{l}
$q {\ge} 3$, \\
$I {=} \lbrace 1 \rbrace$
\end{tabular}
&
$3$
&
$(\pi_1; \pi_1)$, $(\pi_1; \pi'_1)$, $(2\pi_1; 0)$
\\

\hline
\no \label{spin_even_n_pq_odd}
&
\renewcommand{\tabcolsep}{0pt}%
\begin{tabular}{l}
$q {\ge} 3$, \\
$I {=} \lbrace n \rbrace$,\\
$p,q$ are odd
\end{tabular}
&
$[\frac{q}2] {+} 1$
&
\renewcommand{\tabcolsep}{0pt}%
\begin{tabular}{l}
$(\pi_{n}; \pi_{\frac{p-1}2} {+} \pi'_{\frac{q-1}2})$, \\
$(2\pi_{n}; \pi_{\frac{p-q}2+k} {+} \pi'_k)$ for $0 {\le} k {\le} [\frac{q}2]{-}1$
\end{tabular}
\\

\hline
\no \label{spin_even_n_pq_even}
&
\renewcommand{\tabcolsep}{0pt}%
\begin{tabular}{l}
$q {\ge} 4$, \\
$I {=} \lbrace n \rbrace$,\\
$p,q$ are even
\end{tabular}
&
$\frac{q}2{+}1$
&
\renewcommand{\tabcolsep}{0pt}%
\begin{tabular}{l}
$(\pi_{n}; \pi_{\frac{p}2-1} {+} \pi'_{\frac{q}2-1})$, $(\pi_{n}; \pi_{\frac{p}2} {+} \pi'_{\frac{q}2})$,\\
$(2\pi_{n}; \pi_{\frac{p-q}2+k} {+} \pi'_k)$ for $0 {\le} k {\le} \frac q2 {-} 2$
\end{tabular}
\\

\hline
\newcase \label{f4_b4}
& \multicolumn{3}{|c|}{$G = \mathsf F_4$, $H = \mathsf B_4$}\\

\hline
\no \label{f4_b4_14}
&
$I {=} \lbrace 1,4 \rbrace$
&
$6$
&
\renewcommand{\tabcolsep}{0pt}%
\begin{tabular}{l}
$(\pi_1; \pi_2)$, $(\pi_1; \pi_4)$, $(\pi_4; \pi_1)$, \\
$(\pi_4; \pi_4)$, $(\pi_4; 0)$, $(\pi_1 {+} \pi_4; \pi_3)$
\end{tabular}
\\

\hline
\no \label{f4_b4_2}
&
$I {=} \lbrace 2 \rbrace$
&
$5$
&
\renewcommand{\tabcolsep}{0pt}%
\begin{tabular}{l}
$(\pi_2; \pi_2)$, $(\pi_2; \pi_3)$, $(\pi_2; \pi_1 {+} \pi_3)$,\\
$(\pi_2; \pi_1 {+} \pi_4)$, $(\pi_2; \pi_2 {+} \pi_4)$
\end{tabular}
\\

\hline
\no \label{f4_b4_3}
&
$I {=} \lbrace 3 \rbrace$
&
$5$
&
\renewcommand{\tabcolsep}{0pt}%
\begin{tabular}{l}
$(\pi_3; \pi_1)$, $(\pi_3; \pi_2)$, $(\pi_3; \pi_3)$,\\
$(\pi_3; \pi_4)$, $(\pi_3; \pi_1 {+} \pi_4)$
\end{tabular}
\\

\hline
\newcase  \label{e6_c4}
& \multicolumn{3}{|c|}{$G = \mathsf E_6$, $H = \mathsf C_4$}\\

\hline
\no \label{e6_c4_1}
&
$I {=} \lbrace 1 \rbrace$
&
$3$
&
$(\pi_1; \pi_2)$, $(2\pi_1; \pi_4)$, $(2\pi_1; 0)$
\\

\hline
\newcase \label{e6_a5xa1}
& \multicolumn{3}{|c|}{$G = \mathsf E_6$, $H = \mathsf A_5 \cdot \mathsf A_1$}\\

\hline
\no \label{e6_a5xa1_1}
&
$I {=} \lbrace 1 \rbrace$
&
$3$
&
$(\pi_1; \pi_4)$, $(\pi_1; \pi_1 {+} \pi'_1)$, $(2\pi_1; \pi_2)$
\\

\hline
\newcase \label{e6_f4}
& \multicolumn{3}{|c|}{$G = \mathsf E_6$, $H = \mathsf F_4$}\\

\hline
\no \label{e6_f4_12}
&
$I {=} \lbrace 1,2 \rbrace$
&
$5$
&
\renewcommand{\tabcolsep}{0pt}%
\begin{tabular}{l}
$(\pi_1; \pi_4)$, $(\pi_1; 0)$, $(\pi_2; \pi_1)$,
$(\pi_2; \pi_4)$, \\
$(\pi_1 {+} \pi_2; \pi_3)$
\end{tabular}
\\

\hline
\no \label{e6_f4_13}
&
$I {=} \lbrace 1, 3 \rbrace$
&
$5$
&
$(\pi_1; \pi_4)$, $(\pi_1; 0)$, $(\pi_3; \pi_1)$, $(\pi_3; \pi_3)$, $(\pi_3; \pi_4)$
\\

\hline
\newcase \label{e7_a7}
& \multicolumn{3}{|c|}{$G = \mathsf E_7$, $H = \mathsf A_7$}\\

\hline
\no \label{e7_a7_7}
&
$I {=} \lbrace 7 \rbrace$
&
$4$
&
$(\pi_7; \pi_2)$, $(\pi_7; \pi_6)$, $(2\pi_7; \pi_4)$, $(2\pi_7; 0)$
\\

\hline
\newcase \label{e7_d6xa1}
& \multicolumn{3}{|c|}{$G = \mathsf E_7$, $H = \mathsf D_6 \cdot \mathsf A_1$}\\

\hline
\no \label{e7_d6xa1_7}
&
$I {=} \lbrace 7 \rbrace$
&
$3$
&
$(\pi_7; \pi_6)$, $(\pi_7; \pi_1 {+} \pi'_1)$, $(2\pi_7; \pi_2)$
\\

\hline

\end{longtable}


\begin{longtable}{|c|l|l|l|}

\caption{The non-Levi non-symmetric subgroup case for $G = \SL_n$}
\label{table_sl}
\\
\hline
No. & Conditions & Rank & Indecomposable elements of~$\Gamma_I(G,H)$ \endfirsthead

\caption{The non-Levi non-symmetric subgroup case for $G = \SL_n$ (continued)} \\
\hline
No. & Conditions & Rank & Indecomposable elements of~$\Gamma_I(G,H)$ \endhead

\hline
\newcase \label{sl_spsl}
& \multicolumn{3}{|c|}{$G = \SL_{n}$, $H' = \Sp_{2p} \times \SL_q$, $2p+q=n$, $p \ge 2$, $q \ge 1$}\\

\hline
\no \label{sl_spsl_12}
&
$I {=} \lbrace 1,2 \rbrace$
&
$6 {-} \delta_q^1$
&
\renewcommand{\tabcolsep}{0pt}%
\begin{tabular}{l}
$(\pi_1; \pi_1 {+} \chi_1)$, $(\pi_1; \pi_1' {+} \chi_2)$, 
$(\pi_2; \pi_2 {+} 2\chi_1)$, $(\pi_2; 2\chi_1)$,\\
$(\pi_2; \pi_1 {+} \pi_1' {+} \chi_1 {+} \chi_2)$, $(\pi_2; \pi_2' {+} 2\chi_2)$ ($q {\ge} 2$)
\end{tabular}
\\

\hline
\no \label{sl_spsl_1n-1}
&
$I {=} \lbrace 1, n{-}1 \rbrace$
&
$6 {-} \delta_q^1$
&
\renewcommand{\tabcolsep}{0pt}%
\begin{tabular}{l}
$(\pi_1; \pi_1 {+} \chi_1)$, $(\pi_1; \pi_1' {+} \chi_2)$, 
$(\pi_{n-1}; \pi_{q-1}' {+} 2p\chi_1 {+} (q{-}1)\chi_2)$, \\ $(\pi_{n-1}; \pi_1 {+} (2p{-}1)\chi_1 {+} q\chi_2)$, 
$(\pi_1 {+} \pi_{n-1}; \pi_2 {+} 2p\chi_1 {+} q \chi_2)$, \\
$(\pi_1 {+} \pi_{n-1}; 2p\chi_1 {+} q\chi_2)$ ($q {\ge} 2$)
\end{tabular}
\\

\hline
\no \label{sl_spsl_i_q=1}
&
\renewcommand{\tabcolsep}{0pt}%
\begin{tabular}{l}
$q {=} 1$,\\
$I {=} \lbrace i \rbrace$,\\
$3 {\le} i {\le} p$
\end{tabular}
&
$1 + i$
&
\renewcommand{\tabcolsep}{0pt}%
\begin{tabular}{l}
$(\pi_i; \pi_{i-2k} {+} i\chi_1)$ 
for $0 \le k \le [\frac i2]$, \\
$(\pi_i; \pi_{i-1-2k} {+} (i{-}1)\chi_1 {+} \chi_2)$ for $0 {\le} k {\le} [\frac{i-1}2]$
\end{tabular}
\\

\hline
\no \label{sl_spsl_3}
&
\renewcommand{\tabcolsep}{0pt}%
\begin{tabular}{l}
$q {\ge} 2$,\\
$I {=} \lbrace 3 \rbrace$
\end{tabular}
&
$7 {-} \delta_p^2 {-} \delta_q^2$
&
\renewcommand{\tabcolsep}{0pt}%
\begin{tabular}{l}
$(\pi_3; \pi_1 {+} 3\chi_1)$, $(\pi_3; \pi_3 {+} 3\chi_1)$ ($p \ge 3$),\\
$(\pi_3; \pi_2 {+} \pi_1' {+} 2\chi_1 {+} \chi_2)$, $(\pi_3; \pi_1' {+} 2\chi_1 {+} \chi_2)$,\\
$(\pi_3; \pi_1 {+} \pi_2' {+} \chi_1 {+} 2\chi_2)$, $(\pi_3; \pi_3' {+} 3\chi_2)$ ($q {\ge} 3$),\\
$(2\pi_3; \pi_2 {+} \pi_2' {+} 4\chi_1 {+} 2\chi_2)$
\end{tabular}
\\

\hline
\no \label{sl_spsl_i_p=2}
&
\renewcommand{\tabcolsep}{0pt}%
\begin{tabular}{l}
$p {=} 2$,\\
$I {=} \lbrace i \rbrace$,\\
$4 {\le} i {\le} q$
\end{tabular}
&
$7$
&
\renewcommand{\tabcolsep}{0pt}%
\begin{tabular}{l}
$(\pi_i; \pi'_{i-4} {+} 4\chi_1 {+} (i{-}4)\chi_2)$, 
$(\pi_i; \pi_1 {+} \pi'_{i-3} {+} 3\chi_1 {+} (i{-}3)\chi_2)$, \\
$(\pi_i; \pi_2 {+} \pi'_{i-2} {+} 2\chi_1 {+} (i{-}2)\chi_2)$, 
$(\pi_i; \pi'_{i-2} {+} 2\chi_1 {+} (i{-}2)\chi_2)$,\\
$(\pi_i; \pi_1 {+} \pi'_{i-1} {+} \chi_1 {+} (i{-}1)\chi_2)$, 
$(\pi_i; \pi'_i {+} i\chi_2)$,\\
$(2\pi_i; \pi_2 {+} \pi'_{i-1} {+} \pi'_{i-3} {+} 4\chi_1 {+} (2i{-}4)\chi_2)$
\end{tabular}
\\

\hline
\newcase \label{sl_spsp}
& \multicolumn{3}{|c|}{$G = \SL_{n}$, $H' = \Sp_{2p} \times \Sp_{2q}$, $2p+2q=n$, $p \ge q \ge 2$}\\

\hline
\no \label{sl_spsp_12}
&
$I {=} \lbrace 1, 2 \rbrace$
&
$7$
&
\renewcommand{\tabcolsep}{0pt}%
\begin{tabular}{l}
$(\pi_1; \pi_1 {+} \chi_1)$, $(\pi_1; \pi_1' {+} \chi_2)$, 
$(\pi_2; \pi_2 {+} 2\chi_1)$, 
$(\pi_2; 2\chi_1)$, \\
$(\pi_2; \pi_1 {+} \pi_1' {+} \chi_1 {+} \chi_2)$, 
$(\pi_2; \pi_2' {+} 2\chi_2)$, $(\pi_2; 2\chi_2)$
\end{tabular}
\\

\hline
\no \label{sl_spsp_1n-1}
&
$I {=} \lbrace 1, n{-}1 \rbrace$
&
$7$
&
\renewcommand{\tabcolsep}{0pt}%
\begin{tabular}{l}
$(\pi_1; \pi_1 {+} \chi_1)$, $(\pi_1; \pi_1' {+} \chi_2)$, 
$(\pi_{n-1}; \pi'_1 {-} \chi_2)$, $(\pi_{n-1}; \pi_1 {-} \chi_1)$, \\
$(\pi_1 {+} \pi_{n-1}; \pi_2 {+} 2p\chi_1 {+} 2q\chi_2)$, 
$(\pi_1 {+} \pi_{n-1}; \pi'_2 {+} 2p\chi_1 {+} 2q\chi_2)$, \\
$(\pi_1 {+} \pi_{n-1}; 2p\chi_1 {+} 2q\chi_2)$
\end{tabular}
\\

\hline
\newcase \label{sl_spslsl}
& \multicolumn{3}{|c|}{$G = \SL_{n}$, $H' = \Sp_{2p} \times \SL_q \times \SL_r$, $2p+q+r=n$, $p \ge 2$, $q \ge r \ge 1$}\\

\hline
\no \label{sl_spslsl_2}
&
$I {=} \lbrace 2 \rbrace$
&
$7 {-} \delta_q^1 {-} \delta_r^1$
&
\renewcommand{\tabcolsep}{0pt}%
\begin{tabular}{l}
$(\pi_2; \pi_2 {+} 2\chi_1)$, $(\pi_2; 2\chi_1)$, 
$(\pi_2; \pi'_2 {+} 2\chi_2)$ ($q {\ge} 2$), \\
$(\pi_2; \pi''_2 {+} 2\chi_3)$ ($r {\ge} 2$), 
$(\pi_2; \pi_1 {+} \pi'_1 {+} \chi_1 {+} \chi_2)$, \\
$(\pi_2; \pi_1 {+} \pi''_1 {+} \chi_1 {+} \chi_3)$, 
$(\pi_2; \pi'_1 {+} \pi''_1 {+} \chi_2 {+} \chi_3)$
\end{tabular}
\\

\hline
\newcase \label{sl_spspsl}
& \multicolumn{3}{|c|}{$G = \SL_{n}$, $H' = \Sp_{2p} \times \Sp_{2q} \times \SL_r$, $2p+2q+r=n$, $p \ge q \ge 2$, $r \ge 1$}\\

\hline
\no \label{sl_spspsl_2}
&
$I {=} \lbrace 2 \rbrace$
&
$8 {-} \delta_r^1$
&
\renewcommand{\tabcolsep}{0pt}%
\begin{tabular}{l}
$(\pi_2; \pi_2 {+} 2\chi_1)$, $(\pi_2; 2\chi_1)$, $(\pi_2; \pi'_2 {+} 2\chi_2)$, $(\pi_2; 2\chi_2)$, \\
$(\pi_2; \pi''_2 {+} 2\chi_3)$ ($r {\ge} 2$), $(\pi_2; \pi_1 {+} \pi'_1 {+} \chi_1 {+} \chi_2)$, \\$(\pi_2; \pi_1 {+} \pi''_1 {+} \chi_1 {+} \chi_3)$,
$(\pi_2; \pi'_1 {+} \pi''_1 {+} \chi_2 {+} \chi_3)$
\end{tabular}
\\

\hline
\newcase \label{sl_spspsp}
& \multicolumn{3}{|c|}{$G = \SL_{n}$, $H' = \Sp_{2p} \times \Sp_{2q} \times \Sp_{2r}$, $2p+2q+2r=n$, $p \ge q \ge r \ge 2$}\\

\hline
\no \label{sl_spspsp_2}
&
$I {=} \lbrace 2 \rbrace$
&
$9$
&
\renewcommand{\tabcolsep}{0pt}%
\begin{tabular}{l}
$(\pi_2; \pi_2 {+} 2\chi_1)$, $(\pi_2; 2\chi_1)$, 
$(\pi_2; \pi'_2 {+} 2\chi_2)$, $(\pi_2; 2\chi_2)$, \\
$(\pi_2; \pi''_2 {+} 2\chi_3)$, $(\pi_2; 2\chi_3)$, 
$(\pi_2; \pi_1 {+} \pi'_1 {+} \chi_1 {+} \chi_2)$, \\
$(\pi_2; \pi_1 {+} \pi''_1 {+} \chi_1 {+} \chi_3)$, 
$(\pi_2; \pi'_1 {+} \pi''_1 {+} \chi_2 {+} \chi_3)$
\end{tabular}
\\

\hline
\newcase \label{sl_spin}
& \multicolumn{3}{|c|}{$G = \SL_8$, $H = \Spin_7$}\\

\hline
\no \label{sl_spin_2}
&
$I {=} \lbrace 2 \rbrace$
&
$4$
&
$(\pi_2; \pi_1)$, $(\pi_2; \pi_2)$, $(2\pi_2; 2\pi_3)$, $(2\pi_2; 0)$
\\

\hline

\end{longtable}

\section{The monoids \texorpdfstring{$\Gamma_I(G,H)$ in case~(\ref{case_C2})}%
{Gamma\_I(G,H) in case (C2)}}
\label{sect_proofs_C2}

In this section, we compute the monoid $\Gamma_I(G,H)$ for each of the cases in Table~\ref{table_sym}.

\subsection{Preliminary remarks}

Throughout this section, the numbers of cases refer to Table~\ref{table_sym} unless otherwise specified. All the notation and conventions for that table are used without extra explanation.

Recall that in case~(\ref{case_C2}) $H$ is the subgroup of fixed points of a nontrivial involutive automorphism $\theta$ of~$G$. By abuse of notation, we also denote by $\theta$ the corresponding involutive automorphism of the Lie algebra~$\mathfrak g$, so that $\mathfrak h = \lbrace x \in \mathfrak g \mid \theta(x) = x \rbrace$.

By \cite[Theorem~7.5]{Stei} it is possible to choose the subgroups $B_G$ and $T_G$ to be $\theta$-stable. In this case, the subgroup $B_G^-$ is also $\theta$-stable, $B_G \cap H$ and $B_G^- \cap H$ are Borel subgroups of~$H$, and $T_G \cap H$ is a maximal torus of~$H$.

Choose a subset $I \subset S$ and consider the parabolic subgroup~$P_I^- \supset B_G^-$.
Put $Q = P_I^- \cap H$, this is a parabolic subgroup of~$H$. Then the intersection $L_I \cap H$ is a Levi subgroup of~$Q$ (see, for instance, \cite[Proposition~3.5(1)]{HNOO}), we denote it by~$M$.

Put $\mathfrak g^{-\theta} = \lbrace x \in \mathfrak g \mid \theta(x) = -x \rbrace$, so that $\mathfrak g = \mathfrak h \oplus \mathfrak g^{-\theta}$. It is easy to see that under the above conditions there is an $M$-module isomorphism $\mathfrak g / (\mathfrak p_I^- + \mathfrak h) \simeq \mathfrak p_I^u \cap \mathfrak g^{-\theta}$. Combining this with Corollary~\ref{crl_rank_of_Gamma_refined} we obtain

\begin{proposition} \label{prop_rank_for_sym}
Under the above notation, suppose that $X_I$ is an $H$-spherical variety. Then $\mathfrak p_I^u \cap \mathfrak g^{-\theta}$ is a spherical $M$-module and $\rk \Gamma_I(G,H) = |I| + \rk_{M} (\mathfrak p_I^u \cap \mathfrak g^{-\theta})$.
\end{proposition}

For each of the cases considered in this section, the rank of $\Gamma_I(G,H)$ is computed using the formula in Proposition~\ref{prop_rank_for_sym}.
To this end, in each case we make use of the information on a particular embedding of $H$ into~$G$, the corresponding group $M$, and the spherical $M$-module $\mathfrak p_I^u \cap \mathfrak g^{-\theta}$ given in~\cite[\S\,5]{HNOO}; the value $\rk_{M} (\mathfrak p_I^u \cap \mathfrak g^{-\theta})$ is computed as described in~\S\,\ref{subsec_spherical_modules}. For each case, all conclusions on indecomposable elements of $\Gamma_I(G,H)$ are obtained by applying Propositions~\ref{prop_indec_I} and~\ref{prop_indec_II}.

For Cases~\ref{f4_b4}--\ref{e7_d6xa1}, involving the exceptional simple groups, all the information on restrictions of simple $G$-modules to~$H$ is obtained by using the program LiE~\cite{LiE}. Specifically, for a given pair $(G,H)$, the restriction to~$H$ of an irreducible representation $R_G(\lambda)$ is computed using the function~\textit{branch}. One of its arguments is the restriction matrix, which describes the character restriction map $\mathfrak X(T_G) \to \mathfrak X(T_H)$ under the assumption $T_H \subset T_G$. When $H$ is semisimple (which happens in all our cases), for the $i$th fundamental weight of~$G$ the coefficients in the expression of its restriction to $T_H$ in the basis of fundamental weights of~$H$ constitute the $i$th row of the restriction matrix. For each of the cases, the restriction matrix is easily computed using the information in~\cite[\S\,5]{HNOO}, we present this matrix at the beginning of each case.

Below we state two auxiliary results that are used in several cases.

\begin{proposition} \label{prop_equal_rank_res}
Let $F$ be a connected reductive group and let $K \subset F$ be a connected reductive subgroup such that $\rk K = \rk F$. Suppose that $V_1, V_2$ are two finite-dimensional $F$-modules such that $\left. V_1 \right|_K \simeq \left. V_2 \right|_K$. Then $V_1 \simeq V_2$.
\end{proposition}

\begin{proof}
As $\rk K = \rk F$, every maximal torus of $K$ is a maximal torus of~$F$. It remains to apply the well-known fact that every finite-dimensional representation of a connected reductive algebraic group is uniquely determined by its restriction to a maximal torus.
\end{proof}

The following proposition is well known and follows, for instance, from \cite[\S\,129, Theorems~2 and~3]{Zh} or \cite[Theorems~8.1.3 and~8.1.4]{GW}.

\begin{proposition} \label{prop_spin_res}
Let $G = \Spin_m$ with $m \ge 4$ and $H = \Spin_{m-1} \subset G$.
\begin{enumerate}[label=\textup{(\alph*)},ref=\textup{\alph*}]
\item
If $m = 2n+1$ then $\left. R_G(\pi_n) \right|_H \simeq R_H(\pi_{n-1}) \oplus R_H (\pi_n)$.

\item
If $m = 2n$ then $\left. R_G(\pi_n) \right|_H \simeq \left. R_G(\pi_{n-1}) \right|_H \simeq R_H(\pi_{n-1})$.
\end{enumerate}
\end{proposition}

\subsection{Case~\ref{sl_so}}
\label{subsec_sl_so}

$G = \SL_n$, $H = \SO_n$, $n \ge 3$. Using~\cite[Table~5, type~$\mathsf A_l$, no.~1]{OV} we find that
\begin{equation} \label{eqn_wedge_products_sl}
R_G(\pi_i) \simeq \wedge^i R_G(\pi_1)
\end{equation}
for all $i = 1, \ldots, n-1$.
Formulas~\cite[Table~5, type~$\mathsf A_l$, no.~3]{OV} imply
\begin{equation} \label{eqn_oplus_otimes_sl}
R_G(2\pi_i)\oplus R_G(\pi_{i-1})\otimes R_G(\pi_{i+1})=R_G(\pi_i)\otimes R_G(\pi_i)
\end{equation}
for all $i = 1, \ldots, n-1$.

\subsubsection{Case~\textup{\ref{sl_so_odd_i}}}
$n = 2l+1$, $I = \lbrace i \rbrace$, $i \le l$. The pair $(M', \mathfrak p_I^u \cap \mathfrak g^{-\theta})$ is equivalent to $(\SL_{i}, \operatorname{S}^2 \FF^{i})$, hence $\rk \Gamma_I(G,H) = i + 1$.

Using (\ref{eqn_wedge_products_sl}) and~\cite[Table~5, types~$\mathsf B_l$, no.~1]{OV}, we find that
\[
\left.R_G(\pi_i) \right|_H \simeq
\begin{cases}
R_H(\pi_i) & \ \text{if} \ i \le l-1;\\
R_H(2\pi_l) & \ \text{if} \ i = l.
\end{cases}
\]
Whence all indecomposable elements of $\Gamma_I(G,H)$ of the form $(\pi_i; *)$.

Restricting formula~(\ref{eqn_oplus_otimes_sl}) to $H$ and using~\cite[Table~5, type~$\mathsf B_l$, no.~3]{OV} we find that
\[
\left.R_G(2\pi_i)\right|_H \simeq W_{i,i}(G,H) \oplus \bigoplus \limits_{0 \le k \le i-1} R_H(2\pi_k).
\]
This yields all indecomposable elements of $\Gamma_I(G,H)$ of the form $(2\pi_i; *)$.

\subsubsection{Case~\textup{\ref{sl_so_even_i}}}
$n = 2l$, $I = \lbrace i \rbrace$, $i \le l$. The pair $(M', \mathfrak p_I^u \cap \mathfrak g^{-\theta})$ is equivalent to $(\SL_{i}, \operatorname{S}^2 \FF^{i})$, hence $\rk \Gamma_I(G,H) = i + 1$.

Using (\ref{eqn_wedge_products_sl}) and~\cite[Table~5, type~$\mathsf D_l$, no.~1]{OV}, we find that
\[
\left.R_G(\pi_i) \right|_H \simeq
\begin{cases}
R_H(\pi_i) & \ \text{if} \ i \le l-2;\\
R_H(\pi_{l-1}+\pi_l) & \ \text{if} \ i = l-1;\\
R_H(2\pi_{l-1}) \oplus R_H(2\pi_l) & \ \text{if} \ i = l.
\end{cases}
\]
Whence all indecomposable elements of $\Gamma_I(G,H)$ of the form $(\pi_i; *)$.

Restricting formula~(\ref{eqn_oplus_otimes_sl}) to $H$ and using~\cite[Table~5, type~$\mathsf D_l$, no.~3]{OV} we find that
\[
\left.R_G(2\pi_i)\right|_H \simeq W_{i,i}(G,H) \oplus \bigoplus \limits_{0 \le k \le k_0} R_H(2\pi_k)
\]
where $k_0 = \min(i-1,n-2)$. This yields all indecomposable elements of $\Gamma_I(G,H)$ of the form $(2\pi_i; *)$.

\subsection{Case~\ref{sl_sp}}
\label{SSspsl1k}

$G=\SL_{2n}$, $H=\Sp_{2n}$, $n \ge 2$. Given an integer $p$ such that $1 \le p \le 2n-1$, put $p' = \min(p, 2n-p)$. Then it follows from~(\ref{eqn_wedge_products_sl}) and~\cite[Table~5, type~$\mathsf C_l$, no.~1]{OV} that
\begin{equation} \label{SL_to_Sp_fund}
\left. R_G(\pi_p) \right|_H \simeq \bigoplus \limits_{0 \le k \le [\frac{p'}2]} R_H(\pi_{p'-2k}).
\end{equation}

By \cite[Table~5, type~$\mathsf A_l$, no.~4]{OV} we have
\begin{equation} \label{SL_aux_1}
R_G(\pi_1)\otimes R_G(\pi_p) \simeq R_G(\pi_1 + \pi_p) \oplus R_G(\pi_{p + 1}).
\end{equation}

Finally, the formula \cite[Table~5, type~$\mathsf C_l$, no.~4]{OV} implies that
\begin{equation} \label{Sp_aux}
R_H(\pi_1) \otimes R_H(\pi_{q}) \simeq R_H(\pi_1 + \pi_q) \oplus R_H(\pi_{q-1}) \oplus R_H(\pi_{q+1})
\end{equation}
for all $1 \le q \le n$ (the last summand is present for $q \le n-1$) .

\subsubsection{Case~\textup{\ref{sl_sp_1ipart1}}}
\label{sss_SL_to_Sp_1k}

$I = \lbrace 1,i \rbrace$, $2 \le i \le n$. The pair $(M', \mathfrak p_I^u \cap \mathfrak g^{-\theta})$ is equivalent to $(\SL_{i-1}, \wedge^2 \FF^{i-1} \oplus \FF^{i-1})$, hence $\rk \Gamma_I(G,H) = i + 1$.

Formula~(\ref{SL_to_Sp_fund}) yields all indecomposable elements of $\Gamma_I(G,H)$ of the form $(\pi_1; *)$ and of the form $(\pi_i; *)$. Next, from (\ref{SL_aux_1}), (\ref{SL_to_Sp_fund}), and (\ref{Sp_aux}) we find that
\[
\left. R_G(\pi_1 + \pi_i) \right|_H \simeq W_{1,i}(G,H) \oplus \bigoplus \limits_{0 \le k \le [\frac{i-3}2]} R_H(\pi_{i-1-2k}),
\]
which yields all indecomposable elements of $\Gamma_I(G,H)$ of the form $(\pi_1 + \pi_i; *)$.

\subsubsection{Case~\textup{\ref{sl_sp_1ipart2}}}

$I = \lbrace 1,i \rbrace$, $i \ge n + 1$.
The pair $(M', \mathfrak p_I^u \cap \mathfrak g^{-\theta})$ is equivalent to
\[
(\Sp_{2i-2n} \times \SL_{2n-i-1}, \FF^{2i-2n} \oplus (\wedge^2 \FF^{2n-i-1} \oplus \FF^{2n-i-1})),
\]
hence $\rk \Gamma_I(G,H) = 2n-i+2$.

Formula~(\ref{SL_to_Sp_fund}) yields all indecomposable elements of $\Gamma_I(G,H)$ of the form $(\pi_1; *)$ and of the form $(\pi_i; *)$. Next, from (\ref{SL_aux_1}), (\ref{SL_to_Sp_fund}), and (\ref{Sp_aux}) we find that
\[
\left. R_G(\pi_1 + \pi_i) \right|_H \simeq W_{1,i}(G,H) \oplus \bigoplus \limits_{0 \le k \le [\frac{2n-i-1}2]} R_H(\pi_{2n-i+1-2k})
\]
which yields all indecomposable elements of $\Gamma_I(G,H)$ of the form $(\pi_1 + \pi_i; *)$.

\subsubsection{Case~\textup{\ref{sl_sp_ii+1}}}
$I = \lbrace i, i + 1 \rbrace$, $i \le n - 1$.
The pair $(M', \mathfrak p_I^u \cap \mathfrak g^{-\theta})$ is equivalent to $(\SL_i, \wedge^2 \FF^i \oplus \FF^{i})$, hence $\rk \Gamma_I(G,H) = i + 2$.

Formula~(\ref{SL_to_Sp_fund}) yields all indecomposable elements of $\Gamma_I(G,H)$ of the form $(\pi_i; *)$ and of the form $(\pi_{i+1}; *)$, which already gives the required number of indecomposable elements.

\subsubsection{Case~\textup{\ref{sl_sp_123}}}
$I = \lbrace 1,2,3 \rbrace$. The pair $(M', \mathfrak p_I^u \cap \mathfrak g^{-\theta})$ is equivalent to
$(\lbrace e \rbrace, \FF^1 \oplus \FF^1 \oplus \FF^1)$ for $n \ge 3$ and to $(\lbrace e \rbrace, \FF^1 \oplus \FF^1)$ for $n=2$.
As a result, $\rk \Gamma_I(G,H) = 6 - \delta_n^2$.

Formula~(\ref{SL_to_Sp_fund}) yields all indecomposable elements of $\Gamma_I(G,H)$ of the form $(\pi_1; *)$, of the form $(\pi_2; *)$, and of the form $(\pi_3; *)$. From the information for Cases~\ref{sl_sp_1ipart1} and~\ref{sl_sp_1ipart2} we extract the last indecomposable element $(\pi_1 + \pi_3; \pi_2)$.

\subsubsection{Case~\textup{\ref{sl_sp_122n-1}}}
$n \ge 3$, $I = \lbrace 1,2,2n-1 \rbrace$. The pair $(M', \mathfrak p_I^u \cap \mathfrak g^{-\theta})$ is equivalent to $(\Sp_{2n-4}, \FF^{2n-4} \oplus \FF^1 \oplus \FF^1)$, hence $\rk \Gamma_I(G,H) = 6$.

Formula~(\ref{SL_to_Sp_fund}) yields all indecomposable elements of $\Gamma_I(G,H)$ of the form $(\pi_1; *)$, of the form $(\pi_2; *)$, and of the form $(\pi_{2n-1}; *)$. From the information for Case~\ref{sl_sp_1ipart2} with $I = \lbrace 1, 2n-1 \rbrace$, we obtain the indecomposable element $(\pi_1 + \pi_{2n-1}; \pi_2)$. Applying duality to Case~\ref{sl_sp_1ipart2} with $I = \lbrace 1, 2n-2 \rbrace$, we get the last indecomposable element $(\pi_2 + \pi_{2n-1}; \pi_3)$.

\subsection{Case~\ref{sp_spsp}}
\label{subsec_sp_spsp}

$G = \Sp_{2n}$, $H = \Sp_{2p} \times \Sp_{2q}$, $p \ge q \ge 1$, $p + q = n$.

\subsubsection{Case~\textup{\ref{sp_spspij}}}

$q = 1$, $I = \lbrace i, j \rbrace$, $i < j$.
The pair $(M', \mathfrak p_I^u \cap \mathfrak g^{-\theta})$ is equivalent to
\[
(\SL_{i-1} \times \SL_{j-i} \times \Sp_{2n-2j}, \FF^{i-1} \oplus (\FF^{j-i} \oplus (\FF^{j-i})^*) \oplus \FF^{2n-2j}),
\]
hence
$\rk \Gamma_I(G,H) = 7 - \delta_i^1 - \delta_j^{i+1} - \delta_n^j$.

In this case, the description of the set of indecomposable elements of $\Gamma_I(G,H)$ follows from well-known branching rules for the pair $(\Sp_{2n},\Sp_{2n-2} \times \Sp_2)$, we use them in the form of~\cite[Theorem~3.3]{WY}; see also other versions in \cite[Theorem~2]{Lep} or~\cite[Theorems~4 and~5]{Lee}.

In what follows, we shall need the following consequence of this case.

\begin{proposition} \label{prop_sp_spsp_i}
Suppose that $q = 1$ and $I = \lbrace i \rbrace$ with $1 \le i \le n$. Then $\rk \Gamma_I(G/H) = 3 - \delta_i^1 - \delta_i^n$ and the indecomposable elements of $\Gamma_I(G/H)$ are $(\pi_i; \pi_i)$ $(i \le n-1)$, $(\pi_i; \pi_{i-1} {+} \pi'_1)$, $(\pi_i; \pi_{i-2})$ $(i \ge 2)$.
\end{proposition}

\subsubsection{Case~\textup{\ref{sp_spsp12}}}

$q \ge 2$, $I = \lbrace 1, 2 \rbrace$. The pair $(M', \mathfrak p_I^u \cap \mathfrak g^{-\theta})$ is equivalent to $(\Sp_{2q}, \FF^{2q} \oplus \FF^{2q})$, hence $\rk \Gamma_I(G,H) = 6$.

Clearly, $\left.R_G(\pi_1)\right|_H \simeq R_H(\pi_1) \oplus R_H(\pi'_1)$, which yields all indecomposable elements of $\Gamma_I(G,H)$ of the form $(\pi_1; *)$. Next, we find the decomposition of $\left. R_G(\pi_2)\right|_H$ into irreducible components. To this end, consider the group $F = \SL_{2n}$. It is easy to see from Case~\ref{sl_sp_123} that
\begin{equation} \label{eqn_pi2_for_sp}
\left. R_F(\pi_2) \right|_G \simeq R_G(\pi_2) \oplus R_G(0).
\end{equation}
Now, we put $K = \SL_{2p} \times \SL_{2q}$ and restrict $R_F(\pi_2)$ to $H$ through the chain $F \supset K \supset H$. Using Proposition~\ref{prop_sl_Levi_pq_i} we get
\[
\left. R_F(\pi_2) \right|_{K} \simeq
R_{K}(\pi_2) \oplus R_{K}(\pi_1 + \pi'_1) \oplus R_{K}(\pi'_2).
\]
Then applying Case~\ref{sl_sp_123} and comparing with~(\ref{eqn_pi2_for_sp}), we finally obtain
\[
\left. R_G(\pi_2) \right|_H \simeq R_H(0) \oplus R_H(\pi_1 + \pi'_1) \oplus R_H(\pi_2) \oplus R_H(\pi'_2),
\]
which yields all indecomposable elements of $\Gamma_I(G,H)$ of the form $(\pi_2; *)$.

\subsubsection{Case~\textup{\ref{sp_spsp3}}}

$q \ge 2$, $I = \lbrace 3 \rbrace$. The pair $(M', \mathfrak p_I^u \cap \mathfrak g^{-\theta})$ is equivalent to $(\Sp_{2q} \times \SL_3$, $\FF^{2q} \otimes\nobreak \FF^3)$ for $p \ge 3$ and to $(\SL_2 \times \SL_2, \FF^2 \otimes (\FF^2 \oplus \FF^1))$ for $p = q = 2$, hence $\rk \Gamma_I(G,H) = 7 - \delta_p^2 - \delta_q^2$.

Consider the group $F = \SL_{2n}$. It is easy to see from Case~\ref{sl_sp_123} that
\begin{equation} \label{eqn_pi3_for_sp}
\left. R_F(\pi_3) \right|_G \simeq R_G(\pi_3) \oplus R_G(\pi_1).
\end{equation}

Next, we put $K = \SL_{2p} \times \SL_{2q}$ and restrict $R_F(\pi_3)$ to $H$ through the chain $F \supset K \supset H$. Using Proposition~\ref{prop_sl_Levi_pq_i} we get
\[
\left. R_F(\pi_3) \right|_{K} \simeq
R_{K}(\pi_3) \oplus R_{K}(\pi_2 + \pi'_1)  \oplus R_{K}(\pi_1 + \pi'_2)\oplus R_{K}(\pi'_3).
\]
Then applying Case~\ref{sl_sp_123}, decomposing $\left. R_G(\pi_1) \right|_H$ via Case~\ref{sp_spsp12}, and comparing with~(\ref{eqn_pi3_for_sp}) we obtain
\[
\left. R_G(\pi_3) \right|_H \simeq R_H(\pi_3) \oplus R_H(\pi_1) \oplus R_H(\pi_2 + \pi'_1) \oplus R_H(\pi_1 + \pi'_2) \oplus R_H(\pi'_1) \oplus R_H(\pi'_3)
\]
where the first summand is present for $p \ge 3$ and the last summand is present for $q \ge 3$. This yields all indecomposable elements of $\Gamma_I(G,H)$ of the form $(\pi_3; *)$.

Next, it follows again from Case~\ref{sl_sp_123} that
\begin{equation} \label{eqn_2pi3_for_sp}
\left. R_F(2\pi_3) \right|_G \oplus R_G(\pi_2) \simeq R_G(2\pi_3) \oplus \left. R_F(\pi_1+\pi_3) \right|_G.
\end{equation}
Restricting $R_F(2\pi_3)$ and $R_F(\pi_1 + \pi_3)$ to $H$ through the chain $F \supset K \supset H$ by using Case~\ref{sl_Levi_pq_1i_part1} of Table~\ref{table_Levi_sl} and Case~\ref{sl_sp_123}, computing $\left. R_G(\pi_2) \right|_H$ by using Case~\ref{sp_spsp12}, and comparing with~(\ref{eqn_2pi3_for_sp}) we finally obtain
\[
\left. R_G(2\pi_3) \right|_H \simeq W_{3,3}(G,H) \oplus R_H(\pi_2 + \pi'_2),
\]
whence $(2\pi_3, \pi_2 + \pi'_2)$ is the last indecomposable element of $\Gamma_I(G,H)$.

\subsubsection{Case~\textup{\ref{sp_spspn}}}
$q \ge 2$, $I = \lbrace n \rbrace$. The pair $(M', \mathfrak p_I^u \cap \mathfrak g^{-\theta})$ is equivalent to $(\SL_p \times \SL_q$, $\FF^p \otimes \FF^q)$, hence $\rk \Gamma_I(G,H) = q + 1$.

Observe from Proposition~\ref{prop_sp_spsp_i} that the formulas for $\rk \Gamma_I(G,H)$ and the indecomposable elements of $\Gamma_I(G,H)$ hold true for $q=1$. Then the required result is proved by induction on~$q$ as follows. Suppose that for a given $q$ the indecomposable elements of $\Gamma_I(G,H)$ have the required form and $p - q \ge 2$. In particular,
\begin{equation} \label{eqn_pi_n_res_H}
\left. R_G(\pi_n) \right|_H \simeq \bigoplus \limits_{0 \le k \le q} R_H(\pi_{p-k} + \pi'_{q-k}).
\end{equation}
Consider the groups $H_1 = \Sp_{2(p-1)} \times \Sp_{2(q+1)}$ and $K = \Sp_{2(p-1)} \times \Sp_2 \times \Sp_{2q}$. Using Proposition~\ref{prop_sp_spsp_i}, we compute the restrictions to $K$ of the right-hand side of~(\ref{eqn_pi_n_res_H}) and of the $H_1$-module
\[
V = \bigoplus \limits_{0 \le k \le q+1} R_H(\pi_{p-1-k} + \pi'_{q+1-k})
\]
and find that $\left. R_G(\pi_n) \right|_K \simeq \left. V \right|_K$. As $\rk H_1 = n = \rk K$, by Proposition~\ref{prop_equal_rank_res} the latter implies that $\left. R_G(\pi_n) \right|_{H_1} \simeq V$, which yields all the required indecomposable elements of $\Gamma_I(G,H_1)$.

\subsubsection{Case~\textup{\ref{sp_spspi}}}

$q = 2$, $I = \lbrace i \rbrace$, $4 \le i \le n-1$. The pair $(M', \mathfrak p_I^u \cap \mathfrak g^{-\theta})$ is equivalent to $(\SL_2 \times \SL_{i-2} \times \Sp_{2n-2i}, \FF^2 \otimes (\FF^{i-2} \oplus\FF^{2n-2i}))$, hence $\rk \Gamma_I(G,H) = 7 - \delta_i^{n-1}$.

Consider the $H$-module
\[
V = R_H(\pi_i) \oplus R_H(\pi_{i-1} + \pi'_1) \oplus R_H(\pi_{i-2} + \pi'_2) \oplus R_H(\pi_{i-2}) \oplus R_H(\pi_{i-3} + \pi'_1) \oplus R_H(\pi_{i-4})
\]
where the first summand is present for $i \le n-2$. Next, consider the group $K = \Sp_{2p} \times \Sp_2 \times \Sp_2 \subset H$. Using Proposition~\ref{prop_sp_spsp_i}, we compute $\left. V \right|_K$ and $\left. R_G(\pi_i) \right|_K$ through the chain $G \supset \Sp_{2n-2} \times \Sp_2 \supset K$ and find that $\left. V \right|_K \simeq \left. R_G(\pi_i) \right|_K$. As $\rk H = n = \rk K$, Proposition~\ref{prop_equal_rank_res} implies that $\left. R_G(\pi_i) \right|_H \simeq V$, which yields all indecomposable elements of $\Gamma_I(G,H)$ of the form $(\pi_i; *)$.

Now consider the $H$-module $V = W_{i,i}(G,H) \oplus R_H(\pi_{i-1} + \pi_{i-3} + \pi'_2)$. Again, using Proposition~\ref{prop_sp_spsp_i} we compute $\left. V \right|_K$ and $\left. R_G(2\pi_i) \right|_K$ through the chain $G \supset \Sp_{2n-2} \times \Sp_2 \supset K$ and find that $\left. V \right|_K \simeq \left. R_G(2\pi_i) \right|_K$. Again, Proposition~\ref{prop_equal_rank_res} implies that $\left. R_G(2\pi_i) \right|_H \simeq V$, whence $(2\pi_i; \pi_{i-1} + \pi_{i-3} + \pi'_2)$ is the last indecomposable element of $\Gamma_I(G,H)$.

\subsection{Case~\ref{spin_odd}}
\label{subsec_spin_odd}

$G = \Spin_{2n+1}$, $H = \Spin_p \times \Spin_q$, $p + q = 2n + 1 \ge 7$, $p$ is even, $q$ is odd.

\subsubsection{Case~\textup{\ref{spin_odd_q=1}}}
$q = 1$, $I = \Pi$. The pair $(M', \mathfrak p_I^u \cap \mathfrak g^{-\theta})$ is equivalent to $(\lbrace e \rbrace, \FF^n)$, hence $\rk \Gamma_I(G,H) = 2n$.

The indecomposable elements of $\Gamma_I(G,H)$ are taken from~\cite[Theorem~7, part~2]{AkP}.

\subsubsection{Case~\textup{\ref{spin_odd_I=1}}}
$p,q \ge 3$, $I = \lbrace 1 \rbrace$. The pair $(M', \mathfrak p_I^u \cap \mathfrak g^{-\theta})$ is equivalent to $(\SO_q, \FF^q)$, hence $\rk \Gamma_I(G,H) = 3$.

Clearly, $\left.R_G(\pi_1)\right|_H \simeq R_H(\pi_1) \oplus R_H(\pi'_1)$, which yields all indecomposable elements of $\Gamma_I(G,H)$ of the form $(\pi_1; *)$. To find the decomposition of $\left. R_G(2\pi_1)\right|_H$ into irreducible components, observe that the action of $G$ (resp.~$H$) on $R_G(2\pi_1)$ descends to an action of $\SO_n$ (resp. $\SO_p \times \SO_q$) and consider the group $F = \SL_{2n+1}$. It is easy to see from Case~\ref{sl_so_odd_i} that
\begin{equation} \label{eqn_2pi1}
\left. R_F(2\pi_1) \right|_G \simeq R_G(2\pi_1) \oplus R_G(0).
\end{equation}
Now, we put $K = \SL_p \times \SL_q$ and restrict $R_F(2\pi_1)$ to $H$ through the chain $F \supset K \supset \SO_p \times \SO_q$. Using Proposition~\ref{prop_sl_Levi_pq_i} we get
\[
\left. R_F(2\pi_1) \right|_{K} \simeq
R_{K}(2\pi_1) \oplus R_{K}(\pi_1 + \pi'_1) \oplus R_{K}(2\pi'_1).
\]
Then applying Cases~\ref{sl_so_odd_i} and~\ref{sl_so_even_i} and comparing with~(\ref{eqn_2pi1}), we finally obtain
\[
\left. R_G(2\pi_1) \right|_H \simeq W_{1,1}(G,H) \oplus R_H(0),
\]
whence $(2\pi_1; 0)$ is the last indecomposable element of $\Gamma_I(G,H)$.

\subsubsection{Case~\textup{\ref{spin_odd_n}}}
$p,q \ge 3$, $I = \lbrace n \rbrace$. The pair $(M', \mathfrak p_I^u \cap \mathfrak g^{-\theta})$ is equivalent to
\[
(\SL_{p/2} \times \SL_{(q-1)/2}, \FF^{p/2} \otimes (\FF^{(q-1)/2} \oplus \FF^1)),
\]
hence $\rk \Gamma_I(G,H) = \min(p,q) + 1$.

Restricting the representation $R_G(\pi_n)$ through the chains
\[
\Spin_{2n+1} \supset \Spin_{2n} \supset \ldots \supset \Spin_p \quad \text{and} \quad \Spin_{2n+1} \supset \Spin_{2n} \supset \ldots \supset \Spin_q
\]
by using Proposition~\ref{prop_spin_res}, we obtain that $\left.R_G(\pi_n)\right|_{\Spin_p}$ is isomorphic to the direct sum of $2^{\frac{q-1}2}$ copies of $R_{\Spin_p}(\pi_{\frac{p}2-1}) \oplus R_{\Spin_p}(\pi_{\frac{p}2})$
and $\left.R_G(\pi_n)\right|_{\Spin_q}$ is isomorphic to the direct sum of $2^{\frac{p}2}$ copies of $R_{\Spin_q}(\pi_{\frac{q-1}2})$. As $\dim (R_{\Spin_p}(\pi_{\frac{p}2-1}) \oplus R_{\Spin_p}(\pi_{\frac{p}2})) = 2^{\frac{p}2}$, $\dim R_{\Spin_q}(\pi_{\frac{q-1}2}) = 2^{\frac{q-1}2}$, and the subgroups $\Spin_p$ and $\Spin_q$ commute, it follows that
\[
\left.R_G(\pi_n)\right|_H \simeq R_H(\pi_{\frac p2-1} + \pi_{\frac{q-1}2}') \oplus R_H(\pi_{\frac p2} + \pi_{\frac{q-1}2}'),
\]
which yields all indecomposable elements of $\Gamma_I(G,H)$ of the form $(\pi_n; *)$. To find the decomposition of $\left. R_G(2\pi_n) \right|_H$ into irreducible summands, observe that the action of $G$ (resp.~$H$) on $R_G(2\pi_n)$ descends to an action of $\SO_n$ (resp. $\SO_p \times \SO_q$) and consider the group $F = \SL_{2n+1}$. It is easy to see from Case~\ref{sl_so_odd_i} that
\begin{equation} \label{eqn_pi_n}
\left. R_F(\pi_n) \right|_G \simeq R_G(2\pi_n).
\end{equation}
Now, we put $K = \SL_p \times \SL_q$ and restrict $R_F(\pi_n)$ to $H$ through the chain $F \supset K \supset \SO_p \times \SO_q$.

Suppose $p < q$. Then using Proposition~\ref{prop_sl_Levi_pq_i} we get
\[
\left. R_F(\pi_n) \right|_{K} \simeq
\bigoplus \limits_{0 \le k \le p} R_{K}(\pi_{k} + \pi'_{n-k}).
\]
Now applying Cases~\ref{sl_so_odd_i} and \ref{sl_so_even_i} and comparing with (\ref{eqn_pi_n}) we obtain
\begin{multline} \label{eqn_min=p}
\left. R_G(2\pi_n) \right|_H \simeq W_{n,n}(G,H) \oplus R_H(\pi_{\frac{p}2-1} + \pi_{\frac{p}2} + \pi'_{\frac{q-3}2}) \oplus \\
\bigoplus \limits_{2 \le k \le \frac{p}2} R_H(\pi_{\frac{p}2-k} + \pi'_{\frac{q-1}2-k}) \oplus
\bigoplus \limits_{2 \le k \le \frac{p}2} R_H(\pi_{\frac{p}2-k} + \pi'_{\frac{q+1}2-k}).
\end{multline}

Suppose $p > q$. Then using Proposition~\ref{prop_sl_Levi_pq_i} we get
\begin{equation} \label{eqn_pi_n_res_K}
\left. R_F(\pi_n) \right|_{K} \simeq
\bigoplus \limits_{0 \le k \le q} R_{K}(\pi_{n-k} + \pi'_k).
\end{equation}
Now applying Cases~\ref{sl_so_odd_i} and~\ref{sl_so_even_i} and comparing with (\ref{eqn_pi_n}) we obtain
\begin{multline} \label{eqn_min=q}
\left. R_G(2\pi_n) \right|_H \simeq W_{n,n}(G,H) \oplus
R_H(\pi_{\frac{p}2-1} + \pi_{\frac{p}2} + \pi'_{\frac{q-3}2}) \oplus \\
\bigoplus \limits_{2 \le k \le \frac{q-1}2} R_H(\pi_{\frac{p}2-k} + \pi'_{\frac{q-1}2-k}) \oplus
\bigoplus \limits_{2 \le k \le \frac{q+1}2} R_H(\pi_{\frac{p}2-k} + \pi'_{\frac{q+1}2-k}).
\end{multline}
From~(\ref{eqn_min=p}) and~(\ref{eqn_min=q}) we get all indecomposable elements of $\Gamma_I(G,H)$ of the form $(2\pi_n; *)$.

\subsection{Case~\ref{spin_even}}

$G = \Spin_{2n}$, $H = \Spin_p \times \Spin_q$, $p \ge q \ge 1$, $p + q = 2n \ge 8$.

\subsubsection{Case~\textup{\ref{spin_even_q=1}}}
$p = 1$, $I = \Pi$. The pair $(M', \mathfrak p_I^u \cap \mathfrak g^{-\theta})$ is equivalent to $(\lbrace e \rbrace, \FF^{n-1})$, whence $\rk \Gamma_I(G,H) = 2n-1$.

The indecomposable elements of $\Gamma_I(G,H)$ are taken from~\cite[Theorem~7, part~3]{AkP}.

\subsubsection{Case~\textup{\ref{spin_even_I=1}}}
$p \ge 3$, $I = \lbrace 1 \rbrace$. The pair $(M', \mathfrak p_I^u \cap \mathfrak g^{-\theta})$ is equivalent to $(\SO_q, \FF^q)$, hence $\rk \Gamma_I(G,H) = 3$.

The remaining argument repeats that for Case~\ref{spin_odd_I=1}.

\subsubsection{Case~\textup{\ref{spin_even_n_pq_odd}}} \label{subsubsec_spin_pq_odd}

$q \ge 3$, $I = \lbrace n \rbrace$, $p,q$ are odd. The pair $(M', \mathfrak p_I^u \cap \mathfrak g^{-\theta})$ is equivalent to $(\SL_{(p-1)/2} \times \SL_{(q-1)/2}, \FF^{(p-1)/2} \otimes \FF^{(q-1)/2})$, hence $\rk \Gamma_I(G,H) = \frac{q+1}2 = [\frac q2] + 1$.

Since the triple $(G,H,\lbrace n-1 \rbrace)$ is obtained from $(G,H, \lbrace n \rbrace)$ by an outer automorphism of $G$ interchanging the $(n-1)$th and $n$th simple roots, it follows that $H$ acts spherically on $X_{\lbrace n-1 \rbrace}$ as well, which will be used in this proof.

Restricting the representation $R_G(\pi_{n})$ through the chains
\[
\Spin_{2n} \supset \Spin_{2n-1} \supset \ldots \supset \Spin_p \quad \text{and} \quad \Spin_{2n} \supset \Spin_{2n-1} \supset \ldots \supset \Spin_q
\]
by using Proposition~\ref{prop_spin_res}, we find that $\left.R_G(\pi_{n})\right|_{\Spin_p}$ is isomorphic to the direct sum of $2^{\frac{q-1}2}$ copies of $R_{\Spin_p}(\pi_{\frac{p-1}2})$
and $\left.R_G(\pi_{n})\right|_{\Spin_q}$ is isomorphic to the direct sum of $2^{\frac{p-1}2}$ copies of $R_{\Spin_q}(\pi_{\frac{q-1}2})$. As $\dim R_{\Spin_p}(\pi_{\frac{p-1}2}) = 2^{\frac{p-1}2}$, $\dim R_{\Spin_q}(\pi_{\frac{q-1}2}) = 2^{\frac{q-1}2}$, and the subgroups $\Spin_p$ and $\Spin_q$ commute, it follows that
\[
\left.R_G(\pi_{n})\right|_H \simeq R_H(\pi_{\frac{p-1}2} + \pi_{\frac{q-1}2}'),
\]
which yields all indecomposable elements of $\Gamma_I(G,H)$ of the form $(\pi_{n}; *)$. To find the decomposition of $\left. R_G(2\pi_{n}) \right|_H$ into irreducible summands, observe that the actions of $G$ (resp.~$H$) on $R_G(2\pi_{n-1})$ and $R_G(2\pi_n)$ descend to actions of $\SO_n$ (resp. $\SO_p \times \SO_q$) and consider the group $F = \SL_{2n}$. It is easy to see from Case~\ref{sl_so_odd_i} that
\begin{equation} \label{eqn_pi_n_bis}
\left. R_F(\pi_{n}) \right|_G \simeq R_G(2\pi_{n-1}) \oplus R_G(2\pi_n).
\end{equation}
Now, we put $K = \SL_p \times \SL_q$ and restrict $R_F(\pi_n)$ to $H$ through the chain $F \supset K \supset \SO_p \times \SO_q$. Using Proposition~\ref{prop_sl_Levi_pq_i} we get~(\ref{eqn_pi_n_res_K}).
Now applying Case~\ref{sl_so_odd_i} and comparing with (\ref{eqn_pi_n_bis}) we obtain
\[
\left.(R_G(2\pi_{n-1}) \oplus R_G(2\pi_n))\right|_H \simeq 2R_H(2\pi_{\frac{p-1}2} + 2\pi'_{\frac{q-1}2}) \oplus 2
\bigoplus \limits_{0 \le k \le \frac{q-3}2} R_H(\pi_{\frac{p-q}2+k}+ \pi'_k).
\]
Since we know that both modules $\left. R_G(2\pi_{n-1})\right|_H$ and $\left. R_G(2\pi_{n})\right|_H$ are multiplicity-free, it follows that
\[
\left. R_G(2\pi_{n})\right|_H \simeq W_{n,n}(G,H) \oplus \bigoplus  \limits_{0 \le k \le \frac{q-3}2} R_H(\pi_{\frac{p-q}2+k}+ \pi'_k),
\]
which yields all indecomposable elements of $\Gamma_I(G,H)$ of the form $(2\pi_{n}; *)$.

\subsubsection{Case~\textup{\ref{spin_even_n_pq_even}}}

$q \ge 3$, $I = \lbrace n \rbrace$, $p,q$ are even. The pair $(M', \mathfrak p_I^u \cap \mathfrak g^{-\theta})$ is equivalent to $(\SL_{p/2} \times \SL_{q/2}, \FF^{p/2} \otimes \FF^{q/2})$, hence $\rk \Gamma_I(G,H) = \frac q2 + 1$.

Since the triple $(G,H,\lbrace n-1 \rbrace)$ is obtained from $(G,H, \lbrace n \rbrace)$ by an outer automorphism of $G$ interchanging the $(n-1)$th and $n$th simple roots, it follows that $H$ acts spherically on $X_{\lbrace n-1 \rbrace}$ as well, which will be used in this proof.

Restricting the representations $R_G(\pi_{n-1})$ and $R_G(\pi_n)$ through the chains
\[
\Spin_{2n} \supset \Spin_{2n-1} \supset \ldots \supset \Spin_p \quad \text{and} \quad \Spin_{2n} \supset \Spin_{2n-1} \supset \ldots \supset \Spin_q
\]
by using Proposition~\ref{prop_spin_res}, we obtain that $\left.R_G(\pi_{n-1})\right|_{\Spin_p} \simeq \left.R_G(\pi_n)\right|_{\Spin_p}$ is isomorphic to the direct sum of $2^{\frac{q}2-1}$ copies of $R_{\Spin_p}(\pi_{\frac{p}2-1}) \oplus\nobreak R_{\Spin_p}(\pi_{\frac{p}2})$
and $\left.R_G(\pi_{n-1})\right|_{\Spin_q} \simeq \left.R_G(\pi_n)\right|_{\Spin_q}$ is isomorphic to the direct sum of $2^{\frac{p}2-1}$ copies of $R_{\Spin_q}(\pi_{\frac{q}2-1}) \oplus R_{\Spin_q}(\pi_{\frac{q}2})$. As $\dim R_{\Spin_p}(\pi_{\frac{p}2-1}) = \dim R_{\Spin_p}(\pi_{\frac{p}2}) = 2^{\frac{p}2-1}$, $\dim R_{\Spin_q}(\pi_{\frac{q}2-1}) = 2^{\frac{q}2-1}$, and the subgroups $\Spin_p$ and $\Spin_q$ commute, it follows that each of $\left. R_G(\pi_{n-1}) \right|_H$ and $\left. R_G(\pi_n) \right|_H$ is isomorphic to one of the two modules
\[
V_1 = R_H(\pi_{\frac{p}2-1} + \pi'_{\frac{q}2-1}) \oplus R_H(\pi_{\frac{p}2} + \pi'_{\frac{q}2}) \ \, \text{or} \ \,
V_2 = R_H(\pi_{\frac{p}2-1} + \pi'_{\frac{q}2}) \oplus R_H(\pi_{\frac{p}2} + \pi'_{\frac{q}2-1})
\]
Changing the order (numbering) of the last two simple roots of $\Spin_p$ (or $\Spin_q$) if necessary, we may assume that $\left. R_G(\pi_{n}) \right|_H \simeq V_1$, which provides all indecomposable elements of $\Gamma_I(G,H)$ of the form $(\pi_{n}; *)$. As we shall see below, it then follows that $\left. R_G(\pi_{n-1}) \right|_H \simeq V_2$.

To find the decomposition of $\left. R_G(2\pi_{n-1}) \right|_H$ and $\left. R_G(2\pi_n) \right|_H$ into irreducible components, observe that the actions of $G$ (resp.~$H$) on $R_G(2\pi_{n-1})$ and $R_G(2\pi_n)$ descend to actions of $\SO_n$ (resp. $\SO_p \times \SO_q$) and consider the group $F = \SL_{2n}$. As in \S\,\ref{subsubsec_spin_pq_odd} we again have~(\ref{eqn_pi_n_bis}).

Now, we put $K = \SL_p \times \SL_q$ and restrict $R_F(\pi_n)$ to $H$ through the chain $F \supset K \supset \SO_p \times \SO_q$. Using Proposition~\ref{prop_sl_Levi_pq_i} we get~(\ref{eqn_pi_n_res_K}).
Now applying Case~\ref{sl_so_odd_i} and comparing with (\ref{eqn_pi_n_bis}) we obtain
\begin{multline} \label{eqn_spin_even}
\left.(R_G(2\pi_{n-1}) \oplus R_G(2\pi_n))\right|_H \simeq \\
R_H(2\pi_{\frac{p}2-1} + 2\pi'_{\frac{q}2-1}) \oplus
R_H(2\pi_{\frac{p}2-1} + 2\pi'_{\frac{q}2}) \oplus
R_H(2\pi_{\frac{p}2} + 2\pi'_{\frac{q}2-1}) \oplus
R_H(2\pi_{\frac{p}2} + 2\pi'_{\frac{q}2}) \oplus \\
2R_H(\pi_{\frac{p}2-1} + \pi_{\frac{p}2} + \pi'_{\frac{q}2-1} + \pi'_{\frac{q}2}) \oplus
2\bigoplus \limits_{0 \le k \le \frac{q}2-2}
R_H(\pi_{\frac{p-q}2+k} + \pi'_k).
\end{multline}
As the left-hand side of~(\ref{eqn_spin_even}) contains $W_{n-1,n-1}(G,H) \oplus W_{n,n}(G,H)$ as a submodule and the right-hand side contains only one simple submodule isomorphic to $R_H(2\pi_{\frac{p}2-1} + 2 \pi'_{\frac{q}2-1})$, the module $\left. R_G(\pi_{n-1}) \right|_H$ cannot be isomorphic to~$V_1$, therefore $\left. R_G(\pi_{n-1}) \right|_H \simeq V_2$.

Since we know that both modules $\left. R_G(2\pi_{n-1})\right|_H$ and $\left. R_G(2\pi_{n})\right|_H$ are multiplicity-free, it follows from~(\ref{eqn_spin_even}) that
\[
\left. R_G(2\pi_{n})\right|_H \simeq W_{n,n}(G,H) \oplus \bigoplus \limits_{0 \le k \le \frac{q}2-2} R_H(\pi_{\frac{p-q}2+k} + \pi'_k),
\]
which yields all indecomposable elements of $\Gamma_I(G,H)$ of the form $(2\pi_{n}; *)$.

\subsection{Case~\ref{f4_b4}}

$G = \mathsf F_4$, $H = \mathsf B_4$. The restriction matrix is
\[
\begin{pmatrix}
0 & 1 & 0 & 0\\
1 & 0 & 1 & 0\\
1 & 0 & 0 & 1\\
1 & 0 & 0 & 0\\
\end{pmatrix}.
\]

\subsubsection{Case~\textup{\ref{f4_b4_14}}}

$I = \lbrace 1, 4 \rbrace$. The pair $(M', \mathfrak p_I^u \cap \mathfrak g^{-\theta})$ is equivalent to $(\Sp_4, \FF^4 \oplus \FF^4)$, hence $\rk \Gamma_I(G,H) = 6$.

Computations using LiE show that
\[
\left.R_G(\pi_1)\right|_H \simeq R_H(\pi_2) \oplus R_H(\pi_4) \quad \text{and} \quad \left.R_G(\pi_4)\right|_H \simeq R_H(\pi_1) \oplus R_H(\pi_4) \oplus R_H(0),
\]
which provides all indecomposable elements of $\Gamma_I(G,H)$ of the form $(\pi_1; *)$ and of the form $(\pi_4; *)$. Another computation yields
$\left.R_G(\pi_1 + \pi_4)\right|_H \simeq W_{1,4}(G,H) \oplus R_H(\pi_3)$, hence the last indecomposable element is $(\pi_1 + \pi_4; \pi_3)$.

\subsubsection{Case~\textup{\ref{f4_b4_2}}}

$I = \lbrace 2 \rbrace$. The pair $(M', \mathfrak p_I^u \cap \mathfrak g^{-\theta})$ is equivalent to $(\SL_2 \times \SL_2, \FF^2 \otimes (\FF^2 \oplus \FF^1))$, hence $\rk \Gamma_I(G,H) = 5$.

A computation using LiE shows that
\[
\left.R_G(\pi_2)\right|_H \simeq R_H(\pi_2) \oplus R_H(\pi_3) \oplus R_H(\pi_1 + \pi_3) \oplus R_H(\pi_1 + \pi_4) \oplus R_H(\pi_2 + \pi_4),
\]
which yields all indecomposable elements of $\Gamma_I(G,H)$ of the form $(\pi_2; *)$. This already gives the required number of indecomposable elements.

\subsubsection{Case~\textup{\ref{f4_b4_3}}}

$I = \lbrace 3 \rbrace$. The pair $(M', \mathfrak p_I^u \cap \mathfrak g^{-\theta})$ is equivalent to $(\SL_3, \FF^3 \oplus (\FF^3)^* \oplus \FF^1)$, hence $\rk \Gamma_I(G,H) = 5$.

A computation using LiE shows that
\[
\left.R_G(\pi_3)\right|_H \simeq R_H(\pi_1) \oplus R_H(\pi_2) \oplus R_H(\pi_3) \oplus R_H(\pi_4) \oplus R_H(\pi_1 + \pi_4),
\]
which yields all indecomposable elements of $\Gamma_I(G,H)$ of the form $(\pi_3; *)$. This already gives the required number of indecomposable elements.

\subsection{Case~\ref{e6_c4}}

$G = \mathsf E_6$, $H = \mathsf C_4$. The restriction matrix is
\[
\begin{pmatrix}
0 & 1 & 0 & 0\\
2 & 0 & 0 & 0\\
1 & 0 & 1 & 0\\
2 & 0 & 0 & 1\\
1 & 0 & 1 & 0\\
0 & 1 & 0 & 0\\
\end{pmatrix}
\]

\subsubsection{Case~\textup{\ref{e6_c4_1}}}

$I = \lbrace 1 \rbrace$.
The pair $(M', \mathfrak p_I^u \cap \mathfrak g^{-\theta})$ is equivalent to $(\SO_5, \FF^5)$, hence $\rk \Gamma_I(G,H) = 3$.

A computation using LiE shows that $\left.R_G(\pi_1)\right|_H \simeq R_H(\pi_2)$, which provides a unique indecomposable element of $\Gamma_I(G,H)$ of the form $(\pi_1; *)$. Another computation yields
$\left.R_G(2\pi_1)\right|_H \simeq W_{1,1}(G,H) \oplus R_H(\pi_4) \oplus R_H(0)$, which gives two more indecomposable elements of $\Gamma_I(G,H)$ of the form $(2\pi_1; *)$.

\subsection{Case~\ref{e6_a5xa1}}

$G = \mathsf E_6$, $H = \mathsf A_5 \cdot \mathsf A_1$. The restriction matrix is
\[
\begin{pmatrix}
1 & 0 & 0 & 0 & 0 & 1\\
0 & 0 & 0 & 0 & 0 & 2\\
0 & 1 & 0 & 0 & 0 & 2\\
0 & 0 & 1 & 0 & 0 & 3\\
0 & 0 & 0 & 1 & 0 & 2\\
0 & 0 & 0 & 0 & 1 & 1\\
\end{pmatrix},
\]
where the fundamental weights of $\mathsf A_5$ correspond to columns 1, \ldots, 5 and the fundamental weight of $\mathsf A_1$ corresponds to the last column.

\subsubsection{Case~\textup{\ref{e6_a5xa1_1}}}

$I = \lbrace 1 \rbrace$.
The pair $(M', \mathfrak p_I^u \cap \mathfrak g^{-\theta})$ is equivalent to $(\SL_5, \wedge^2 \FF^5)$, hence $\rk \Gamma_I(G,H) = 3$.

A computation using LiE shows that $\left.R_G(\pi_1)\right|_H \simeq R_H(\pi_4) \oplus R_H(\pi_1 + \pi'_1)$, which yields all indecomposable elements of $\Gamma_I(G,H)$ of the form $(\pi_1; *)$. Another computation yields $\left.R_G(2\pi_1)\right|_H \simeq W_{1,1}(G, H) \oplus R_H(\pi_2)$, hence $(2\pi_1; \pi_2)$ is the last indecomposable element of $\Gamma_I(G,H)$.

\subsection{Case~\ref{e6_f4}}

$G = \mathsf E_6$, $H = \mathsf F_4$. The restriction matrix is
\[
\begin{pmatrix}
0 & 0 & 0 & 1 \\
1 & 0 & 0 & 0 \\
0 & 0 & 1 & 0 \\
0 & 1 & 0 & 0 \\
0 & 0 & 1 & 0 \\
0 & 0 & 0 & 1 \\
\end{pmatrix}.
\]

\subsubsection{Case~\textup{\ref{e6_f4_12}}}

$I = \lbrace 1, 2\rbrace$.
The pair $(M', \mathfrak p_I^u \cap \mathfrak g^{-\theta})$ is equivalent to $(\Sp_4, \FF^4 \oplus \FF^1\oplus \FF^1)$, hence $\rk \Gamma_I(G,H) = 5$.

Computations using LiE show that
\begin{equation} \label{Res_E6_to_F4_1}
\left.R_G(\pi_1)\right|_H \simeq R_H(\pi_4) \oplus R_H(0)
\end{equation}
and
$\left.R_G(\pi_2)\right|_H \simeq R_H(\pi_1) \oplus R_H(\pi_4)$, which yields all indecomposable elements of $\Gamma_I(G,H)$ of the form $(\pi_1; *)$ and of the form $(\pi_2; *)$. Another computation  yields $\left.R_G(\pi_1 + \pi_2)\right|_H \simeq W_{1,2}(G,H) \oplus R_H(\pi_3)$, hence $(\pi_1 + \pi_2; \pi_3)$ is the last indecomposable element of $\Gamma_I(G,H)$.

\subsubsection{Case~\textup{\ref{e6_f4_13}}}

$I = \lbrace 1, 3 \rbrace$. The pair $(M', \mathfrak p_I^u \cap \mathfrak g^{-\theta})$ is equivalent to $(\SL_3, \FF^3 \oplus \FF^1 \oplus \FF^1)$, hence $\rk \Gamma_I(G,H) = 5$.

A computation using LiE shows that $\left.R_G(\pi_3)\right|_H \simeq R_H(\pi_1) \oplus R_H(\pi_3) \oplus R_H(\pi_4)$. From this and~(\ref{Res_E6_to_F4_1}) we find all indecomposable elements of $\Gamma_I(G,H)$ of the form $(\pi_1; *)$ and $(\pi_3; *)$, which already gives the required number of indecomposable elements.

\subsection{Case~\ref{e7_a7}}

$G = \mathsf E_7$, $H = \mathsf A_7$. The restriction matrix is
\[
\begin{pmatrix}
1 & 0 & 0 & 0 & 0 & 0 & 1\\
0 & 0 & 0 & 0 & 0 & 0 & 2\\
0 & 1 & 0 & 0 & 0 & 0 & 2\\
0 & 0 & 1 & 0 & 0 & 0 & 3\\
0 & 0 & 0 & 1 & 0 & 0 & 2\\
0 & 0 & 0 & 0 & 1 & 0 & 1\\
0 & 0 & 0 & 0 & 0 & 1 & 0
\end{pmatrix}.
\]

\subsubsection{Case~\textup{\ref{e7_a7_7}}}

$I = \lbrace 7 \rbrace$. The pair $(M', \mathfrak p_I^u \cap \mathfrak g^{-\theta})$ is equivalent to $(\SL_6, \wedge^2 \FF^6)$, hence $\rk \Gamma_I(G,H) = 4$.

A computation using LiE shows that $\left.R_G(\pi_7)\right|_H \simeq R_H(\pi_2) \oplus R_H(\pi_6)$, which provides all indecomposable elements of $\Gamma_I(G,H)$ of the form $(\pi_7; *)$. Another computation yields $\left.R_G(2\pi_7)\right|_H \simeq W_{7,7}(G,H) \oplus R_H(\pi_4) \oplus R_H(0)$, which gives two more indecomposable elements of $\Gamma_I(G,H)$ of the form $(2\pi_7; *)$.

\subsection{Case~\ref{e7_d6xa1}}
\label{subsec_e7_de6xa1}

$G = \mathsf E_7$, $H = \mathsf D_6 \cdot \mathsf A_1$. The restriction matrix is
\[
\begin{pmatrix}
0 & 0 & 0 & 0 & 0 & 0 & 2\\
0 & 0 & 0 & 0 & 0 & 1 & 2\\
0 & 0 & 0 & 0 & 1 & 0 & 3\\
0 & 0 & 0 & 1 & 0 & 0 & 4\\
0 & 0 & 1 & 0 & 0 & 0 & 3\\
0 & 1 & 0 & 0 & 0 & 0 & 2\\
1 & 0 & 0 & 0 & 0 & 0 & 1\\
\end{pmatrix},
\]
where the fundamental weights of $\mathsf D_6$ correspond to columns 1, \ldots, 6 and the fundamental weight of $\mathsf A_1$ corresponds to the last column.

\subsubsection{Case~\textup{\ref{e7_d6xa1_7}}}

$I = \lbrace 7 \rbrace$. The pair $(M', \mathfrak p_I^u \cap \mathfrak g^{-\theta})$ is equivalent to $(\Spin_{10}, \FF^{16})$ where $\Spin_{10}$ acts on $\FF^{16}$ via a half-spinor representation, hence $\rk \Gamma_I(G,H) =3$.

A computation using LiE shows that $\left.R_G(\pi_7)\right|_H \simeq R_H(\pi_6) \oplus R_H(\pi_1 + \pi'_1)$, which provides all indecomposable elements of $\Gamma_I(G,H)$ of the form $(\pi_7; *)$. Another computation yields $\left.R_G(2\pi_7)\right|_H \simeq W_{7,7}(G, H) \oplus R_H(\pi_2)$, hence $(2\pi_7; \pi_2)$ is the last indecomposable element of $\Gamma_I(G,H)$.

\section{The monoids \texorpdfstring{$\Gamma_I(G,H)$ in case~(\ref{case_C3})}{{Gamma\_I(G,H) in case (C3)}}}
\label{sect_proofs_C3}

In this section, we compute the monoid $\Gamma_I(G,H)$ for each of the cases in Table~\ref{table_sl}.

\subsection{Preliminary remarks}

Throughout this section, the numbers of cases refer to Table~\ref{table_sl} unless otherwise specified. All the notation and conventions for that table are used without extra explanation.

We choose $T_G$ (resp. $B_G$, $B_G^-$) to be the subgroup of diagonal (resp. upper triangular, lower triangular) matrices in~$G$.

For explicit computations, we regard $\Sp_{2m}$ as the group of $(2m\times 2m)$-matrices preserving the skew-symmetric bilinear form with matrix
\[
\begin{pmatrix}
0 & A \\
-A & 0
\end{pmatrix}
\]
where $A$ is the $(m \times m)$-matrix with ones on the antidiagonal and zeros elsewhere.

In Cases~\ref{sl_spsl}--\ref{sl_spspsp} we assume that $H$ is embedded in~$G$ in the block-diagonal form so that the $i$th factor of $H'$ is embedded as the $i$th block. With the above convention for~$\Sp_{2m}$, we may (and do) assume that
\begin{equation} \label{eqn_sl_subgroups}
T_H = T_G \cap H, \quad B_H = B_G \cap H, \quad \text{and} \quad B_H^- = B_G^- \cap H.
\end{equation}
For Case~\ref{sl_spin}, a concrete embedding of $H$ in~$G$ is described in \S\,\ref{subsec_sl_spin}, this embedding satisfies~(\ref{eqn_sl_subgroups}) as well.

In each of the cases, formulas~(\ref{eqn_sl_subgroups}) imply that $Q = P_I^- \cap H$ is a parabolic subgroup of~$H$; we denote by $M$ the Levi subgroup of $Q$ containing~$T_H$.

Thanks to Corollary~\ref{crl_rank_of_Gamma_refined}, in all the cases $\mathfrak g / (\mathfrak p_I^- + \mathfrak h)$ is a spherical $M$-module and $\rk \Gamma_I(G,H) = |I| + \rk_{M} (\mathfrak g / (\mathfrak p_I^- + \mathfrak h))$. In Cases~\ref{sl_spsl}--\ref{sl_spspsp}, the pair $(M, \mathfrak g/(\mathfrak p_I^- + \mathfrak h))$ is easily computed, and we omit the details.

For each of the cases, the rank of the spherical $M$-module $\mathfrak g / (\mathfrak p_I^- + \mathfrak h)$ is always computed as described in~\S\,\ref{subsec_spherical_modules}; the information on the ranks of indecomposable saturated spherical modules is taken from~\cite[\S\,5]{Kn}. All conclusions on indecomposable elements of $\Gamma_I(G,H)$ are obtained by applying Propositions~\ref{prop_indec_I} and~\ref{prop_indec_II}.

\subsection{Case~\ref{sl_spsl}}
\label{subsec_sl_spsl}

$G = \SL_{n}$, $H' = \Sp_{2p} \times \SL_q$, $2p+q=n$, $p \ge 2$, $q \ge 1$. For all the cases in this subsection, we consider the intermediate subgroup $F = C_H \cdot (\SL_{2p} \times \SL_q)$, so that $H \subset F \subset G$.

\subsubsection{Case~\textup{\ref{sl_spsl_12}}}
$I = \lbrace 1,2 \rbrace$. The pair $(M', \mathfrak g / (\mathfrak p_I^- + \mathfrak h))$ is equivalent to $(\SL_q, \FF^q \oplus \FF^q \oplus \FF^1)$, hence $\rk \Gamma_I(G,H) = 6 - \delta_q^1$.

Restricting the representations $R_G(\pi_1)$ and $R_G(\pi_2)$ to $H$ through the chain $G \supset F \supset H$ by using the information for Case~\ref{sl_Levi_pq_1i_part1} of~Table~\ref{table_Levi_sl} and for Case~\ref{sl_sp_1ipart1} of Table~\ref{table_sym} we find that
\begin{gather*}
\left. R_G(\pi_1) \right|_H \simeq R_H(\pi_1 + \chi_1) \oplus R_H(\pi'_1 + \chi_2), \\
\left. R_G(\pi_2) \right|_H \simeq R_H(\pi_2 + 2\chi_1) \oplus R_H(2\chi_1) \oplus R_H(\pi_1 + \pi'_1 + \chi_1 + \chi_2) \oplus R_H(\pi'_2 \oplus 2\chi_2)
\end{gather*}
(the last summand is present for $q \ge 3$). This yields all indecomposable elements of~$\Gamma_I(G,H)$ of the form $(\pi_1; *)$ and of the form $(\pi_2; *)$, which already gives the required number of indecomposable elements.

\subsubsection{Case~\textup{\ref{sl_spsl_1n-1}}}
$I = \lbrace 1, n-1 \rbrace$. The pair $(M', \mathfrak g / (\mathfrak p_I^- + \mathfrak h))$ is equivalent to $(\SL_{q-1} \times \Sp_{2p-2}$, $\FF^{q-1} \oplus \FF^{2p-2} \oplus \FF^1 \oplus \FF^1)$, hence $\rk \Gamma_I(G,H) = 6 - \delta_q^1$.

Restricting the representations $R_G(\pi_1)$, $R_G(\pi_{n-1})$, and $R_G(\pi_1 + \pi_{n-1})$ to $H$ through the chain $G \supset F \supset H$ by using the information for Cases~\ref{sl_Levi_pq_1i_part1},\,\ref{sl_Levi_pq_1i_part2} of~Table~\ref{table_Levi_sl} and for Case~\ref{sl_sp_1ipart2} of Table~\ref{table_sym} we find that
\begin{gather*}
\left. R_G(\pi_1) \right|_H \simeq R_H(\pi_1 + \chi_1) \oplus R_H(\pi'_1 + \chi_2), \\
\left. R_G(\pi_{n-1}) \right|_H \simeq R_H(\pi_1 + (2p-1)\chi_1 + q\chi_2) \oplus R_H(\pi'_{q-1} + 2p \chi_1 + (q-1)\chi_2),
\end{gather*}
which yields all indecomposable elements of $\Gamma_I(G,H)$ of the form $(\pi_1; *)$ and of the form $(\pi_{n-1}; *)$, and
\[
\left. R_G(\pi_1 + \pi_{n-1}) \right|_H \simeq W_{1,n-1}(G,H) \oplus R_H(\pi_2 + 2p \chi_1 + q\chi_2) \oplus R_H(2p\chi_1 + q\chi_2)
\]
(the last summand is present for $q \ge 2$), which yields two more indecomposable elements of $\Gamma_I(G,H)$ of the form $(\pi_1 + \pi_{n-1}; *)$.

\subsubsection{Case~\textup{\ref{sl_spsl_i_q=1}}}
$q = 1$, $I = \lbrace i \rbrace$, $3 \le i \le p$. The pair $(M', \mathfrak g / (\mathfrak p_I^- + \mathfrak h))$ is equivalent to $(\SL_i, \wedge^2 \FF^i \oplus \FF^i)$, hence $\rk \Gamma_I(G,H) = i + 1$.

Restricting the representation $R_G(\pi_i)$ to $H$ through the chain $G \supset F \supset H$ by using the information for Case~\ref{sl_Levi_p1} of Table~\ref{table_Levi_sl} and for Case~\ref{sl_sp_1ipart1} of Table~\ref{table_sym} we find that
\[
\left. R_G(\pi_i) \right|_H \simeq \bigoplus \limits_{0 \le k \le [\frac i2]} R_H(\pi_{i-2k} + i \chi_1) \oplus \bigoplus \limits_{0 \le k \le [\frac{i-1}2]} R_H(\pi_{i-1-2k} + (i-1)\chi_1 + \chi_2).
\]
This yields all indecomposable elements of~$\Gamma_I(G,H)$ of the form $(\pi_i; *)$, which already gives the required number of indecomposable elements.

\subsubsection{Case~\textup{\ref{sl_spsl_3}}}
$q \ge 2$, $I = \lbrace 3 \rbrace$. The pair $(M', \mathfrak g / (\mathfrak p_I^- + \mathfrak h))$ is equivalent to $(\SL_q \times \SL_3$, $(\FF^q \otimes \FF^3) \oplus (\FF^3)^*)$ for $p \ge 3$ and to $(\SL_q \times \SL_2, \FF^q \otimes (\FF^2 \oplus \FF^1))$ for $p = 2$,
hence $\rk \Gamma_I(G,H) = 7 - \delta_p^2 - \delta_q^2$.

Restricting the representations $R_G(\pi_3)$ and $R_G(2\pi_3)$ to $H$ through the chain $G \supset F \supset H$ by using the information for Case~\ref{sl_Levi_pq_1i_part1} of Table~\ref{table_Levi_sl} and for Case~\ref{sl_sp_123} of Table~\ref{table_sym} we find that
\begin{multline*}
\left. R_G(\pi_3) \right|_H \simeq R_H(\pi_3 + 3\chi_1) \oplus R_H(\pi_1 + 3\chi_1) \oplus R_H(\pi_2 + \pi'_1 + 2\chi_1 + \chi_2) \oplus \\
R_H(\pi'_1 + 2\chi_1 + \chi_2) \oplus R_H(\pi_1 + \pi'_2 + \chi_1 + 2\chi_2) \oplus R_H(\pi'_3 + 3\chi_2)
\end{multline*}
(the first summand is present for $p \ge 3$ and the last summand is present for $q \ge 3$), which yields all indecomposable elements of $\Gamma_I(G,H)$ of the form $(\pi_3; *)$, and
\[
\left. R_G(2\pi_3) \right|_H \simeq W_{3,3} (G,H) \oplus R_H(\pi_2 + \pi'_2 + 4\chi_1 + 2\chi_2),
\]
which shows that the last indecomposable element of~$\Gamma_I(G,H)$ is $(2\pi_3; \pi_2 + \pi'_2 + 4\chi_1 + 2\chi_2)$.

\subsubsection{Case~\textup{\ref{sl_spsl_i_p=2}}}
$p = 2$, $I = \lbrace i \rbrace$, $4 \le i \le q$. The pair $(M', \mathfrak g / (\mathfrak p_I^- + \mathfrak h))$ is equivalent to $(\Sp_4 \times \SL_{n-i}, \FF^4 \otimes \FF^{n-i})$, hence $\rk \Gamma_I(G,H) = 7$.

Restricting the representations $R_G(\pi_i)$ and $R_G(2\pi_i)$ to $H$ through the chain $G \supset F \supset H$ by using the information for Case~\ref{sl_Levi_pq_1i_part1} of Table~\ref{table_Levi_sl} and for Case~\ref{sl_sp_123} of Table~\ref{table_sym} we find that
\begin{multline*}
\left. R_G(\pi_i) \right|_H \simeq R_H(\pi'_i + i \chi_2) \oplus R_H(\pi_1 + \pi'_{i-1} + \chi_1 + (i-1)\chi_2) \oplus \\
R_H(\pi_2 + \pi'_{i-2} + 2\chi_1 + (i-2)\chi_2) \oplus
R_H(\pi'_{i-2} + 2\chi_1 + (i-2)\chi_2) \oplus \\
R_H(\pi_1 + \pi'_{i-3} + 3\chi_1 + (i-3)\chi_2) \oplus
R_H(\pi'_{i-4} + 4\chi_1 + (i-4) \chi_2),
\end{multline*}
which yields all indecomposable elements of $\Gamma_I(G,H)$ of the form $(\pi_3; *)$, and
\[
\left. R_G(2\pi_i) \right|_H \simeq W_{i,i}(G,H) \oplus R_H(\pi_2 + \pi'_{i-1} + \pi'_{i-3} + 4\chi_1 + (2i-4) \chi_2),
\]
whence $(2\pi_i; \pi_2 + \pi'_{i-1} + \pi'_{i-3} + 4\chi_1 + (2i-4) \chi_2)$ is the last indecomposable element of~$\Gamma_I(G,H)$.

\subsection{Case~\ref{sl_spsp}}
\label{subsec_sl_spsp}

$G = \SL_{n}$, $H' = \Sp_{2p} \times \Sp_{2q}$, $2p+2q=n$, $p \ge q \ge 2$. For all the cases in this subsection, we consider the intermediate subgroup $F = C_H \cdot (\Sp_{2p} \times \SL_{2q})$, so that $H \subset F \subset G$.

\subsubsection{Case~\textup{\ref{sl_spsp_12}}}
$I = \lbrace 1, 2 \rbrace$. The pair $(M', \mathfrak g / (\mathfrak p_I^- + \mathfrak h))$ is equivalent to $(\Sp_{2q}, \FF^{2q} \oplus \FF^{2q} \oplus \FF^1)$, hence $\rk \Gamma_I(G,H) = 7$.

Restricting the representations $R_G(\pi_1)$ and $R_G(\pi_2)$ to $H$ through the chain $G \supset F \supset H$ by using the information for Case~\ref{sl_spsl_12} and for Case~\ref{sl_sp_1ipart1} (or~\ref{sl_sp_123}) of Table~\ref{table_sym} we find that
\begin{gather*}
\left. R_G(\pi_1) \right|_H \simeq R_H(\pi_1 + \chi_1) \oplus R_H(\pi'_1 + \chi_2), \\
\left. R_G(\pi_2) \right|_H \simeq R_H(\pi_2 {+} 2\chi_1) \oplus R_H(2\chi_1) \oplus R_H(\pi_1 {+} \pi'_1 {+} \chi_1 {+} \chi_2) \oplus R_H(\pi'_2 {+} 2\chi_2) \oplus R_H(2\chi_2).
\end{gather*}
This yields all indecomposable elements of~$\Gamma_I(G,H)$ of the form $(\pi_1; *)$ and of the form $(\pi_2; *)$, which already gives the required number of indecomposable elements.

\subsubsection{Case~\textup{\ref{sl_spsp_1n-1}}}
$I = \lbrace 1, n-1 \rbrace$. The pair $(M', \mathfrak g / (\mathfrak p_I^- + \mathfrak h))$ is equivalent to $(\Sp_{2p-2} \times \Sp_{2q-2}$, $\FF^{2p-2} \oplus \FF^{2q-2} \oplus \FF^1 \oplus \FF^1 \oplus \FF^1)$, hence $\rk \Gamma_I(G,H) = 7$.

Restricting the representations $R_G(\pi_1)$, $R_G(\pi_{n-1})$, and $R_G(\pi_1 + \pi_{n-1})$ to $H$ through the chain $G \supset F \supset H$ by using the information for Case~\ref{sl_spsl_1n-1} of~Table~\ref{table_Levi_sl} and for Case~\ref{sl_sp_1ipart2} of Table~\ref{table_sym} we find that
\begin{gather*}
\left. R_G(\pi_1) \right|_H \simeq R_H(\pi_1 + \chi_1) \oplus R_H(\pi'_1 + \chi_2), \\
\left. R_G(\pi_{n-1}) \right|_H \simeq R_H(\pi_1 + (2p-1)\chi_1 + 2q\chi_2) \oplus R_H(\pi'_1 + 2p \chi_1 + (2q-1)\chi_2),
\end{gather*}
which yields all indecomposable elements of $\Gamma_I(G,H)$ of the form $(\pi_1; *)$ and of the form $(\pi_{n-1}; *)$, and
\begin{multline*}
\left. R_G(\pi_1 + \pi_{n-1}) \right|_H \simeq W_{1,n-1}(G,H) \oplus \\
R_H(\pi_2 + 2p \chi_1 + 2q\chi_2) \oplus
R_H(\pi'_2 + 2p\chi_1 + 2q\chi_2) \oplus
R_H(2p\chi_1 + 2q\chi_2),
\end{multline*}
which yields all indecomposable elements of $\Gamma_I(G,H)$ of the form $(\pi_1 + \pi_{n-1}; *)$.

\subsection{Case~\ref{sl_spslsl}}
\label{subsec_sl_spslsl}

$G = \SL_{n}$, $H' = \Sp_{2p} \times \SL_q \times \SL_r$, $2p+q+r=n$, $p \ge 2$, $q \ge r \ge 1$. We consider the intermediate subgroup $F = C_H \cdot (\SL_{2p} \times \SL_q \times \SL_r)$, so that $H \subset F \subset G$.

\subsubsection{Case~\textup{\ref{sl_spslsl_2}}}
$I = \lbrace 2 \rbrace$. The pair $(M', \mathfrak g / (\mathfrak p_I^- + \mathfrak h))$ is equivalent to $(\SL_2 \times \SL_q \times \SL_r$, $\FF^2 \otimes (\FF^q \oplus \FF^r) \oplus \FF^1)$, hence $\rk \Gamma_I(G,H) = 7 - \delta_q^1 - \delta_r^1$.

Restricting the representation $R_G(\pi_2)$ to $H$ through the chain $G \supset F \supset H$ by using the information for Cases~\ref{sl_Levi_pq1_part1},\,\ref{sl_Levi_pqr} of Table~\ref{table_Levi_sl} and for Case~\ref{sl_sp_1ipart1} (or~\ref{sl_sp_123}) of Table~\ref{table_sym} we find that
\begin{multline*}
\left. R_G(\pi_2) \right|_H \simeq R_H(\pi_2 + 2\chi_1) \oplus
R_H(2\chi_1) \oplus R_H(\pi'_2 + 2\chi_2) \oplus
R_H(\pi''_2 + 2\chi_3) \oplus \\
R_H(\pi_1 + \pi'_1 + \chi_1 + \chi_2) \oplus
R_H(\pi_1 + \pi''_1 + \chi_1 + \chi_3) \oplus
R_H(\pi'_1 + \pi''_1 + \chi_2 + \chi_3)
\end{multline*}
where the summand $R_H(\pi'_2 + 2\chi_2)$ is present for $q \ge 2$ and the summand $R_H(\pi''_2 + 2\chi_3)$ is present for $r \ge 2$. This yields all indecomposable elements of~$\Gamma_I(G,H)$ of the form $(\pi_2; *)$, which already gives the required number of indecomposable elements.

\subsection{Case~\ref{sl_spspsl}}

$G = \SL_{n}$, $H' = \Sp_{2p} \times \Sp_{2q} \times \SL_r$, $2p+2q+r=n$, $p \ge q \ge 2$, $r \ge 1$. We consider the intermediate subgroup $F = C_H \cdot (\Sp_{2p} \times \SL_{2q} \times \SL_r)$, so that $H \subset F \subset G$.

\subsubsection{Case~\textup{\ref{sl_spspsl_2}}}
$I = \lbrace 2 \rbrace$. The pair $(M', \mathfrak g / (\mathfrak p_I^- + \mathfrak h))$ is equivalent to $(\SL_2 \times \Sp_{2q} \times \SL_r$, $\FF^2 \otimes (\FF^{2q} \oplus \FF^r) \oplus \FF^1)$, hence $\rk \Gamma_I(G,H) = 8 - \delta_r^1$.

Restricting the representation $R_G(\pi_2)$ to $H$ through the chain $G \supset F \supset H$ by using the information for Case~\ref{sl_spslsl_2} and for Case~\ref{sl_sp_1ipart1} (or~\ref{sl_sp_123}) of Table~\ref{table_sym} we find that
\begin{multline*}
\left. R_G(\pi_2) \right|_H \simeq R_H(\pi_2 + 2\chi_1) \oplus
R_H(2\chi_1) \oplus R_H(\pi'_2 + 2\chi_2) \oplus R_H(2\chi_2) \oplus
R_H(\pi''_2 + 2\chi_3) \oplus \\
R_H(\pi_1 + \pi'_1 + \chi_1 + \chi_2) \oplus
R_H(\pi_1 + \pi''_1 + \chi_1 + \chi_3) \oplus
R_H(\pi'_1 + \pi''_1 + \chi_2 + \chi_3)
\end{multline*}
where the summand $R_H(\pi''_2 + 2\chi_3)$ is present for $r \ge 2$. This yields all indecomposable elements of~$\Gamma_I(G,H)$ of the form $(\pi_2; *)$, which already gives the required number of indecomposable elements.

\subsection{Case~\ref{sl_spspsp}}

$G = \SL_{n}$, $H' = \Sp_{2p} \times \Sp_{2q} \times \Sp_{2r}$, $2p+2q+2r=n$, $p \ge q \ge r \ge 2$. We consider the intermediate subgroup $F = C_H \cdot (\Sp_{2p} \times \Sp_{2q} \times \SL_{2r})$, so that $H \subset F \subset G$.

\subsubsection{Case~\textup{\ref{sl_spspsp_2}}}
$I = \lbrace 2 \rbrace$. The pair $(M', \mathfrak g / (\mathfrak p_I^- + \mathfrak h))$ is equivalent to $(\SL_2 \times \Sp_{2q} \times \Sp_{2r}$, $\FF^2 \otimes (\FF^{2q} \oplus \FF^{2r}) \oplus \FF^1)$, hence $\rk \Gamma_I(G,H) = 9$.

Restricting the representation $R_G(\pi_2)$ to $H$ through the chain $G \supset F \supset H$ by using the information for Case~\ref{sl_spspsl_2} and for Case~\ref{sl_sp_1ipart1} (or~\ref{sl_sp_123}) of Table~\ref{table_sym} we find that
\begin{multline*}
\left. R_G(\pi_2) \right|_H \simeq R_H(\pi_2 + 2\chi_1) \oplus
R_H(2\chi_1) \oplus R_H(\pi'_2 + 2\chi_2) \oplus
R_H(2\chi_2) \oplus R_H(\pi''_2 + 2\chi_3) \oplus \\
R_H(2\chi_3) \oplus
R_H(\pi_1 {+} \pi'_1 {+} \chi_1 {+} \chi_2) \oplus
R_H(\pi_1 {+} \pi''_1 {+} \chi_1 {+} \chi_3) \oplus
R_H(\pi'_1 {+} \pi''_1 {+} \chi_2 {+} \chi_3).
\end{multline*}
This yields all indecomposable elements of~$\Gamma_I(G,H)$ of the form $(\pi_2; *)$, which already gives the required number of indecomposable elements.

\subsection{Case~\ref{sl_spin}}
\label{subsec_sl_spin}

$G = \SL_8$, $H = \Spin_7$. To describe the embedding of $H$ into~$G$, we first consider the group $K = \SO_8$ preserving the symmetric bilinear form on $\FF^8$ whose matrix consists of ones on the antidiagonal and zeros elsewhere. With this realization of~$K$, we may (and do) assume that $T_K = T_G \cap K$, $B_K = B_G \cap K$, and $B_K^- = B_G^- \cap K$. Choose the simple roots $\beta_1, \beta_2, \beta_3, \beta_4 \in \mathfrak X(T_K)$ in such a way that for any $t = \operatorname{diag}(t_1,t_2,t_3,t_4, t_4^{-1},t_3^{-1},t_2^{-1},t_1^{-1}) \in T_K$ one has $\beta_1(t) = t_1t_2^{-1}$, $\beta_2(t) = t_2t_3^{-1}$, $\beta_3(t) = t_3t_4^{-1}$, and $\beta_4(t) = t_3t_4$. Second, let $H_1 = \SO_7$ be the stabilizer in $K$ of the vector $e_4 - e_5$, where $e_i$ denotes the $i$th vector of the standard basis of~$\FF^8$. Now let $\mathfrak h \subset \mathfrak k$ be the image of $\mathfrak h_1$ under an outer automorphism of $\mathfrak k$ preserving $\mathfrak t_K$ and~$\mathfrak b_K$ and interchanging the simple roots $\beta_1$ and~$\beta_4$. Finally, we take $H$ to be the connected subgroup of~$G$ with Lie algebra~$\mathfrak h$.

\subsubsection{Case~\textup{\ref{sl_spin_2}}}

$I = \lbrace 2 \rbrace$. Direct computations using the above-described embedding of~$\mathfrak h$ into~$\mathfrak g$ show that the pair $(M', \mathfrak g / (\mathfrak p_I^- + \mathfrak h))$ is equivalent to $(\SL_2 \times \SL_2, \operatorname{S}^2 \FF^2 \oplus \FF^2)$, hence $\rk \Gamma_I(G,H) = 4$.

Applying Case~\ref{sl_so_even_i} of Table~\ref{table_sym}, we obtain \[
\left. R_G(\pi_2) \right|_K \simeq R_K(\pi_2) \ \text{ and } \ \left. R_G(2\pi_2) \right|_K \simeq R_K(2\pi_2) \oplus R_K(2\pi_1) \oplus R_K(0).
\]
Then applying Case~\ref{spin_even_q=1} of Table~\ref{table_sym} we obtain
\[
\left. R_G(\pi_2) \right|_H \simeq R_H(\pi_2) \oplus R_H(\pi_1),
\]
which yields all indecomposable elements of $\Gamma_I(G,H)$ of the form $(\pi_2; *)$. As $\mathfrak h$ is obtained from $\mathfrak h_1$ by the above-mentioned outer automorphism of~$\mathfrak k$, to compute $\left. R_G(2\pi_2) \right|_H$ we apply the modified version of Case~\ref{spin_even_q=1} of Table~\ref{table_sym} in which the fundamental weights~$\pi_1$ and~$\pi_4$ of~$K$ are interchanged. This yields
\[
\left. R_G(2\pi_2) \right|_H \simeq W_{2,2}(G,H) \oplus R_H(2\pi_3) \oplus R_H(0),
\]
whence all indecomposable elements of $\Gamma_I(G,H)$ of the form $(2\pi_2; *)$.


\end{document}